%% file: SI_JEMS.tex
\documentclass[12pt,english]{amsart}
\usepackage{amssymb,amsmath,amscd,graphicx,fontenc,amsthm,mathrsfs}
\usepackage[pdftex,bookmarks,colorlinks,breaklinks]{hyperref} 
\hypersetup{linkcolor=blue,citecolor=red}
\usepackage{multicol}
\usepackage{enumitem}
\usepackage[numbers]{natbib}
\usepackage{tikz}
\usetikzlibrary{matrix,cd}
\renewcommand{\H}{\mathscr{H}}

\newtheorem{theorem}{Theorem}
\newtheorem{lem}[theorem]{Lemma}
\newtheorem{proposition}[theorem]{Proposition}
\newtheorem{claimm}[theorem]{Claim}
\newtheorem{corollary}[theorem]{Corollary}

\newtheorem*{sublemma}{Sublemma}
\newtheorem*{claim}{Claim}

 \newtheorem{thmA}{Theorem}

\theoremstyle{definition}

\newtheorem{definition}{Definition}

\newcommand{\ef}{\mathrm{e}}

\newcommand{\ZZ}{\mathbb{Z}}   
\newcommand{\CC}{\mathbb{C}}  
\newcommand{\RR}{\mathbb{R}} 

 \renewcommand{\d}[1]{\ \mathrm{d}#1}
\def\ncal{\mathcal{N}}

\newcommand{\G}{\mathbb{G}}  

\newcommand{\B}{\mathscr{B}}   
 
\newcommand{\T}{\mathscr{T}} 

\newcommand{\cpl}{\mathscr{C}} 

\input{micros.tex}

\newcommand{\Dim}{\mathrm{DC}}
 \newcommand{\HS}{\textrm{HS}}
\newcommand{\cf}{\bb1}
\newcommand\dGi{d_{0i}}
\newcommand\diam{{\rm diam}}
\newcommand\crho{\rho}
\newcommand\rhonu{\nu^\crho}
\newcommand\rhonup{\nu}
\newcommand{\SO}{{\rm SO}}

  \newcommand\done{d_1}
 \newcommand\dtwo{d_2}
 \newcommand\cone{C_1}
 \newcommand\ctwo{C_2}
 
 \newcommand{\X}{\Bigl(2 C_{01}C_{02}C_{11}C_{12}   \Bigr)^{ O_{_{d_{01},d_{02}, L }} ( 1)}}

\usepackage[top=2cm,bottom=2cm,right=2cm,left=2cm]{geometry}
\title{Spectral independence}
\author[Alireza S Golsefidy]{Alireza S Golsefidy}
\address{Mathematics Department, University of California, San Diego, CA 92093-0112, USA}
\email{golsefidy@ucsd.edu}
\thanks{A.S.G.\ acknowledges support by the NSF grants 1602137, 1902090, 2302519.}
\author[Keivan Mallahi-Karai]{Keivan Mallahi-Karai}
\address{ Constructor University, Campus Ring I, 28759, Bremen, Germany}
\email{kmallahikarai@constructor.university}
\thanks{K.MK.\ acknowledges support by the DFG grant DFG MA 6503/1-1.}
\author[Amir Mohammadi]{Amir Mohammadi}
\address{Mathematics Department, University of California, San Diego, CA 92093-0112, USA }
\email{ammohammadi@ucsd.edu}
\thanks{A.M.\ acknowledges support by the NSF, grants 1764246 and 2055122.}

\begin{document}
\begin{abstract}
We prove the spectral gap property for random walks on the product of two non-locally isomorphic analytic real or $p$-adic compact groups with simple Lie algebras, under the necessary condition that the marginals posses a spectral gap. 
Furthermore, we give additional control on the spectral gap depending on certain specific properties of the given groups and marginals; in particular, we prove some new cases of the super-approximation conjecture. 

One ingredient of the proof is a local Ulam stability result which is introduced and proved in this paper. This result characterizes partially defined almost homomorphisms between two analytic compact groups with simple Lie algebras. 
\end{abstract}
\maketitle

\setcounter{tocdepth}{1}
\tableofcontents

\section{Introduction}\label{intro}

Let $G$ be a compact group, and let $\mu$ be a Borel symmetric probability measure on $G$. An $\ell$-step random walk with respect to $\mu$ is 
\[
X^{(\ell)}:=X_1\cdots X_\ell
\]
where  $X_1, X_2, \dots$ is a sequence of independent random variables with probability law $\mu$. If the group generated by the support of $\mu$ is dense in $G$, then for every continuous function $f$, 
\[
\lim_{\ell\to \infty} \bbe[f(X^{(\ell)})]=\int_G f(x)\h dm_G(x),
\]
where $m_G$ is the probability Haar measure of $G$. The rate of convergence is governed by the operator norm $\lambda(\mu)$ of the convolution operator 
\[
T_{\mu}:L_0^2(G,m_G)\to L_0^2(G,m_G),\qquad (T_{\mu}(f))(x):=\int_G f(g^{-1}x) d\mu(g),
\]
where $L^2_0(G,m_G)$ is the orthogonal complement of constant functions. We say $\mu$ has the \emph{spectral gap property} if $\lambda(\mu)<1$. Likewise, a symmetric random variable $X$ with values in $G$ is said to have the spectral gap property if its probability law has spectral gap.   

In this work, we will investigate the spectral gap property of a pair  $(X,Y)$ of random variables on $G_1\times G_2$ where $G_i$'s are compact groups. It is clear that if $(X,Y)$ has the spectral gap property, then both $X$ and $Y$ must have this property. Motivated by this, we will say $G_1$ and $G_2$ are {\em spectrally independent} if this necessary condition is also sufficient. The main result of this paper shows that $G_1$ and $G_2$ are spectrally independent for a wide class of compact groups. 

\begin{thmA}\label{thm1:intro}
For $i=1,2$, let $F_i$ be the field of real $\mathbb R$ or $p$-adic $\bbq_p$ numbers. 
Let $\mathbb G_i$ be an $F_i$-almost simple group, and let $G_i$ be an open compact subgroup of $\mathbb G_i(F_i)$. Assume that $G_1$ and $G_2$ are not locally isomorphic. Then $G_1$ and $G_2$ are spectrally independent.
\end{thmA}

Note that if $G_1$ and $G_2$ are spectrally independent, then they should be {\em algebraically independent}, that means that they have no common nontrivial topological quotients.
Indeed, if $\phi_i:G_i\to H$ is a common nontrivial quotient, then $K=\{(g_1,g_2)\in G_1\times G_2: \phi_1(g_1)=\phi_2(g_2)\}$ carries a measure with marginals $m_{G_1}$ and $m_{G_2}$, however, $K$ is a proper closed subgroup of $G_1\times G_2$.    
It is also worth noting that in Appendix~\ref{app:abelian}, we provide examples of algebraically independent groups which are not spectrally independent.  

\medskip
One of the key ingredients in the proof of Theorem \ref{thm1:intro} is Theorem~\ref{thm:approximate-hom}, which is a local stability theorem developed in this paper and is of independent interest.
Roughly speaking, this stability result states that a {\em partial approximate homomorphism}, see Definition~\ref{def: approxmiate local hom}, between $G_1$ and $G_2$ as in Theorem \ref{thm1:intro} is close to an isogeny between $\bbg_1$ and $\bbg_2$. Stability theorems in this vein have a long history, 
indeed Grove, Karcher, Ruh~\cite[Theorem 4.3]{GKR-AlmostHom} studied approximate homomorphisms between compact Lie groups; later Kazhdan~\cite{Kazhdan-approx-hom} studied approximate homomorphisms from amenable groups to unitary operators on a Hilbert space. One of the main novelties of our result is that we only require the approximate homomorphism to be defined on a neighborhood of the identity rather than the entire group; this generalization requires new techniques as the averaging techniques, used in aforementioned results, can no longer be applied.
We refer the reader to Theorem~\ref{thm:approximate-hom} for the precise statement, here we only state two special cases of that theorem. 

In the following, $1_\rho$ denotes the ball of radius $\rho$ in $G$ with respect to the operator norm. 
Also, for any prime $p$ and any positive integer $k$, we let $\SL_n(\ZZ_p)[k]$ denote the kernel of the reduction mod $p^k$ from $\SL_n(\ZZ_p)$ to $\SL_n(\ZZ/p^k\ZZ)$. 

\begin{thmA}\label{thm2:intro}
    \begin{enumerate}
\item {\rm (Real Case)} Let $G_i=\SU(n_i)$ for $n_i\geq 2$. Then there exist $m_0>1$ and $0<c<1$, depending only on $n_1$ and $n_2$, such that for every $m>m_0m'$, the following holds. For $\rho<1/2$, suppose $f:1^{(1)}_\rho\to G_2$ is a map which satisfies
\[
f(g^{-1})f(g)\in 1^{(2)}_{\rho^m} \qquad\text{and}\qquad f(g_1g_2)(f(g_1)f(g_2))^{-1}\in 1^{(2)}_{\rho^m}
\]
whenever these expressions are defined. Assume further that the $\rho^m$ neighborhood of ${\rm Im}(f)$ contains the $\rho^{m'}$ neighborhood of the identity in $G_2$. Then $n_1=n_2=n$ and there is an isomorphism $\Psi:\SU(n)\to \SU(n)$ so that 
\[
f(g)\Psi(g)^{-1}\in 1_{\rho^{cm}} \qquad\text{for all $g\in 1^{(1)}_{\rho^m}$}.
\]

\item {\rm ($p$-adic Case)} Let $n_1, n_2\geq 2$ be two integers. There exist $m_0>1$ and $0<c<1$, depending only on $n_1$ and $n_2$, so that for every $m>m_0m'$ the following holds. Let $k\in\bbz^+$. Suppose  
\[
f:\SL_{n_1}(\ZZ_{p})[k]/\SL_{n_1}(\bbz_p)[km]\to \SL_{n_2}(\ZZ_{p})/\SL_{n_2}(\bbz_p)[km]
\]
is a group homomorphism such that 
\[
\SL_{n_2}(\ZZ_{p})[km']/\SL_{n_2}(\bbz_p)[km]\subseteq {\rm Im}(f).
\]
Then $n_1=n_2=n$ and there is an isomorphism $\Psi:\SL_n(\bbq_p)\to \SL_n(\bbq_p)$ so that 
\[
f\bigl(g\SL_{n}(\bbz_p)[km]\bigr)\equiv \Psi(g) \pmod{p^{ckm}}\qquad\text{for all $g\in\SL_{n}(\ZZ_{p})[k]$}.
\]
    \end{enumerate}
\end{thmA}

In the next section, we will present more precise statements of results of this work, including Theorems~\ref{thm1:intro} and~\ref{thm2:intro}. We will also state a {\em global} version, Theorem~\ref{thm:Q-form}, of Theorem~\ref{thm1:intro} where we obtain  spectral gap which is uniform across all places.

\section{Statement of results}

Recall that for a symmetric Borel probability measure $\mu$ on a compact group $G$, the contraction factor $\lambda(\mu)$ of $\mu$ is the operator norm of the convolution operator 
\[
T_{\mu}:L_0^2(G,m_G)\to L_0^2(G,m_G),\qquad (T_{\mu}(f))(x):=\int_G f(g^{-1}x) d\mu(g),
\]
where $L^2_0(G,m_G)$ is the orthogonal complement of constant functions. Given a symmetric random variable $X$ with values in $G$ and probability law $\mu$, we define
\[
\mathcal L(X;G):=-\log(\lambda(\mu)).
\]
When $G$ is clear from the context, we often denote $\lcal(X;G)$ by $\lcal(X)$. 

\begin{theorem}\label{thm:main1}
For $i=1,2$, let $F_i$ be the field of real $\bbr$ or $p$-adic $\bbq_p$ numbers. 
Let $\mathbb G_i$ be an $F_i$-almost simple group and let $G_i$ be an open compact subgroup of $\mathbb G_i(F_i)$. Suppose $G_1$ and $G_2$ are not locally isomorphic and
$X=(X_1, X_2)$ is a $G_1 \times G_2$-valued
random variable with a Borel probability law. Then the marginal bounds $\min  \left(  \mathcal{L}(X_1), \mathcal { L}(X_2)  \right) \ge c_0 >0$ imply
\[
\mathcal{L}(X) \gg_{c_0, G_1, G_2} 1.
\]
In particular, $G_1$ and $G_2$ are spectrally independent.
\end{theorem}

Recall that the {\em super-approximation conjecture} states that for a finite symmetric subset $\Omega$ of $\GL_n(\bbz[1/q_0])$, the connected component of the Zariski closure of the group $\Gamma=\langle\Omega\rangle$ is perfect if and only if the family of Cayley graphs $\{{\rm Cay}(\pi_m(\Gamma),\pi_m(\Omega)\}$ forms a family of \emph{expanders} as $m$ ranges over the set of positive integers with $\gcd(q_0,m)=1$.
Our next theorem proves a special case of the super-approximation conjecture. 

Expander graphs are roughly highly connected sparse graphs that have arbitrarily large number of vertices. The first explicit construction of expander graphs was done by Margulis~\cite{Margulis73} using Kazhdan's property (T). Later, the most optimal expanders, known as Ramanujan graphs, were constructed by Lubotzky, Phillips, and Sarnak~\cite{LPS}, and independently by Margulis~\cite{Margulis88}. The reader can see more on expanders and their connections with the spectral gap property in \cite{Lubotzky-book}. 

In the past two decades, there has been a lot of progress on the super-approximation conjecture starting with the seminal work of Bourgain and Gamburd~\cite{Bourgain-Gamburd-08}. 
We refer the reader to the following articles to see the more recent results on this conjecture~\cite{Bourgain-Gamburd-Sarnak,Varju12,Bourgain-Varju,Golsefidy-Varju12,Golsefidy-IMRN-SA,SG-super-approx-II, Saxce-He}. 

It is worth pointing out that in \cite{Saxce-He}, the super-approximation conjecture is proved if $\Omega$ consists of \emph{integral} matrices and the Zariski closure of $\Gamma$ is $\bbq$-almost simple. By refining the tools developed for the proof of Theorem~\ref{thm:main1}, we provide  
an affirmative answer to the super-approximation conjecture with no integrality assumption in the case where the Zariski closure is absolutely almost simple and moduli have at most two distinct prime factors. 
 
\begin{theorem}\label{thm:Q-form}
Let $ \bbg$ be an absolutely almost simple $\bbq$-algebraic group. Let $ \Omega \subseteq \bbg( \bbq)$ be a 
finite symmetric subset such that the group  $\Gamma=\langle \Omega\rangle$ is Zariski dense in $ \bbg$.
Denote by $V_{ \Gamma }$  the set of all places $\nu$ of $ \bbq$ such that $\Gamma$ is a bounded subset of $ \bbg( \bbq_\nu)$.
For distinct $\nu_1, \nu_2 \in V_{ \Gamma}$, let $\Gamma_{\nu_1, \nu_2}$ denote the closure of $\Gamma$ in 
$ \bbg( \bbq_{\nu_1} ) \times   \bbg( \bbq_{\nu_2} )$. Then 
\[ \inf_{ \nu_1 \neq \nu_2  \in V_{ \Gamma } }  \mathcal{L} ( X; \Gamma_{ \nu_1, \nu_2} ) >0,  \]
 where $X$ is a random variable with the uniform distribution on $ \Omega$. 
\end{theorem}

As it was explained in the introduction, one of the main ingredients in the proof of Theorem~\ref{thm:main1} is a classification of \emph{partial almost homomorphisms} between two compact analytic groups. To formulate our result, we need a few definitions and notation.

For a metric compact group $G$ and $g\in G$, we let $g_\rho$ denote the $\rho$-neighborhood of $g$. 

\begin{definition}\label{def: approxmiate local hom}
Suppose $G_1$ and $G_2$ are two compact groups equipped with bi-invariant metrics compatible with their topology. Let $\delta>0$ and $S$ be a symmetric subset of $G_1$ containing the identity element $1^{(1)}$. 
A function $f:S\rightarrow G_2$ is called an 
{\em $S$-partial, $\delta$-approximate homomorphism} if the following three properties are satisfied 
\begin{enumerate}
\item $f(1^{(1)})=1^{(2)}$, 
\item $f(g^{-1})\in (f(g)^{-1})_{\delta}$  for every $g\in S$, and 
\item for all $g_1,g_2\in S$ with $g_1g_2\in S$, we have $f(g_1g_2)\in \bigl(f(g_1)f(g_2)\bigr)_{\delta}$.
\end{enumerate}
We will refer to a $1^{(1)}_{\rho}$-partial $\delta$-approximate homomorphism simply as {\em $\rho$-partial, $\delta$-approximate homomorphism}.
\end{definition}

Let $\bbg\subseteq(\SL_{N})_\bbq$ be a Zariski connected, absolutely almost simple $\bbq$-algebraic subgroup. Let $\Sigma_\bbg$ be the set of places $v$ of $\bbq$ where $\bbg(\bbq_\nu)$ contains a compact open subgroup; note that $\Sigma_\bbg$ equals the set of all places if $\bbg(\bbr)$ is compact, and equals the set of all finite places otherwise. We say a family 
\[
\{G_\nu:\nu\in\Sigma_\bbg\}
\]
of compact groups is a \emph{coherent family} attached to $\bbg$ if the following holds: $G_\nu\subseteq \bbg(\bbq_\nu)$ is a compact open  and $G_\nu=\bbg(\bbq_\nu)\cap \SL_N(\bbz_\nu)$ for all but finitely many places $\nu$. 

In the sequel, we let $p_\nu=\nu$ if $\nu$ is non-Archimedean and $p_\nu=2$ if $\nu=\infty$.

\begin{theorem}\label{thm:thm:approximate-hom-intro}\label{thm:thm:approximate-hom}\label{thm:approximate-hom}
Let $\nu_1$ and $\nu_2$ be two (possibly equal) places of $\bbq$, and let $F_i:=\bbq_{\nu_i}$. 
Let $\bbg_i\subseteq (\SL_{n_i})_{F_i}$ be a Zariski connected, $F_i$-almost simple subgroup, and let $G_i\subseteq \bbg_i(F_i)$ be a compact open subgroup. 

If $F_i=\bbr$, we assume that $\bbg_i$ is given by an $\bbr$-embedding in $(\SL_{n_i})_\bbr$ 
such that $\bbg_i(\bbr)\subseteq {\rm SO}_{n_i}(\bbr)$. 
If $F_i=\bbq_p$, then we assume that $G_i\subseteq \SL_{n_i}(\bbz_p)$. 
In both cases, we consider the metric induced by the operator norm. 

Then there is a positive number $c=c(\dim \bbg_1,\dim \bbg_2)$ such that for every 
$m\gg_{G_1,G_2} m'$ and positive number $\rho\ll_{G_1,G_2} 1$ the following holds: 

If $f:1_{\rho}^{(1)}\rightarrow  G_2$ is a $\rho$-partial, $\rho^{m}$-approximate homomorphism which satisfies 
\[
1^{(2)}_{\rho^{m'}}\subseteq ({\rm Im}(f))_{\rho^m}\qquad\qquad\text{\emph{(Large image)}},
\]
then $F_1=F_2=F$, $\Lie(\bbg_1)(F)\simeq \Lie(\bbg_2)(F)$, and $f$ is \emph{near an isogeny} in the following sense: There is an $F$-central isogeny $\Psi:\wt{\bbg}_1\rightarrow \bbg_2$ where $\wt{\bbg}_1$ is the simply-connected cover of $\bbg_1$ so that  
\[
\text{$f(g)\in \Psi(\wt{g})_{\rho^{cm}}\quad$ for every $g\in 1^{(1)}_{\rho}$,}
\]
where $\wt{g}$ is the unique lift of $g$ under the covering map from 
$\wt{\bbg}_1$ to $\bbg_1$ which belongs to the image of $4\rho$-neighborhood of $0$ under the exponential map.   

Moreover, if $\fcal=\{G_\nu\}$ is a coherent family attached to a $\bbq$-group $\bbg$, then there exist $\kappa:=\kappa_\bbg$ and $m_\fcal,\rho_\fcal>0$ so that the following holds.   
For all distinct $\nu_1,\nu_2$ and $\rho\leq\rho_\fcal\cdot\min\{p_{\nu_1}^{-\kappa},p_{\nu_2}^{-\kappa}\}$, there is no $\rho$-partial, $\rho^{m_\fcal}$-approximate homomorphism $f:1_{\rho}^{(1)}\rightarrow G_{\nu_2}$
which satisfies $1^{(2)}_{\rho}\subseteq ({\rm Im}(f))_{\rho^{m_\fcal}}$.
\end{theorem}

\subsection{Outline of the arguments} 
We start with outlining the proof of Theorem~\ref{thm:main1}. The proof will be carried out in several steps and relies heavily on the results of~\cite{MMSG}.

{\bf Step 1.}\ The first step of the proof is to reduce proof of the spectral gap property of the random walk by $\mu$, the probability law of $X$, to the study of functions which live at small scales. This is done using the notion of locally random group and the Littlewood-Paley theory for these groups, which was developed in~\cite{MMSG}. In particular, we showed in \cite[Theorem 2.10, Theorem 9.3]{MMSG} that if for every $\eta\leq \eta_0\ll_{G_1,G_2} 1$ and every function $f$ that lives at scale $\eta$, we have the polynomial contraction 
\be\label{eq: contr scale eta outline}
\|\mu^{(\ell)} \ast f\|_2\leq \eta^c \|f\|_2\qquad \text{for some integer $\ell\leq C \log(1/\eta)$,}
\ee
then for every function $g\in L^2(G_1\times G_2)$ which is orthogonal to an \emph{exceptional} finite dimensional subspace $\cal_0$, we have   $\|\mu\ast g\|_2\leq 2^{-c/C} \|g\|_2$ --- spectral gap for the space orthogonal to the exceptional subspace. We refer the reader to Definition \ref{def:fun-at-scale} for the definition of functions living at scale $\eta$.

{\bf Step 2.}\ To obtain~\eqref{eq: contr scale eta outline}, we discretize the groups $G_1$ and $G_2$ at scale $\eta^{O(1)}$ and use the spectral gap of the marginals of $\mu$ to 
find a coupling $\sigma$ of the Haar measures $m_{G_1}$ and $m_{G_2}$ which is {\em close} to $\mu^{(\ell)}$ at scale $\eta^{O(1)}$. The aforementioned coupling is constructed using the transportation problem, see~\S\ref{sec: transport}. This reduces the proof of~\eqref{eq: contr scale eta outline} to showing that for all $f$ that live at scale $\eta^{O(1)}$, we have  
\be\label{eq: contraction coupling outline}
\|\sigma^{(n_0)} \ast f\|_2\leq \eta^c \|f\|_2\qquad\text{for some $n_0$ depending only on $\dim G_1$ and $\dim G_2$.}
\ee

{\bf Step 3.}\ In this step, we use the mixing inequality~\cite[Theorem 2.6]{MMSG} and the multiscale Bourgain-Gamburd proved in~\cite[Theorem 2.12]{MMSG} to show that the failure of~\eqref{eq: contraction coupling outline} yields a partial approximate homomorphism between $G_1$ and $G_2$, in the sense of Definition~\ref{def: approxmiate local hom}.

{\bf Step 4.}\ This step relies on Theorem~\ref{thm:approximate-hom}, which is of independent interest. Indeed by loc.\ cit.\ we conclude that such partial approximate homomorphisms exists only if $G_1$ and $G_2$ are locally isomorphic. This concludes the proof of Theorem~\ref{thm:main1}. 

\medskip

Let us now briefly outline the proof of Theorems~\ref{thm:approximate-hom}. The proof relies on the Baker-Campbell-Hausdorff formula and bounded generation properties of simple Lie groups at small scales, see Proposition~\ref{prop:conjugation-by-large-ball-multiplication-large-image}. 
Our argument also relies on the effective version of Nullstellensatz in the form of {\L}ojaswicz inequality (real case) and the work of Greenberg ($p$-adic case). The details occupy \S\ref{sec: proof of thm approx} in the paper. 

We end this outline by discussing the proof of Theorem~\ref{thm:Q-form}. As alluded to before, the proof relies on Theorem~\ref{thm:main1}. Indeed, our proof of Theorem~\ref{thm:main1} yields  estimates on the implied constants in terms of group theoretic properties of the groups $G_1$ and $G_2$ in loc.\ cit. In the case considered in Theorem~\ref{thm:Q-form}, these estimates depend only on the set $\Omega$ --- this deduction relies on strong approximation theorem.
This uniformity reduces the proof to the analysis of the exceptional representation $\mathcal H_0$ appearing in Step 1. In~\cite[Theorem 9.3]{MMSG} we showed that this representation has dimension bounded by $(p_{\nu_1}p_{\nu_2})^{O(1)}$. We use this fact and results in~\cite{Golsefidy-IMRN-SA, SG-super-approx-II} to adapt the above general outline  and study functions in $\mathcal H_0$, thus completing the proof of Theorem~\ref{thm:Q-form}.

\section{Preliminaries and notation}\label{sec:def}
In this section, we will set some notation needed for the paper and recall a number of basic facts that we will refer to in the sequel.

\subsection{Analysis on compact groups}
Let $G$ be a compact (separable) topological group. As it is well known, $G$ can be equipped with a bi-invariant metric that induces the topology of $G$. All measures on $G$ in this paper will be assumed to be finite measures; thus they are automatically Radon measures. We let $m_G$ denote the unique bi-invariant probability Haar measure on $G$.
For any Borel subset $A\subseteq G$, the Haar measure of $A$ is denoted by $m_G(A)$ or $|A|$. The cardinality of a finite set $A$ will be denoted by $\# A$.

For a Borel measurable function $f: G \to \CC$, the integral of $f$ with respect to  the Haar measure is denoted, interchangeably, by  $\int_G f$ or $ \int_G f(y) \d{y}$. 
We denote by $L^p(G)$ the space $L^p(G, m_G)$. For $ f\in L^p(G)$, 
we write 
\[
\| f \|_p = \Big(  \int_G |f(x)|^p \d{x} \Big)^{1/p}.
\]

We also denote by $C(G)$ the Banach space of complex-valued continuous functions $f: G \to \CC$, equipped with the supremum norm.  
For $f, g \in L^1(G)$ the convolution $f \ast g$ is defined by 
\begin{equation}\label{eq:def-conv-func}
  ( f \ast g) (x)= \int_G f(y) g( y^{-1} x) \d{y}. 
\end{equation}
It is a fact that $(L^1(G), +, \ast)$ is a Banach algebra and if $f \in L^1(G)$ is a class function, then $f$ is in the center of this Banach algebra.

For Borel measures $\mu$ and $\nu$ on $G$, the convolution $\mu\ast \nu$ is the unique Borel measure on $G$ such that for all $f\in C(G)$,
\[ 
\int_G f \d{(\mu \ast \nu)} = \int_G \int_G f(xy) \d\mu(x) \d{\nu}(y).
\]
For a Borel measure $\mu$ on $G$ and $f \in L^1(G)$, the convolution $ \mu \ast f$ is defined by   
\begin{equation}\label{eq:def-conv-measure}
(\mu \ast f)(x) = \int_G f(y^{-1} x) \d{\mu}(y). 
\end{equation}
The following special cases of Young's inequality will be freely used in this paper: for $f, g \in L^2(G)$ and probability measure $\mu$,
\be\label{eq:Young-ineq}
\|  f \ast g   \|_2 \le \| f \|_1 \ \| g \|_2, \quad    \| f \ast g  \|_{\infty}   \le \| f \|_2 \ \| g \|_2, \quad \| \mu \ast f \|_2 \le \| f \|_2.
\ee

Let $G$ be a compact Hausdorff second countable topological group. When $\H$ is a Hilbert space and $T: \H \to \H$ is a bounded linear operator, we define the operator norm of $T$ by 
\[ \| T \|_{\op}:= \sup_{ v \in \H \setminus \{ 0 \}  } \frac{ \| Tv \|}{ \| v \|}. \]
When $\H$ is finite-dimensional, the Hilbert-Schmidt norm of $T$ is defined by 
\[ \| T \|_{\HS} := ( \tr (  TT^{\ast}) ) ^{1/2}, \]
where $T^{\ast}$ denotes the conjugate transpose of the operator $T$. Note that when $S$ and $T$ are linear operators on a finite-dimensional Hilbert space $\H$, the following 
inequality holds 
\[ \| TS \|_{\HS} \le \| T \|_{\op}  \|  S \|_{\HS}. \]

\subsection{The Peter-Weyl theorem}
The set of equivalence classes of irreducible unitary representations of $G$ is called the unitary dual of $G$ and is denoted by $ \widehat{G}$. 

The group $G$ acts on $L^2(G)$  via $ (g \cdot f)( x) = f( g^{-1}x)$, preserving the $L^2$-norm. Hence, it defines a unitary representation of $G$ on $L^2(G)$, 
the regular representation of $G$. 

Let us enumerate a number of well known facts about unitary representations of $G$. It is well known that every $ \pi \in \widehat{G}$ is of finite dimension, and that every unitary representation of $G$ can be decomposed as an orthogonal direct sum of $ \pi \in \widehat{G}$. 
A function $f \in  L^2(G)$ is called $G$-finite if there exists a finite-dimensional $G$-invariant subspace of $L^2(G)$ containing $f$. It is clear that 
$G$-finite functions form a  subspace of $L^2(G)$. We will denote this subspace by $\mathcal{E}(G)$. It follows from the classical theorem of Peter-Weyl that $\mathcal{E}(G) \subseteq C(G)$
and that $\mathcal{E}(G)$ is dense in $L^2(G)$.

For  $\pi \in \widehat{G}$ and $f \in L^1(G)$, the Fourier coefficient $ \widehat{f}(\pi)$ is defined by 
\[ \widehat{f}(\pi)= \int_G f(g) \pi(g)^{\ast} \d\mu(g). \]
One can show that for $f, g \in L^1(G)$ and $\pi \in  \widehat{G}$, we have
\[ 
\widehat{f \ast g} (  \pi)=  \widehat{g}( \pi) \widehat{f}(\pi). 
\]
Parseval's theorem states that for all $f \in L^2(G)$ the following identity holds:
\[ 
\| f \|_2^2= \sum_{\pi\in \widehat{G} } \dim \pi \h \|\widehat{f}(\pi)\|_{\HS}^2. 
\]

\subsection{Spectral Gap}

Let $\mu$ be a Borel probability measure on $G$ and $(X_i)_{i \ge 1}$ be a sequence of independent  $G$-valued random variables with probability law $\mu$. An $\ell$-step random walk on $G$ with respect to $\mu$ is given by the random variable 
\[X^{(\ell)} := X_1 \cdots X_\ell.\] 
Let us note that the law of $X^{(\ell)}$ is given by the $\ell$-fold convolution $\mu^{(\ell)}$ of $\mu$. 
Assume that $\mu$ is symmetric and that the subgroup generated by the support of $\mu$ is dense in $G$. 
Consider the averaging operator
\[ T_{\mu}: L^2_0(G) \to L^2_0(G), \quad T_{\mu} (f)= \mu \ast f,  \]
 where $ L^2_0(G):=\{g\in L^2(G)\mid \int f=0\}$.

\begin{definition}
We say that $\mu$ has \emph{the spectral gap property} if  
\[ \lambda(\mu;G):= \| {T_{\mu}} \|_{\op}  <1. \]
\end{definition}

More generally, for a subrepresentation $(\pi, \cal_\pi)$ of 
$L^2_0(G)$, we let $$\lambda(\mu;\cal_\pi):= \|T_{\mu}|_{\cal_{\pi}}\|_{\op} \quad \text{and} \quad \lcal(\mu;\cal_\pi):=-\log \lambda(\mu;\cal_\pi).$$

\subsection{Metric and R\'enyi entropy} 
In this subsection, we will collect a number of definitions from additive combinatorics and \cite{MMSG} that will be needed later. Let $G$ be as above, and let $d$ denote a bi-invariant metric on $G$.  The ball of radius $ \eta>0$ centered at $x \in G$ is denoted by $x_{\eta}$. The $ \eta$-neighborhood of a set $A$, denoted by $A_\eta$, is the union of all $x_{\eta}$ with $x \in A$.    

A subset $A \subseteq G$ is said to be  $\eta$-separated if the distance between every two points in $A$ is at least $\eta$. An $\eta$-cover for $A$ is a collection of balls of radius 
$\eta$ with centers in $A$ whose union covers $A$.  Recall that the minimum size of an $\eta$-cover of $A$ (which is finite by compactness of $G$) is denoted by $\ncal_{\eta}(A)$.
The value 
$$ h(A;\eta):=\log \ncal_{\eta}(A)$$ is called the metric entropy of $A$ at scale $\eta$. 

The characteristic function of a set $A$ is denoted by $\cf_A$. For $ \eta >0$, we write $P_{ \eta}= \frac{\cf_{1_{\eta} }}{|1_{\eta}|}$. Note that $P_\eta$ belongs to the center
of the Banach algebra $L^1(G)$.  For $ f\in L^1(G)$ ($\mu$ a probability measure on $G$, respectively) we write $f_{\eta}$ ($\mu_\eta$, respectively) instead
of $f \ast P_{\eta}$ ($\mu \ast P_{\eta}$, respectively).  

The R\'{e}nyi entropy of a $G$-valued Borel random variable $X$ at scale $\eta>0$ is defined by 
\[
H_2(X;\eta):=\log(1/|1_{\eta}|)-\log \|\mu_{\eta}\|_2^2,
\]
where $\mu$ is the probability law of $X$. As $H_2(X;\eta)$ depends only on the law $\mu$ of $X$, we will sometimes write $H_2(\mu;\eta)$ instead of
$H_2(X;\eta)$. 

Let us also recall~\cite[Definition 8.7]{MMSG}:

\begin{definition}\label{def:fun-at-scale}
We say $f\in L^2(G)$ \emph{lives at scale} $\eta$ \emph{(with parameter $0<a<1$)} if 
\begin{itemize}
	\item (Averaging to zero) $\|f_{\eta^{1/a}}\|_2\le \eta^{1/(2a)} \|f\|_2$.
	\item (Almost invariant) $\|f_{\eta^{a^2}}-f\|_2\le \eta^{a/2} \|f\|_2$.
\end{itemize} 
\end{definition}

\subsection{Local randomness and the dimension condition}\label{sec: local randomness}
In this subsection, we will recall the definitions of local randomness and the dimension condition from \cite{MMSG}.
\begin{definition}\label{def:dim-cond-intro}
Let $G$ be a compact group equipped with a compatible metric $d$. We say $(G,d)$ satisfies a dimension condition {$\Dim(C_1, d_0)$} if 
there exist $C_1 \ge 1$ and $d_0>0$ such that for all $\eta \in (0, 1 )$ the following bounds hold.
\be\label{eq:dimension-condition-intro}\tag{$\Dim$}
\frac{1}{C_1}\eta^{d_0} \le |1_{\eta}| \le C_1 \eta^{d_0}. 
\ee
\end{definition}

If the investigation involves two groups $G_1$ and $G_2$, 
we will distinguish their corresponding constants by an additional subscript, e.g., $C_{11}$ and $C_{12}$.  

\begin{definition}\label{def:local-random-with-metric}
	Suppose $G$ is a compact group and $d$ is a compatible bi-invariant metric on $G$. For parameters $C_0 \ge 1$ and $L \ge 1$, we say $(G,d)$ is $L$\emph{-locally random} with coefficient $C_0$ if for every irreducible unitary representation $\pi$ of $G$ and all $x,y\in G$ the following inequality holds:
	\be\label{eq:L-loc-rand-given-scale}
	\|\pi(x)-\pi(y)\|_{\op}\le C_0 (\dim \pi)^L d(x,y).
	\ee
	
We say a compact group $G$ is \emph{locally random} if $(G,d)$ is $L$-locally random with coefficient $C_0$ for some bi-invariant metric $d$ on $G$,  and some values of  $L$ and $C_0$.  
	\end{definition}

\subsection{Standard metrics on analytic Lie groups}\label{sec: metric p-adic real Lie}

We will be primarily interested in compact analytic real and $p$-adic Lie groups with simple Lie algebras. Let $G$ be such a group, a bi-invariant metric $d$ on $G$ will be said to be {\em standard} if $\diam(G)\leq 1$ and the following properties are satisfied: 
\begin{enumerate}
    \item If $G$ is a compact real Lie group, let $d_0$ be the metric induced from the Killing form. We assume there is $c_{d}\geq 1$ so that 
    \[
    c_{d}^{-1} d_0(g,1)\leq d(g,1)\leq c_{d} d_0(g,1)\quad\text{for all $g\in G$}
    \]
    For instance, if $G\subseteq \SO_n(\bbr)$, the metric induced by the operator norm satisfies the above property.  
    \item In the $p$-adic case, we assume $G\subseteq \SL_n(\bbz_p)$ and take $d$ to be the metric induced from the operator norm on $\SL_n(\bbz_p)$. Note that in this case: 
    \begin{enumerate}
        \item $1_\eta$ is a subgroup for every $0<\eta\leq 1$, and 
        \item the condition~\ref{eq:dimension-condition-intro}$(C,d_0)$ holds for some positive constants $C$ and $d_0=\dim G$.
    \end{enumerate}  
\end{enumerate}

\begin{lem}\label{lem: locally random family}
Both of the following hold.  
\begin{enumerate}
\item Let $G$ be a real and $p$-adic Lie groups with simple Lie algebra, then $G$ is $L$-locally random with coefficient $C_0$.

\item Let $ \bbg$ be an absolutely almost simple simply connected $\bbq$-algebraic group.
Let $ \Gamma \subseteq \bbg( \bbq)$ be a 
finitely generated Zariski dense subgroup in $ \bbg$.
Denote by $V_{ \Gamma }$  the set of all places $\nu$ of $ \bbq$ such that $\Gamma$ is a bounded subset of $ \bbg( \bbq_\nu)$.
For distinct $\nu_1, \nu_2 \in V_{ \Gamma}$, let $\Gamma_{\nu_1, \nu_2}$ denote the closure of $\Gamma$ in 
$ \bbg( \bbq_{\nu_1} ) \times \bbg( \bbq_{\nu_2} )$, and assume that $\Gamma_{\nu_1,\nu_2}=\Gamma_{\nu_1}\times\Gamma_{\nu_2}$. Then $\Gamma_{\nu_1,\nu_2}$ is $L$-locally random with coefficient $C_0$ where $L$ and $C_0$ depend only on $\Gamma$.
\end{enumerate}
\end{lem}

\begin{proof}
Part~(1) is proved in~\cite[Section 5]{MMSG}.

We now turn to the proof of part~(2). In view of~\cite[Lemma 5.2]{MMSG} and the fact that $\Gamma_{\nu_1,\nu_2}=\Gamma_{\nu_1}\times \Gamma_{\nu_2}$. It suffices to prove that $\Gamma_\nu$ is $L'$-locally random with coefficient $C'_0$ where $L'$ and $C_0'$ depend only on $\Gamma$ for all $\nu\in V_\Gamma$. To see this, let $\hat\Gamma$ denote the closure of $\Gamma$ in $\prod_{\nu\in V_{\Gamma,{\rm f}}}\bbg(\bbz_{\nu})$, where $ V_{\Gamma,{\rm f}}=V_\Gamma\setminus\{\infty\}$. Then  by~\cite[Proposition 19]{SG-super-approx-II}, for all but finitely many representations $\rho$ of $\hat\Gamma$,  we have  
\[
\ell(\rho) < \dim(\rho)^A
\]
where $A$ depends only on $\Gamma$ and $\ell(\rho)$ denotes the smallest integer $k$ so that $\rho$ factors through $k$-th congruence quotient of $\hat \Gamma$.  

Let now $\bar C$ denote the maximum of levels of the finitely many exceptional representations (in the above sense); then $\Gamma_\nu$ is $(\bar C,A)$-metric quasi random in sense of~\cite[Definition 5.7]{MMSG}. This and~\cite[Proposition 5.9]{MMSG} finish the proof.  
\end{proof}

\subsection{The transportation problem and coupling of measures}\label{sec: transport}

Recall that a coupling of probability measures $\mu_1$ and $\mu_2$ defined on probability spaces $( \Omega_1, \B_1, \mu_1)$ and $(\Omega_2, \B_2, \mu_2)$ is a probability measure $\mu$ on the product probability space $( \Omega_1 \times \Omega_2, \B_1 \otimes \B_2)$ such that 
$$\pr_i(\mu)=\mu_i$$
holds for $i=1,2$.  The set of all couplings of the probability measures $\mu_1$ and $\mu_2$ is denoted by $ \cpl(\mu_1, \mu_2)$. 
A simple example of coupling of  two measures $ \mu_1$ and $\mu_2$  is the product measure  $ \mu_1 \otimes \mu_2$. There are, however, 
many more examples. For instance, when $\Omega_1$ and $\Omega_2$ are finite sets, respectively of cardinality $n_1$ and $n_2$ elements,
then $ \cpl(\mu_1, \mu_2)$ is a convex set of dimension $(n_1-1)(n_2-1)+1$. 
When the probability spaces $\Omega_1, \Omega_2$ are finite, the question of determining the couplings has been of interest in operation research. In the special case that $ \Omega_1$ and $\Omega_2$ have the same cardinality, couplings of $\mu_1$ and $\mu_2$ correspond to doubly stochastic matrices. It is a well-known theorem that every doubly stochastic matrix can be expressed as a convex combination of 
permutation matrices. We will use a less well-known generalization of this result, which is established in~\cite{Klee-Witzgall}.

\begin{proposition}\label{prop: finding coupling for finite sets}
Let $Y_1$ and $Y_2$ be two finite sets with $|Y_i|=N_i>1$, $i=1,2$. Let $\tilde\mu$ be a probability measure on $Y_1\times Y_2$ which satisfies the following: there exists some $A> 2$ so that
\be\label{eq: marginal close to uniform}
\bigl|\pi_i\tilde\mu(y)-\tfrac{1}{N_i}\bigr|\leq \tfrac{1}{(N_1N_2)^A}\qquad\text{for all $y\in Y_i$ and $i=1,2$},
\ee
where $\pi_i$ denotes the projection onto the $i$-th coordinate. 
Then there exists a coupling $\tilde\nu$ of the uniform measures on $Y_i$ so that
\[
|\tilde\mu(y_1,y_2)-\tilde\nu(y_1,y_2)|\leq \tfrac{1}{(N_1N_2)^{A-1}}.
\] 
\end{proposition}

We begin by fixing some notation. 
Let $T(Y_1,Y_2)$ denote the set of spanning trees of the complete bipartite graph $K_{Y_1,Y_2}$.
Given probability measures $\sigma_i$ on $Y_i$, for $i=1,2$, and a spanning tree $\tau\in T(Y_1,Y_2)$, define 
\[
M_{\sigma_1,\sigma_2}^\tau: Y_1\times Y_2\to\bbr\qquad\text{by} \qquad M_{\sigma_1,\sigma_2}^\tau(y_1,y_2):=\sigma_1(Y_1')-\sigma_2(Y_2')
\] 
where $Y_i'\subset Y_i$ and $Y'_1\cup Y'_2$ is the connected component of $\tau\setminus \overline{y_1y_2}$ that contains $y_1$; 
as usual, $ \overline{y_1y_2}$ denotes the edge connecting $y_1$ to $y_2$. Put
\be\label{eq: def T sigma1 sigma2}
\T(\sigma_1,\sigma_2):=\{\tau\in T(Y_1,Y_2): M^\tau_{\sigma_1,\sigma_2}\geq 0\}.
\ee

The proof of Proposition~\ref{prop: finding coupling for finite sets} is based on the following.   

\begin{theorem}[~\cite{Klee-Witzgall}]\label{thm: KW Transport}
Let $\sigma\in\cpl(\sigma_1,\sigma_2)$. Then $\sigma$ belongs to the convex hull of 
\[
\{M^\tau_{\sigma_1,\sigma_2}: \tau\in \T(\sigma_1,\sigma_2)\}.
\] 
\end{theorem}

\begin{proof}[Proof of Proposition~\ref{prop: finding coupling for finite sets}]
Let us denote the uniform measure on $Y_i$ by $m_{Y_i}$, and write $\tilde\mu_i=\pi_i(\tilde\mu)$ for $i=1,2$. 
We first show the following: 
\be\label{eq: polytope of uniform}
\T(\tilde\mu_1, \tilde\mu_2)\subset \T(m_{Y_1}, m_{Y_2}). 
\ee

To see~\eqref{eq: polytope of uniform}, let $\tau\in\T(\tilde\mu_1, \tilde\mu_2)$. Then  
\begin{align*}
M^{\tau}_{m_{Y_1}, m_{Y_2}}(y_1,y_2)&=m_{Y_1}(Y'_1)-m_{Y_2}(Y'_2)\\
&=M^{\tau}_{\tilde\mu_1,\tilde\mu_2}(y_1,y_2)+\bigl(m_{Y_1}(Y'_1)-\tilde\mu_1(Y'_1)\bigr)-\bigl(m_{Y_2}(Y'_2)-\tilde\mu_2(Y'_2)\bigr)
\end{align*}
Since $M^{\tau}_{\tilde\mu_1,\tilde\mu_2}(y_1,y_2)\geq0$, we conclude from the above and~\eqref{eq: marginal close to uniform} that 
\be\label{eq: lower bound for M tau uniform}
M^{\tau}_{m_{Y_1}, m_{Y_2}}(y_1,y_2)\geq -\tfrac{|Y'_1|}{(N_1N_2)^A}-\tfrac{|Y'_2|}{(N_1N_2)^A}\geq -\tfrac{1}{(N_1N_2)^{A-1}}(\tfrac{1}{N_1}+\tfrac{1}{N_2})> -\tfrac{1}{N_1N_2},
\ee
where in the last inequality we used $A>2$.

On the other hand, if $M^{\tau}_{m_{Y_1}, m_{Y_2}}(y_1,y_2)<0$, then 
\[
M^{\tau}_{m_{Y_1}, m_{Y_2}}(y_1,y_2)=\tfrac{|Y'_1|}{N_1}-\tfrac{|Y'_2|}{N_2}\leq -\tfrac{1}{N_1N_2}.
\]
This and~\eqref{eq: lower bound for M tau uniform}, imply that $M^{\tau}_{m_{Y_1}, m_{Y_2}}(y_1,y_2)\geq0$ 
as was claimed in~\eqref{eq: polytope of uniform}.

Recall now that $\tilde\mu\in\cpl(\tilde\mu_1,\tilde\mu_2)$, by Theorem~\ref{thm: KW Transport}, thus, there exists 
$\{c_\tau\in [0,1]: \tau\in\T(\tilde\mu_1,\tilde\mu_2)\}$, with $\sum c_\tau=1$ so that 
\[
\tilde\mu=\sum_{\tau\in\T(\tilde\mu_1,\tilde\mu_2)} c_\tau M^\tau_{\tilde\mu_1,\tilde\mu_2}.
\]
Define $\tilde\nu:=\sum_{\tau\in\T(\tilde\mu_1,\tilde\mu_2)} c_\tau M^\tau_{m_{Y_1},m_{Y_2}}$; 
we will show that the proposition holds with $\tilde\nu$. 
In view of~\eqref{eq: polytope of uniform}, $\tilde\nu\in\cpl(m_{Y_1},m_{Y_2})$. 
Moreover, since $\sum c_\tau=1$, we have 
\begin{align*}
|\tilde\mu(y_1,y_2)-\tilde\nu(y_1,y_2)|&\leq \max_{\tau\in\T(\tilde\mu_1,\tilde\mu_2)}|M^\tau_{m_{Y_1},m_{Y_2}}(y_1,y_2)-M^\tau_{\tilde\mu_1,\tilde\mu_2}(y_1,y_2)|\\
&\leq |m_{Y_1}(Y'_1)-\tilde\mu_1(Y'_1)|+|m_{Y_2}(Y'_2)-\tilde\mu_1(Y'_2)|\\
&\leq \tfrac{1}{(N_1N_2)^{A-1}}(\tfrac{1}{N_1}+\tfrac{1}{N_2})\leq \tfrac{1}{(N_1N_2)^{A-1}}.
\end{align*}

The proof is complete.
\end{proof}

\section{Approximate homomorphisms}\label{sec: proof of thm approx}

The main goal of this section is to prove Theorem~\ref{thm:thm:approximate-hom-intro}, in which we investigate {\em partial approximate homomorphisms} 
between open compact subgroups of almost simple analytic groups, see Definition~\ref{def: approxmiate local hom}. 

A $G_1$-partial $\delta$-approximate homomorphism is 
simply called a $\delta$-approximate homomorphism. 
Approximate homomorphisms have been studied extensively. 
In~\cite{Kazhdan-approx-hom}, Kazhdan used cohomological methods to show that approximate homomorphisms from an amenable group to ${\rm U}_n(\bbc)$ are close to group homomorphisms. His argument is based on defining an averaging operator on the space of cocycles and proving that this operator is a {\em contraction}. These arguments do not appear to work when the domain of $f$ is not a group or when the target is a $p$-adic group. In~\cite{GKR-AlmostHom}, authors used a similar approach to prove that approximate homomorphisms between Lie groups are close to homomorphisms. 

In~\cite{Farah1,Farah2}, Farah, using more combinatorial techniques, proved a similar result for approximate homomorphisms from finite groups of product form. 

Our arguments here are different from these works. We rely on local analysis and passing to an infinitesimal setting, we also appeal to effective Nullstellensatz and the \L ojasiewicz inequality. 

It will be more convenient to treat the cases where $F_2=\bbq_p$ and $F_2=\bbr$, separately. 

	\subsection{Target is $p$-adic analytic}\label{sec: F2 = Qp}
In order to simplify the notation in this section, 
we will write $\done$ for $d_{01}$ and $\dtwo$ for $d_{02}$, 
also, we will write $\cone$ for $C_{11}$ and $\ctwo$ for $C_{12}$. 
Without loss of generality we will assume $C_i\geq 2$; note also that $d_i\geq 3$. 
 
	In this section, we assume that $F_2=\bbq_p$. 
	Note that in this case, $1^{(2)}_{\rho^m}$ is a normal subgroup of $G_2$. Therefore, 
	\[
	\bar{f}:1^{(1)}_{\rho}\rightarrow G_2/1^{(2)}_{\rho^m},\quad \bar{f}(g_1):=[f(g_1)]=f(g_1)_{\rho^m}
	\]
	has the following properties: for every $g_1,g_1'\in 1^{(1)}_\rho$ if $g_1g_1'\in 1^{(1)}_\rho$, then
	\[
	\bar{f}(g_1g_1')=\bar{f}(g_1)\bar{f}(g_1')\quad\text{and} \quad 
	\bar{f}(g_1^{-1})=\bar{f}(g_1)^{-1}.
	\]
	
\begin{lem}\label{lem: F1 R F2 Qp not possible}
Suppose that $F_1=\bbr$, $F_2=\bbq_p$, and $\rho<1/3$. 
Then there is no $\rho$-partial $\rho^m$ almost homomorphism from $1_\rho^{(1)}$ into $G_2$ 
so that $1_{\rho^{m'}}^{(2)}\subset \bigl(f(1_\rho^{(1)})\bigr)_{\rho^m}$ so long as $m\gg _{d_1,d_2}m'$. 
\end{lem}	

\begin{proof}
Assume contrary to the claim that $f:1_\rho^{(1)}\to G_2$ is a $\rho$-partial, $\rho^m$-almost homomorphism 
which satisfies $1_{\rho^{m'}}^{(2)}\subset \bigl(f(1_\rho^{(1)})\bigr)_{\rho^m}$. Define $\bar f$ as above. 

Since $F_1=\bbr$, $\log:1^{(1)}_\rho\rightarrow \gfr_1$ is a convergent series, and $\|\log g_1\|\leq2\|g_1-I\|$. 
We also note that 
\be\label{eq: exp and log are inverses}
\|\exp(x)-I\|\leq 3\|x\|
\ee
for every $x\in \gfr_1$ with $\|x\|<1$. 

Let $r:=|G_2/1^{(2)}_{\rho^m}|$, and for $g_1\in 1^{(1)}_{\rho/6}$, let $x:=\frac{\log g_1}{r}$. Then $\|x\|\leq \rho/(3r)$ and for every integer $0\leq j\leq r$ we have $\exp(jx)=\exp(x)^j\in 1^{(1)}_\rho$. Hence
\be\label{eq:large-kernel}
\bar{f}(g_1)=\bar{f}(\exp(x))^r=[1^{(2)}].
\ee
As $\bar{f}$ preserves multiplication, \eqref{eq:large-kernel} and~\eqref{eq:dimension-condition-intro} imply that  
\[
|\im \bar{f}|\leq e^{h(1^{(1)}_\rho;\rho/6)}\leq \cone^26^{d_{01}}
\]  
where $h(1^{(1)}_\rho;\rho/6)$ denotes the metric entropy of $1^{(1)}_\rho$ at scale $\rho/6$.  

By our assumption, $1^{(2)}_{\rho^{m'}}/1^{(2)}_{\rho^m}\subseteq\im \bar{f}$, 
which implies that $\ctwo\rho^{(m'-m)d_{02}}\leq \cone^26^{d_{01}}$. 
This is a contradiction for $m\gg_{\done,\dtwo} m
'$ as $\cone$ depends only on $\done$.  
\end{proof}

In view of Lemma~\ref{lem: F1 R F2 Qp not possible}, 
we will assume $F_1=\bbq_q$ in the remaining parts of \S\ref{sec: F2 = Qp}. 
In particular, $1^{(1)}_\rho$ is a pro-$q$ group, and $1^{(2)}_{\rho}/1^{(2)}_{\rho^m}$ is a finite $p$-group 
which is in the image of the group homomorphism $\bar{f}$. This implies that $p=q$.  
In this case, all the balls centered at the identity $1^{(i)}$ are congruence subgroups of $G_i$. The ball of radius $p^{-k}$ centered at the identity $1^{(i)}$ is 
\be\label{eq:congruence-subgp}
G_{i,k}:=\{g\in G_i|\h g\equiv I \pmod{p^k}\}.
\ee

The following theorem applied with $\varphi=\bar f$ implies Theorem~\ref{thm:thm:approximate-hom} in this case. 

\begin{theorem}\label{thm:p-adic-approx-hom}
Let $\bbg_1$ and $\bbg_2$ be almost $\bbq_p$-simple groups.
Suppose $G_i\subseteq \GL_{n_i}(\bbz_p)$ are open compact subgroups of $\bbg_i(\bbq_p)$ for $i=1,2$. 
Then there is $0<c_3\leq 1$ and a positive integer $m_0$ both depending only on $\dim \bbg_1$ and $\dim\bbg_2$ 
such that the following holds: 

If $k_0\gg_{G_1,G_2} 1$ and $m\geq m_0m'$, and
\[
\varphi:G_{1,k_0}\rightarrow G_2/G_{2,k_0m}
\]
 is a group homomorphism, satisfying that $G_{2,k_0m'}/G_{2,k_0m}\subseteq \im(\varphi)$, 
 then there is a group homomorphism $\pi:G_{1,k_0}\rightarrow G_2$ so that  
 \[
 \varphi\equiv \pi \pmod{p^{\lfloor c_3k_0 m\rfloor}}
 \]
 and $\dif \pi$ induces an isomorphism between $\Lie(\bbg_1)(\bbq_p)$ and $\Lie(\bbg_2)(\bbq_p)$.
 Moreover if $\bbg=\wt{\bbg}\otimes_{\bbq}\bbq_p$ where $\wt{\bbg}$ is an absolutely almost simple $\bbq$-group, then the constant $k_0$ depends only on $\wt\bbg$. 
\end{theorem}

\begin{proof}
Let $m_0\geq 1$ be a constant which will be explicated at the end of the argument. 

	Let $\gfr_i:=\gl_{n_i}(\bbz_p)\cap \Lie(\bbg_i)(\bbq_p)$ be the Lie $\bbz_p$-algebra of $G_i$. We start by recalling properties of finite logarithmic functions from \cite[Lemma 34]{Golsefidy-IMRN-SA}. We can and will assume that $k_0$ is large enough such that for every integer $n\geq k_0$, 
	\[
	\exp: p^n \gfr_i \rightarrow G_{i,n}
	\quad\text{and} \quad
	\log: G_{i,n}\rightarrow p^n \gfr_i
	\]
	 are well-defined and inverse of each other; in particular, $\|\exp(x)-I\|=\|x\|$ and $\|\log g\|=\|g-I\|$. For every integers $l\geq k_0$ and $n\in [l, 2l-k_0]$, the function 
	 \[
	 \Psi_{p^{l}}^{p^{n}}:G_{i,l}/G_{i,n}\rightarrow \gfr_i/p^{n-l}\gfr_i,
	 \quad
	 \Psi_{p^{l}}^{p^{n}}(gG_{i,n}):=\frac{g-1^{(i)}}{p^{l}}+p^{n-l}\gfr_i
	 \]
	  is a well-defined bijective $G_{i}$-equivariant function, where $G_{i}$ acts by conjugation on $G_{i,l}/G_{i,n}$ and on $\gfr_i/p^{n-l}\gfr_i$ via the adjoint representation. Notice that we further have 
\[
 \Psi_{p^{l}}^{p^{n}}(gG_{i,n})=\frac{\log g}{p^l}+p^{n-l}\gfr_i.
\]	    
	  If $l,l'\geq k_0$, $n\in [l,2l-k_0]$, and $n'\in [l',2l'-k_0]$, then 
	  \be\label{eq:commutator-finite-log}
	  \Psi_{p^{l+l'}}^{p^{n''}}([g_i,g'_i]G_{i,n''})=[\Psi_{p^l}^{p^n}(gG_{i,n}) , \Psi_{p^{l'}}^{p^{n'}}(gG_{i,n'})]+p^{n''-l-l'}\gfr_i
	  \ee
	  where $n'':=\min(n+l',n'+l)$.
	  
	    By \cite[Lemma 39]{SG-super-approx-II},  the Frattini subgroup $\Phi(G_{i,n})$ of $G_{i,n}$ is $G_{i,n+1}$ 
	    for every integer $n\geq k_0$. Hence for every positive integer $l\leq k_0(m-1)$, we have 
	  \be\label{eq:image-of-congruence-subgroups}
	  G_{2,k_0m'+l}/G_{2,k_0m}\subseteq  \varphi(G_{1,k_0+l}) \subseteq G_{2,1+l}/G_{2,k_0m}.
	  \ee
	  In view of \eqref{eq:image-of-congruence-subgroups}, $\varphi$ induces a group homomorphism
	  \[
	  \varphi_{l,n}:G_{1,l}/G_{1,n}\rightarrow G_{2,l-k_0+1}/G_{2,n-k_0+1},
	  \quad
	   \varphi_{l,n}(g_1 G_{1,n}):=\varphi(g_1) \pmod{G_{2,n-k_0+1}},
	  \]
	  where $k_0< l\leq n\leq k_0(m-1)$. 
	  
	  Using the finite logarithmic maps, for integers $k_0<l\leq n\leq 2l-2k_0+1\leq k_0(m-1)$, there is an additive group homomorphism $\theta_{l,n}$ such that the following is a commuting diagram:
\be\label{eq:maps-on-lie-alg}
\begin{tikzcd}
G_{1,l}/G_{1,n} \arrow[r, "\varphi_{l,n}"] \arrow[d,"\Psi_{p^l}^{p^n}"]
& G_{2,l-k_0+1}/G_{2,n-k_0+1} \arrow[d, "\Psi_{p^{l-k_0+1}}^{p^{n-k_0+1}}"] \\ 
\gfr_1/p^{n-l}\gfr_1 \arrow[r, "\theta_{l,n}"] &\gfr_2/p^{n-l}\gfr_2. 
\end{tikzcd}	  
\ee
By \eqref{eq:commutator-finite-log} and \eqref{eq:maps-on-lie-alg}, we deduce that
\[
\theta_{2l-k_0+1,l+n-k_0+1}: \gfr_1/p^{n-l}\gfr_1 \rightarrow \gfr_2/p^{n-l}\gfr_2
\]
is a Lie ring homomorphism  for integers $k_0<l\leq n\leq 2l-2k_0+1$ and $l+n-k_0+1\leq k_0(m-1)$. 
We get a Lie ring homomorphism $\theta_m:=\theta_{l_m,n_m}$ such that the following is a commuting diagram
\be\label{eq:commuting-finite-log-critical-level}
\begin{tikzcd}
G_{1,l_m}/G_{1,n_m} \arrow[r, "\varphi_{l_m,n_m}"] \arrow[d,"\Psi_{p^{l_m}}^{p^{n_m}}"]
& G_{2,l_m-k_0+1}/G_{2,n_m-k_0+1} \arrow[d, "\Psi_{p^{l_m-k_0+1}}^{p^{n_m-k_0+1}}"] \\ 
\gfr_1/p^{n_m-l_m}\gfr_1 \arrow[r, "\theta_{m}"] &\gfr_2/p^{n_m-l_m}\gfr_2,
\end{tikzcd}	  
\ee
where 
\[
l_m:=2\Big\lceil \frac{k_0}{3}m+\frac{2k_0-2}{3} \Big\rceil-k_0+1,
\quad
\text{and}
\quad
n_m:= k_0(m-1)-2.
\]
Notice that $n_m-l_m\geq  \frac{k_0}{3}(m-4)-\frac{8}{3}$.

Let $\{e_{1}^{(i)},\ldots,e_{d_i}^{(i)}\}$ be a $\bbz_p$-basis of $\gfr_i$, and suppose $c_{jks}^{(i)}\in \bbz_p$ are the corresponding structural constants. That is,
\[
[e_{j}^{(i)},e_{k}^{(i)}]=\sum_{s=1}^{d_i} c_{jks}^{(i)} e_{s}^{(i)}.
\]
Then a $\bbz_p$-linear map $T:\gfr_1\rightarrow \gfr_2, T(e_j^{(1)})=\sum_{s=1}^{d_2} x_{js} e_s^{(2)}$ is a Lie ring homomorphism if and only if $T\bigl(\bigl[e_j^{(1)},e_k^{(1)}\bigr]\bigr)=\bigl[T(e_j^{(1)}),T(e_k^{(1)})\bigr]$ for every $1\leq j,k\leq d_1$. This is equivalent to having the following equations:
\be\label{eq:equations-Lie-ring-hom}
\sum_{1\leq i_1,i_2\leq d_2} c_{i_1 i_2 r}^{(2)} x_{j i_1} x_{k i_2}=\sum_{i=1}^{d_2} c_{jki}^{(1)} x_{ri}
\ee
for every integers $1\leq j,k\leq d_1$ and $1\leq r\leq d_2$. 

Let $V$ be the $\bbz_p$-affine scheme given by equations in \eqref{eq:equations-Lie-ring-hom}. 
By the main theorem of \cite{Greenberg-local-solutions}, there are positive integers $C_4(V)$ and $c_4(V)$ such that for every point in $\overline{\xbf}\in V(\bbz_p/p^n\bbz_p)$ there is a point $\xbf\in V(\bbz_p)$ such that 
\[
\xbf \equiv \overline{\xbf} \pmod{p^{\frac{n-c_4(V)}{C_4(V)}}}.
\] 
A close examination of the argument in \cite{Greenberg-local-solutions} 
yields the following: $C_4(V)$ depends only on the number of variables and the degree of 
the defining equations; hence $C_4(V)$ only depends on $\dim \bbg_1$ and $\dim\bbg_2$. 
The constant  $c_4(V)$ could depend on the defining equations of $V$ viz.\ the complexity of rational numbers $c_{jks}^{(i)}$. 
In particular, if $\bbg=\wt{\bbg}\otimes_{\bbq}\bbq_p$ where $\wt{\bbg}$ is an absolutely almost simple $\bbq$-group, then the constant $c_4(V)$ depends only on $\wt\bbg$. 

Note that the map $\theta_m$ in~\eqref{eq:commuting-finite-log-critical-level} is a point in $V(\bbz_p/p^{n_m-l_m}\bbz_p)$. 
Applying the above with $\overline{\xbf}=\theta_m$, there exists 
a Lie ring homomorphism $\theta:\gfr_1\rightarrow \gfr_2$ such that 
\be\label{eq:finding-Lie-hom}
\theta\equiv \theta_m \pmod{p^{\lfloor c_3'k_0 m\rfloor}},
\ee
where $c_3':=(8C_4(V))^{-1}$ so long as $k_0$ is large enough so that 
\[
n_m-l_m\geq \tfrac{k_0}{3}(m-4)-\tfrac{8}{3}\geq \tfrac{k_0}{4}m\geq 2c_4(V).
\]
In view of the above discussion thus, if $\bbg=\wt{\bbg}\otimes_{\bbq}\bbq_p$ where $\wt{\bbg}$ is an absolutely almost simple $\bbq$-group, then the constant $k_0$ depends only on $\wt\bbg$ --- recall that $m\geq m_0m'\geq 1$.

Note that for $k_0\geq 2$, $p^{k_0}\gfr_i$ is a powerful Lie ring. Hence by the Baker-Campbell-Hausdorff formula (see \cite[Chapter 9.4]{Book-pAdicAnalyticProPGroups}), $\pi:G_{1,k_0}\rightarrow G_{2,k_0}, \pi(g_1):=\exp(\theta(\log(g_1)))$ is a group homomorphism. 
Moreover, by the definition the following is a commuting diagram
\be\label{eq:group-hom}
\begin{tikzcd}
G_{1,k_0} \arrow[r, "\pi"] \arrow[d,"\log"]
& G_{2,k_0} \arrow[d, "\log"] \\ 
p^{k_0}\gfr_1 \arrow[r, "\theta"] &p^{k_0} \gfr_2.
\end{tikzcd}
\ee
Using the fact that $\log$ is $G_i$-equivariant function, where $G_i$ acts on $G_{i,k_0}$ 
by conjugation and on $p^{k_0}\gfr_i$ via the adjoint action, 
by \eqref{eq:commuting-finite-log-critical-level}, \eqref{eq:finding-Lie-hom}, and \eqref{eq:group-hom}, 
we deduce that for every $\overline{x}\in \gfr_2/p^{\lfloor c_3'k_0m\rfloor}\gfr_2$ which is in the image of 
\be\label{eq:finite-log-of-approx-hom}
\Psi_{p^{l_m-k_0+1}}^{p^{l_m-k_0+1+\lfloor c_3'k_0m\rfloor}}\circ \varphi
\ee
and every $g\in G_{1, k_0}$, the following holds:
\be\label{eq:showing-pi-is-a-lift}
\Ad(\varphi(g))(\overline{x})=
 \Ad(\pi(g)) (\overline{x}).
 \ee
 Note also that, applying~\eqref{eq:image-of-congruence-subgroups} with $\ell=\ell_m$,  we conclude $p^{k_0m'} \gfr_2/p^{\lfloor c_3'k_0m\rfloor}\gfr_2$ is contained 
 in the image of the map in \eqref{eq:finite-log-of-approx-hom}. 
 Therefore, by \eqref{eq:showing-pi-is-a-lift}, 
 \[
 \Ad(\varphi(g))\equiv \Ad(\pi(g)) \pmod{p^{\lfloor c_3'k_0m\rfloor-k_0m'}}
 \]
 where $\Ad(\varphi(g))$ and $\Ad(\pi(g))$ are written in matrix form with respect to a 
 $\bbz_p$-basis of $\gfr_2$. Since $\bbg_2$ is an almost simple $\bbq_p$-group, choosing $k_0$ large enough we deduce that 
 \[
 \varphi(g)\equiv \pi(g) \pmod{p^{\lfloor c_3'k_0m\rfloor-k_0m'-k_0}}.
 \] 
The claim follows with $c_3=c'_3/2$ so long as $m\geq m_0m'\geq 8m'/c'_3$. 
\end{proof}

\subsection{Target is a Lie group}\label{sec: target Lie}
In this section, we assume that $F_2=\bbr$. 

Let $p_0:=2$ if $F_1=\bbr$, and $p_0:=p$ if $F_1=\bbq_p$. We will also use the condition~\eqref{eq:dimension-condition-intro}; 
recall that in order to simplify the notation in this section we will write $\done$ for $d_{01}$ and $\dtwo$ for $d_{02}$, 
also, we will write $\cone$ for $C_{11}$ and $\ctwo$ for $C_{12}$. 
Without loss of generality we will assume $C_i\geq 2$; note also that $d_i\geq 3$.

\begin{lem}
\label{lem:finding-far-from-id-mapped-to-not-so-far}
In the setting of Theorem~\ref{thm:approximate-hom}, suppose $F_2=\bbr$. If $m\geq\done$ and $\log_{p_0} (1/\rho)\gg_{d_i} 1$. Then there is $g_1\in 1^{(1)}_{\rho}$ such that $d(g_1,1^{(1)})\geq \rho^2$ and 
$d(f(g_1),1^{(2)})\leq \rho^{\done/(4\dtwo)}$.
\end{lem}

\begin{proof}
Let $\{h_1,\ldots,h_l\}$ be a set of maximal $\rho^2$-separated points in $1^{(1)}_\rho$. 
Hence $l\geq \rho^{-\done/2}$. The group $G_2$ can be covered by 
$\leq \ctwo \rho'^{-\dtwo}$-many balls of radius $\rho':=\ctwo \rho^{\done/(2\dtwo)}$ (recall that $\ctwo\geq 2$). 
Note that 
\[
\ctwo \rho'^{-\dtwo} =\ctwo^{1-\dtwo} \rho^{-\done/2}< \rho^{-\done/2} \leq l.
\]
Hence there are $i\neq j$ such that $d(f(h_i),f(h_j))<\rho'$. Let  $g_1:=h_ih_j^{-1}$. Then $d(g_1,1^{(1)})\geq  \rho^2$, and 
\begin{align*}
	d(f(g_1),1^{(2)})\leq &  d(f(h_i)f(h_j^{-1}),1^{(2)})+\rho^m
	\\
	 \leq  & d(f(h_i)f(h_j)^{-1},1^{(2)})+2\rho^m
	\\
  \leq &  \rho'+2\rho^m = \ctwo \rho^{\done/(2\dtwo)} +2\rho^m 
  \\
  \leq  & \rho^{\done/(4\dtwo)}.
\end{align*}
The claim follows.
\end{proof}

In the next lemma we will use the constants $C$ and $c$ appearing in Proposition~\ref{prop:conjugation-by-large-ball-multiplication-large-image}. We note that these two constants depend only on $\wt{\bbg}$ if $\bbg=\wt{\bbg}\otimes_\bbq\bbq_p$, see Proposition~\ref{prop:conjugation-by-large-ball-multiplication-large-image}.

\begin{lem}\label{lem:image-large-ball-under-approx-hom}
Let $C$ and $c$ be as in Proposition~\ref{prop:conjugation-by-large-ball-multiplication-large-image}. In the setting of Theorem~\ref{thm:approximate-hom}, suppose $m\gg_{d_{i}} 1$ and assume $\rho\leq c^2$ and $\log_{p_0}(1/\rho)\gg_{d_{i}} 1$. Then 
	\[
	f(1^{(1)}_{\rho^{7C}})\subseteq 1^{(2)}_{\rho^{{\done}/(8{\dtwo})}}.
	\]
\end{lem}

\begin{proof}
By Lemma~\ref{lem:finding-far-from-id-mapped-to-not-so-far}, there is $h_0\in 1^{(1)}_{\rho}$ such that $d(h_0,1^{(1)})\geq \rho^2$ and $d(f(h_0),1^{(2)})\leq \rho^{{\done}/(4{\dtwo})}$. 
Then by Proposition~\ref{prop:conjugation-by-large-ball-multiplication-large-image}, applied with $G_1$, 
we obtain that the following holds:
\be\label{eq:generating-large-ball}
\bigl\{(g_{1,\rho} [h_0,a_1] g_{1,\rho}^{-1})\cdots (g_{\done^2,\rho}[h_0,a_{\done^2}] g_{\done^2,\rho}^{-1})|\h a_i\in 1_{c \|h_0-I\| \rho^C}\bigr\}\supseteq 1_{c^2 \|h_0-I\|^2 \rho^{2C}},
\ee	
for some $g_{i,\rho}\in 1^{(1)}_\rho$ where $c$ is as in Proposition~\ref{prop:conjugation-by-large-ball-multiplication-large-image}.  

Recall that $\rho\leq c^2$; thus 
\be\label{eq:estimating-radius-of-the-inner-ball}
c^2 \|h_0-I\|^2 \rho^{2C} \geq c^2 \rho^{4+2C} \geq \rho^{5+2C}\geq \rho^{7C}.
\ee
Since $f$ is $\rho^m$-almost homomorphism, 
\begin{multline}\label{eq:towards-upper-bound-distance-from-identity}	
\qquad \qquad d\biggl( f\Big((g_{1,\rho} [h_0,  a_1] g_{1,\rho}^{-1}) \cdots  (g_{\done^2,\rho}[h_0,a_{\done^2}] g_{\done^2,\rho}^{-1})\Big),1^{(2)} \biggr)\leq \\d\biggl( \Big(f(g_{1,\rho}) [f(h_0),f(a_1)] f(g_{1,\rho})^{-1}\Big)\cdots \Big(f(g_{\done^2,\rho})[f(h_0),f(a_{\done^2})] f(g_{{\dtwo}^2,\rho})^{-1}\Big),1^{(2)} \biggr) +9\done^2 \rho^m.
\end{multline}
Since $d(hh',1^{(2)})\leq d(h,1^{(2)})+d(h',1^{(2)})$ and $d(hh'h^{-1},1^{(2)})=d(h',1^{(2)})$ 
for every $h,h'\in G_2$, by \eqref{eq:towards-upper-bound-distance-from-identity} we obtain the following upper bound:
\be\label{eq:upper-bound-distance-from-identity-2nd}
d\biggl( f\Big((g_{1,\rho} [h_0,  a_1] g_{1,\rho}^{-1}) \cdots  (g_{\done^2,\rho}[h_0,a_{\done^2}] g_{\done^2,\rho}^{-1})\Big),1^{(2)} \biggr)\leq \done^2 \max_{i} d([f(h_0),f(a_i)],1^{(2)})+ 9\done^2\rho^m.
\ee
Note also that $d([h,h'],1^{(2)})\leq d(h,1^{(2)})+d(h'h^{-1}h'^{-1},1^{(2)})\leq 2 d(h,1^{(2)})$ for all $h,h'\in G_2$.
Therefore, using \eqref{eq:upper-bound-distance-from-identity-2nd}, 
we deduce the following
\be\label{eq:upper-bound-distance-from-identity}
d\biggl( f\Big((g_{1,\rho} [h_0,  a_1] g_{1,\rho}^{-1}) \cdots  (g_{\done^2,\rho}[h_0,a_{\done^2}] g_{\done^2,\rho}^{-1})\Big),1^{(2)} \biggr)\leq
2 \done^2 \rho^{{\done}/(4{\dtwo})} +9 \done^2 \rho^m\leq \rho^{{\done}/(8{\dtwo})}.
\ee
By \eqref{eq:generating-large-ball}, \eqref{eq:estimating-radius-of-the-inner-ball}, 
and \eqref{eq:upper-bound-distance-from-identity}, the claim follows.
\end{proof}

\begin{lem}\label{lem: f is non-trivial on 1 rho-7C}
Let the notation be as above. If $m\geq \frac{(7C-1)\done}{\dtwo}+m'+1$ and $\log_{p_0}(1/\rho)\gg_{d_i}1$, 
then there exists some $g_1\in 1^{(1)}_{\rho^{7C}}$ so that 
\[
d(f(g_1),1^{(2)})\geq \rho^{\frac{(7C-1)\done}{\dtwo}+m'+1}.
\]
\end{lem}

\begin{proof}
Let us write $b=\frac{(7C-1)\done}{\dtwo}+m'+1$.
We prove the lemma by contradiction. Thus, assume that $f(1^{(1)}_{\rho^{7C}})\subset 1^{(2)}_{\rho^b}$. 
Fix a set $\{a_1,\ldots, a_k\}$ of coset representatives 
for $1^{(1)}_{\rho}/1^{(1)}_{\rho^{7C}}$ where $k\leq \cone^2\rho^{(1-7C)\done}$.

Recall also our assumption that
\[
1^{(2)}_{\rho^{m'}}\subset ({\rm Im}(f))_{\rho^m}=f(1^{(1)}_\rho)_{\rho^m}.
\]
Consequently, for every $h\in 1^{(2)}_{\rho^{m'}}$, there exists some $1\leq i\leq k$ and some $g'\in 1^{(1)}_{\rho^{7C}}$
so that $d(h, f(a_ig'))\leq \rho^m$. Since $f$ is $\rho^m$-almost homomorphism, we have $d(f(a_ig'), f(a_i)f(g'))\leq \rho^m$. 
Therefore,
\[
d(h, f(a_i))\leq d(h, f(a_ig'))+d(f(a_ig'), f(a_i)f(g'))+d(f(a_i)f(g'), f(a_i))\leq 3\rho^b
\] 
we used $d(f(g'),1^{(2)})\leq \rho^b$ and $m\geq b$. 
We thus conclude that   
\be\label{eq: rho ball is covered by cosets}
1^{(2)}_{\rho^{m'}}\subset \{f(a_1),\ldots, f(a_k)\}_{3\rho^b}.
\ee

In view of~\eqref{eq:dimension-condition-intro}, one gets
\[
\Bigl|\{f(a_1),\ldots, f(a_k)\}_{3\rho^b}\Bigr|\leq \cone^2\rho^{(1-7C)\done}\cdot \ctwo(3\rho^b)^{\dtwo}=3^{\dtwo}\cone^2\ctwo \rho^{(1-7C)\done+b\dtwo}\leq 3^{\dtwo}\cone^2\ctwo \rho^{(m'+1)\dtwo}.
\]
This contradicts~\eqref{eq: rho ball is covered by cosets}, if $\rho^{\dtwo}< 3^{-\dtwo}\cone^{-2}\ctwo^{-2}$. 
\end{proof}

\begin{corollary}\label{cor:p-adic-to-real-not-possible}
In the setting of Theorem~\ref{thm:thm:approximate-hom-intro}, we cannot have $F_1=\bbq_p$ and $F_2=\bbr$ if 
$m\gg_{\done,\dtwo, C} m'$ and $\log_p(1/\rho) \gg_{\done, \dtwo} 1$.
\end{corollary}

\begin{proof}
Suppose to the contrary that there is such an approximate homomorphism. 
By Lemma~\ref{lem:image-large-ball-under-approx-hom} and Lemma~\ref{lem: f is non-trivial on 1 rho-7C}, 
there exists $g_1\in 1^{(1)}_{\rho^{7C}}$ such that 
\be\label{eq: f g1 is of controlled size}
\rho^{b_1}\leq d(f(g_1), 1^{(2)})\leq \rho^{b_2}.
\ee
where $b_1=\frac{(7C-1)\done}{\dtwo}+m'+1$ and $b_2=\frac{\done}{8\dtwo}$.

Since $F_2=\bbr$, assuming $\rho^{b_2}\leq 1/10$, for all $0< k\leq 3\rho^{b_2-b_1}$ we have 
\[
d(f(g_1)^k, 1^{(2)})\geq \frac12 k \rho^{b_1};
\]
therefore, for some $0< k\leq 3\rho^{b_2-b_1}$, we have  
\be\label{eq: f g1 is of controlled size'}
d(f(g_1)^k, 1^{(2)})>2\rho^{b_2}.
\ee

Let $m$ be large enough so that $3\rho^{b_2-b_1}\rho^{m}<\rho^{b_2}$.
Since $F_1=\bbq_p$, we have $g_1^k\in 1^{(1)}_{\rho^{7C}}$, and
\begin{align}
\notag
	d(f(g_1^k),1^{(2)})&\geq d(f(g_1)^k,1^{(2)}) -(k-1)\rho^m\\
\label{eq:getting-large-in-euclidean}
&\geq 
 2\rho^{b_2} -3\rho^{b_2-b_1}\rho^{m} > \rho^{b_2}
\end{align}
where we used~\eqref{eq: f g1 is of controlled size'} and $3\rho^{b_2-b_1}\rho^{m}<\rho^{b_2}$. 

However, since $g_1^k\in 1^{(1)}_{\rho^{7C}}$, Lemma~\ref{lem:image-large-ball-under-approx-hom}
implies that $d(f(g_1^k),1^{(2)})< \rho^{b_2}$. This contradicts~\eqref{eq:getting-large-in-euclidean} and finishes the proof.   
\end{proof}

\subsection*{The case where $F_1=F_2=\bbr$}
There are certain similarities  between this case and the case where $F_1=F_2=\bbq_p$. 
However, since in this case there is no reduction mod $p$ map, the argument is more involved.

The following lemmas can be viewed as the Archimedean analogue of \eqref{eq:image-of-congruence-subgroups}. 
We start with finding a {\em large} subset, see also Lemma~\ref{lem: f is non-trivial on 1 rho-7C}.

\begin{lem}\label{lem:image-of-balls-under-approximate-hom-lower-bound}
		In the setting of Theorem~\ref{thm:approximate-hom}, suppose $F_1=F_2=\bbr$. Then there is $0<c'':=c''(G_2)\leq 1$  
		such that for every positive integer $k$ the following holds
		\[
		1^{(2)}_{c''^{k-1}\rho^{km'}}\subseteq f(1^{(1)}_{2^{k-1}\rho^k})_{6\rho^m}.
		\]
\end{lem}
\begin{proof}
We proceed by the induction on $k$. The base of induction $k=1$ is part of the assumption. By Proposition~\ref{prop:commutator-small-nhbds}, we have 
\be\label{eq:commutator-induction-step}
1^{(2)}_{c''^{k}\rho^{m'+km'}}\subseteq [1^{(2)}_{\rho^{m'}},1^{(2)}_{c''^{k-1}\rho^{km'}}].
\ee
By the hypothesis and the induction hypothesis, for every $h'\in 1^{(2)}_{\rho^{m'}}$ and $g'\in 1^{(2)}_{c''^{k-1}\rho^{km'}}$, there are $h\in 1^{(1)}_{\rho}$ and $g\in 1^{(1)}_{2^{k-1}\rho^{k}}$ such that
\be\label{eq:induction-hypothesis-comm}
f(h)\in h'_{\rho^m},\quad \text{and} \quad 
f(g)\in g'_{6\rho^m}.
\ee
By \eqref{eq:induction-hypothesis-comm}, as part of the Solovay-Kitaev theorem (see the following claim) 
for $0<\rho\ll 1$, 
\be\label{eq:perturbation-comm}
[f(h),f(g)]\in [h',g']_{O(\rho^{2m})}\subseteq  [h',g']_{\rho^{2m-1}}.
\ee
Since $f$ is a $\rho^m$-approximate homomorphism, 
\be\label{eq:approx-hom-comm}
f([h,g])\in [f(h),f(g)]_{5\rho^m}.
\ee
By \eqref{eq:approx-hom-comm} and \eqref{eq:perturbation-comm},  we obtain that the following holds
\be\label{eq:comm-close-image-approx-hom}
[h',g']\in f([h,g])_{5\rho^m+\rho^{2m-1}} \subseteq f([h,g])_{6\rho^m}.
\ee
By \eqref{eq:commutator-induction-step} and \eqref{eq:comm-close-image-approx-hom}, we deduce that
\be\label{eq:ball-image-approx-hom}
1^{(2)}_{c''^{k}\rho^{(k+1)m'}}\subseteq f([1^{(1)}_\rho,1^{(1)}_{2^{k-1}\rho^k}])_{6\rho^m}.
\ee
{\bf Claim.} Suppose $g,h\in {\rm SU}(n)$, $\|g-I\|\leq r$ and $\|h-I\|\leq r'$. Then $\|[g,h]-I\|\leq 2rr'$.

\medskip

\noindent
{\em Proof of Claim.} Suppose $x,y\in \M_n(\bbc)$ such that $\|x\|=\|y\|\leq 1$ and $g=I+rx$ and $h=I+r'y$. Then 
\begin{align*}
\|[g,h]-I\| = & \|ghg^{-1}h^{-1}-I\| = \|(gh-hg)g^{-1}h^{-1}\| \\
= & \|(I+rx)(I+r'y)-(I+r'y)(I+rx)\| \\ 
= & rr' \|xy-yx\| 
\leq  2rr'.
\end{align*}
The claim follows.

By the above claim, $1^{(1)}_{2^k\rho^{k+1}}\supseteq [1^{(1)}_\rho,1^{(1)}_{2^{k-1}\rho^k}]$.
This and \eqref{eq:ball-image-approx-hom} imply that 
$
1^{(2)}_{c''^{k}\rho^{(k+1)m'}}\subseteq f(1^{(1)}_{2^k\rho^{k+1}})_{6\rho^m},
$
which finishes the proof.
\end{proof}

In order to study the image of the restriction of $f$ to a small ball, we will be using the $n$-th roots of elements of a compact group. In the next lemma, we recall some basic properties of taking the $n$-th roots. 

For $g\in 1_{1/3}$ in a compact Lie group and positive integer $n$, we let 
\[
g^{1/n}:=\exp\biggl(\frac{1}{n}\log g\biggr);
\]
recall from the discussion leading to~\eqref{eq: exp and log are inverses}, that $\exp$ and $\log$ are well-defined on the considered neighborhoods.

\begin{lem}
    \label{lem:properties of n th root}
    In the above setting and $n\in \bbn$, the following statements hold.
    \begin{enumerate}
        \item $1_{\eta/n}\subseteq 1_\eta^{1/n}\subseteq 1_{6\eta/n}$ for every $\eta<1/3$. 
        \item $(g^{1/n})_{\eta/n}\subseteq (g_\eta)^{1/n}\subseteq (g^{1/n})_{18\eta/n}$ for every $0<\eta<1/3$ and $\|g-I\|\ll 1$ where the implied constants are universal.
    \end{enumerate}
\end{lem}
\begin{proof}
    For $g\in 1_{\eta/n}$, $g^n\in 1_{\eta}$, and so $g\in 1_{\eta}^{1/n}$. 

    For $g\in 1_{\eta}^{1/n}$, we have that $g^n\in 1_{\eta}$, and so $\|\log g^n\|\leq 2 \eta$. Hence $\|\frac{1}{n}\log g^n\| \leq 2\eta/n$, which implies 
    \[
    \|g-I\|=\|\exp\biggl(\frac{1}{n}\log g^n\biggr)-I\|\leq 3\|\frac{1}{n}\log g^n\|\leq 6\eta/n.
    \]
    The first set of inclusions follows.

    To show the second claim, we use the Baker-Campbell-Hausdorff formula. We give an extended discussion on this around \eqref{eq:BCHD}. For now, we just mention that for $x,y$ in a ball of radius $1/6$ in the Lie algebra $\gfr$, 
    \[
    x\# y :=\log(\exp(x)\exp(y))\in \gfr,
    \]
        and 
    \be\label{eq:BCH-1st-order-approx}
    \|x\# y -x -y\|\leq \bar C \|x\|\|y\|,
    \ee
    where $\bar C$ is a fixed universal constant (see \ref{eq:BCHD-2nd-approx}). 

    Suppose $h\in (g_\eta)^{1/n}$. Then $\log h=\frac{1}{n} (x\# y)$ where $x:=\log g$ and $y\in 0_{2\eta}$. Let $z:=\log (hg^{-1/n})$. Then, we have
    \be\label{eq:perturb-n-th-root}
    \frac{1}{n} (x\# y)=\log ((hg^{-1/n})g^{1/n})=z\# (x/n).
    \ee
    By \eqref{eq:BCH-1st-order-approx} and \eqref{eq:perturb-n-th-root}, we deduce that
    \be\label{eq:Lipschitz-n-th-root}
        \|z\|-\frac{1}{n}(C\|x\|+1)\|y\|\leq 
        \| \frac{1}{n} (x\# y)-\frac{x}{n}-z\| \leq
        \frac{\bar C}{n}\|z\|\|x\|. 
    \ee
    Hence if $\|x\|\leq 1/(2\bar C)$, 
    \[
    \|z\|\leq 
    \biggl(\frac{1+\bar C\|x\|}{n-\bar C\|x\|}\biggr)\|y\|\leq \frac{6}{2n-1} \eta,
    \]
    which implies that $hg^{-1/n}\in 1_{6\eta/n}$. 

    On the other hand, for every $h\in 1_{\eta/n}$, we have 
    \[
    (g^{1/n}h)^n=\underbrace{g^{1/n}h \cdots g^{1/n}h}_{n \text{ times }}
    =(g^{1/n}hg^{-1/n})(g^{2/n}hg^{-2/n})
    \cdot (g^{n/n}hg^{-n/n}) g \in g_{\eta}; 
    \]
    and so $(g^{1/n})_{\eta/n}\subseteq (g_{\eta})^{1/n}$.
\end{proof}

The next lemma extends the result of Lemma~\ref{lem:image-large-ball-under-approx-hom} to smaller balls. Essentially, 
we show that $f(1^{(1)}_r)\subseteq 1^{(2)}_a$ for small values of $r$ with the property that $a/r$ is bounded by $\rho^{-O_{G_i}(1)}$.

\begin{lem}\label{lem:image-of-balls-under-approx-hom}
In the setting of Theorem~\ref{thm:thm:approximate-hom-intro}, suppose $F_2=\bbr$. If $m\gg_{d_{0i}} 1$ and $C:=C(G_1)$ is as in Proposition~\ref{prop:conjugation-by-large-ball-multiplication-large-image}, see Lemma~\ref{lem:image-large-ball-under-approx-hom}, then for every $4\rho^{7C+m}\leq r\leq \rho^{7C}/3$ the following holds
\[
f(1^{(1)}_r)\subseteq 1^{(2)}_{s_{\rho}(r)}
\] 
where $s_\rho(r)=6(2\rho^{\done/(8\dtwo)}\rho^{-7C}r+\rho^m)$. 
\end{lem}

\begin{proof}
	Let $C=C(G_1)$ be as in Lemma~\ref{lem:image-large-ball-under-approx-hom}; let $r_0:=\rho^{7C}$ and $a_0:=\rho^{{\done}/(8{\dtwo})}$. Then 
	\be\label{eq:initial-step-image-ball}
	f(1^{(1)}_{r_0})\subseteq 1^{(2)}_{a_0}.
	\ee
For every positive integer $k$ and $g\in 1^{(1)}_{r_0/k}$, $g^k\in 1^{(1)}_{r_0}$. 
Hence using \eqref{eq:initial-step-image-ball} and the fact that $f$ is $\rho^m$-approximate homomorphism, we deduce 
\be\label{eq:f-and-1/k-roots}
f(g)^k\in f(g^k)_{k\rho^m} \subseteq 1^{(2)}_{a_0+k\rho^m}. 
\ee
Assuming that $a_0+k\rho^m<1/3$, by \eqref{eq:f-and-1/k-roots} and part~(1) of Lemma~\ref{lem:properties of n th root}, we obtain the following 
\be\label{eq:image-1/k-radius}
f(g) \in (1^{(2)}_{a_0+k\rho^m})^{1/k} \subseteq 1^{(2)}_{6(\frac{a_0}{k}+\rho^m)}.
\ee
For every $4\rho^m\leq \vare<1/3$, let $k$ be an integer such that $(k+1)^{-1}< \vare \leq k^{-1}$. 
Then $a_0+k\rho^m<1/3$, and by \eqref{eq:image-1/k-radius}, we have 
\be\label{eq:epsilon-times-r-0}
f(1^{(1)}_{r_0 \vare})\subseteq f(1^{(1)}_{r_0/k}) \subseteq 1^{(2)}_{6(\frac{a_0}{k}+\rho^m)}\subseteq 1^{(2)}_{6(2a_0\vare+\rho^m)}.
\ee
Therefore, for every $4\rho^{7C+m}\leq r < \rho^{7C}/3$, we get 
\[
f(1^{(1)}_{r})\subseteq 1^{(2)}_{6(2a_0\rho^{-7C}r+\rho^m)}
\]
which finishes the proof.
\end{proof}

\begin{lem}\label{lem:image-of-balls-under-approximate-hom-lower-bound'}
		In the setting of Theorem~\ref{thm:approximate-hom}, suppose $F_1=F_2=\bbr$. Then there is $0<\hat c:=\hat c(G_2)\leq 1$  
		such that for every positive integer $7C\leq k\leq m/3$ the following holds
		\be\label{eq: induction on k image large real}
		1^{(2)}_{{\hat c}7^{-k}\rho^{k+7Cm'}}\subseteq f\bigl(1^{(1)}_{\rho^{k-1}}\bigr)_{18k\rho^{m}}.
		\ee
  so long as $0<\rho< 1$ is small enough depending on $G_1$.  
\end{lem}

\begin{proof}
We will use the following two facts. First,
\be\label{eq: nth power approx hom}
f(g^n)\in (f(g)^n)_{n\rho^m}
\ee
if $g$ and $n$ are so that $f(g^i)$ is defined for all $1\leq i\leq n$. Second, the following consequence of the first part of Lemma~\ref{lem:properties of n th root}:
\be\label{eq: fist part of nth root use}
1^{(1)}_{\rho^{k-1}}\subseteq (1^{(1)}_{6\rho^{k-1}/n})^n\subseteq (1^{(1)}_{\rho^{k}})^n.
\ee

Let $C$ be as in Lemma~\ref{lem:image-of-balls-under-approx-hom}. Then, by Lemma~\ref{lem:image-of-balls-under-approx-hom}, $\bigl(f(1^{(1)}_{\rho^{k-1}})_{18k\rho^m}\bigr)^{1/n}$ is defined for all $7C\leq k\leq m/3$. Thus, we will assume throughout that $7C\leq k\leq m/3$.  

Suppose that $1^{(2)}_{\rho_k}\subseteq f(1^{(1)}_{\rho^{k-1}})_{18k\rho^m}$ for some $\rho_k<1/3$. Then, for $n\geq \max\{6/\rho,18\}$, we have 

\begin{align*}
    (1^{(2)})_{\rho_k/n} \subseteq  & (1^{(2)}_{\rho_k})^{1/n}
    &
    \text{by part 1 of Lemma~\ref{lem:properties of n th root}}\\
    \subseteq  & \bigl(f(1^{(1)}_{\rho^{k-1}})_{18k\rho^m}\bigr)^{1/n}
    &
    \\
    \subseteq & (f(1^{(1)}_{\rho^{k-1}})^{1/n})_{18\times 18k\rho^m/n}
    &
    \text{by part 2 of Lemma~\ref{lem:properties of n th root}}
    \\
    \subseteq  & 
    (f((1^{(1)}_{\rho^{k}})^n)^{1/n})_{18\times 18k\rho^m/n}
    &\text{by~\eqref{eq: fist part of nth root use}}
    \\
    \subseteq & (((f(1^{(1)}_{\rho^{k}})^n)_{n\rho^m})^{1/n})_{18\times 18k\rho^m/n}
    &\text{by~\eqref{eq: nth power approx hom}}
    \\
    \subseteq &
    f(1^{(1)}_{\rho^{k}})_{18\rho^m+18\times 18k\rho^m/n}
    &\text{by part 2 of Lemma~\ref{lem:properties of n th root}}
    \\
    \subseteq & f(1^{(1)}_{\rho^{k}})_{18(k+1)\rho^m}
    &\text{since $n\geq 18$.}
\end{align*}
Hence, if we assume $\rho<1/3$ and put $n=\lceil 6/\rho\rceil$, then applying the above, repeatedly, we conclude that, $1^{(2)}_{\rho_{k_0} } \subseteq f(1^{(1)}_{\rho^{k_0-1}})_{18k_0 \rho^m}$ for some $k_0\in \bbn$ implies
\be\label{eq:large-image-almost-done}
1^{(2)}_{\rho_{k_0}(\rho/7)^{k-k_0}}\subseteq f(1^{(1)}_{\rho^{k-1}})_{18k\rho^m}
\ee
for all $k\geq k_0$.

Applying Lemma~\ref{lem:image-of-balls-under-approximate-hom-lower-bound} with $7C$, we deduce that \eqref{eq:large-image-almost-done} holds for $k=k_0=7C$, $\rho\leq 2^{-7C}$, and $\rho_{k_0}:=c''^{7C-1}\rho^{7Cm'}$. Therefore for all $7C\leq k\leq m/3$, we have
 \[
1^{(2)}_{{c''}^{7C-1}7^{-k+7C}\rho^{k+7Cm'-7C}}\subseteq f\bigl(1^{(1)}_{\rho^{k-1}}\bigr)_{18k\rho^{m}},
\]
and the claim follows.
\end{proof}

For every positive integer $k$, let 
\[
\theta_k:0_{\rho/2}\rightarrow \gfr_2,\quad \theta_k(x):=\frac{\log(f(\exp(\rho^k x)))}{\rho^k}.
\]
By the definition, we have
\be\label{eq:interpretation-theta}
\exp(\rho^k \theta_k(x))=f(\exp(\rho^k x))
\ee
for every $x\in 0_{\rho/2}$. 

Lemma~\ref{lem:image-of-balls-under-approx-hom} implies that $\theta_k$ is {\em approximately} Lipschitz. 
\begin{lem}\label{lem:theta-lip}
In the above setting, for $x\in 0_{\rho/2}$, $\rho\ll 1$, and $7C< k<m$, we have
\be\label{eq:bound-norm-theta-1}
\|\theta_k(x)\|\ll \rho^{-7C} \|x\|+\rho^{m-k}.
\ee
In particular, if $\rho\ll 1$ and $7C<k< m$, then
\be\label{eq:bound-norm-theta}
\|\theta_k(x)\| < \rho^{-7C}
\ee	
for all $x\in 0_{\rho/2}$.
\end{lem}

\begin{proof}
	Since $\exp(\rho^kx)\in 1^{(1)}_{\rho^k\|x\|}$, by Lemma~\ref{lem:image-of-balls-under-approx-hom} we have 
	\[
	\exp(\rho^k \theta_k(x))=f(\exp(\rho^k x))\in 1^{(2)}_{b(\rho^{k-7C} \|x\|+\rho^m)}.
	\] 
	Hence
	$\|\rho^k \theta_k(x)\|\ll \rho^{k-7C} \|x\|+\rho^m$,
	as we claimed in~\eqref{eq:bound-norm-theta-1}.
	
	The claim in~\eqref{eq:bound-norm-theta} follows from~\eqref{eq:bound-norm-theta-1}. 
\end{proof}

To understand further properties of $\theta_k$, we start by recalling some of the consequences of the Baker-Campbell-Hausdorff-Dynkin and the Zassenhaus formulas. 

For $x,y\in \gfr_i$ with $\|x\|,\|y\|< 1/2$, put
\[
x\# y:=\log(\exp(x)\exp(y)).
\]
By the Baker-Campbell-Hausdorff-Dynkin formula, we have 
\be\label{eq:BCHD}
x\# y=\sum_{k=1}^{\infty} \frac{(-1)^{k-1}}{k} 
\sum_{m_i,n_i\geq 0, m_i+n_i>0} \frac{1}{(\sum_{i=1}^k(m_i+n_i)) \prod_{i=1}^k(m_i!n_i!)}Z_{\mbf,\nbf}(x,y),
\ee
where $\mbf:=(m_1,\ldots,m_k)$, $\nbf:=(n_1,\ldots,n_k)$, and 
\[
Z_{\mbf,\nbf}(x,y):=[\underbrace{x,\ldots,x}_{m_1},\underbrace{y,\ldots,y}_{n_1},\ldots,\underbrace{x,\ldots,x}_{m_k},\underbrace{y,\ldots,y}_{n_k}]
\]
is a long commutator. Let us observe that for every $\mbf$, $\nbf$, $x$ and $y$ we have 
\be\label{eq:trivial-upper-bound-terms-in-BCHD}
\|Z_{\mbf,\nbf}(x,y)\| \leq 2^{\|\mbf\|_1+\|\nbf\|_1-1} \|x\|^{\|\mbf\|_1} \|y\|^{\|\nbf\|_1}.
\ee
Suppose $x,y\in \gfr_i$ and $\|v\|\leq \eta \|x\|$ for some $0<\eta< 1$. Using the multi-linearity of long commutators and \eqref{eq:trivial-upper-bound-terms-in-BCHD}, we deduce that 
\be\label{eq:Lip-terms-BCHD} 
\|Z_{\mbf,\nbf}(x+v,y)- Z_{\mbf,\nbf} (x,y)\| 
\leq 2^{\|\mbf\|_1+\|\nbf\|_1+1} \eta \|x\|^{\|\mbf\|_1}\|y\|^{\|\nbf\|_1}.
\ee
We show this by induction on $\|\mbf\|_1+\|\nbf\|_1$. Here is the induction step:
\begin{align}
\notag	\|Z_{\mbf,\nbf}(x+v,y)-& Z_{\mbf,\nbf} (x,y)\|  = \|\ad(x+v)(Z_{\mbf- \ebf_1,\nbf}(x+v,y))-\ad(x)(Z_{\mbf-\ebf_1,\nbf}(x,y))\| 
	\\
\notag \leq & 2 \eta \|x\| \|Z_{\mbf- \ebf_1,\nbf}(x+v,y)\| + 2 \|x\| \|Z_{\mbf- \ebf_1,\nbf}(x+v,y)-Z_{\mbf-\ebf_1,\nbf}(x,y)\|
	\\
\notag 
\leq & 2^{\|\mbf\|_1+\|\nbf\|_1} \eta \|x\|^{\|\mbf\|_1} \|y\|^{\|\nbf\|_1} + 2^{\|\mbf\|_1+\|\nbf\|_1} \eta \|x\|^{\|\mbf\|_1} \|y\|^{\|\nbf\|_1}
\\
\notag
\leq &  2^{\|\mbf\|_1+\|\nbf\|_1+1} \eta \|x\|^{\|\mbf\|_1}\|y\|^{\|\nbf\|_1}.
\end{align}
By \eqref{eq:Lip-terms-BCHD} and \eqref{eq:BCHD}, we obtain the following  perturbation estimate: for every $x,y,v\in \gfr_i$ with $\|v\|\leq \eta \|x\|$, $\|x\|,\|y\|\ll 1$, and $0<\eta<1$,
\begin{align}
\notag
	\|(x+v)\# y- x\#y\|  & 
	\leq \Big( \sum_{k=1}^{\infty} \frac{1}{k} \sum_{m_i,n_i\geq 0, m_i+n_i>0} \frac{2^{\|\mbf\|_1+\|\nbf\|_1+1} \|x\|^{\|\mbf\|_1}\|y\|^{\|\nbf\|_1}.}{(\|\mbf\|_1+\|\nbf\|_1) \prod_{i=1}^k(m_i!n_i!)}\Big) \eta
	\\
	\label{eq:BCH-perturbation}
	\leq & 
	2\Big( \sum_{k=1}^{\infty} \frac{(e^{2\|x\|+2\|y\|}-1)^k}{k} \Big) \eta
	= 2\log((2-e^{2\|x\|+2\|y\|})^{-1}) \eta. 
\end{align}

Next we note that by \eqref{eq:BCHD} and an argument similar to \eqref{eq:BCH-perturbation}, for every $0<\eta<1 $ and $\|x\|,\|y\|< 1$ the following holds
\be\label{eq:BCHD-2nd-approx}
(\eta x) \# (\eta y)= \eta (x+ y)+ \eta^2 z, \quad\text{for some } z:=z(x,y)\in \gfr_i \text{ with } \|z\|\ll \|x\|\|y\|.
\ee
By \eqref{eq:BCH-perturbation} and \eqref{eq:BCHD-2nd-approx}, for $\|x\|\leq \|y\|<1$ and $0<\eta\ll 1$, we obtain the following upper bound,
\begin{align}
\notag
\| ((\eta x+\eta y) \# (-\eta x)) \# (-\eta y)\| = & \|(\eta y+ \eta^2 z(\eta (x+y),-\eta x))\# (-\eta y)\| 
\\
\notag
= & \|(\eta y+ \eta^2 z)\# (-\eta y)- (\eta y)\# (-\eta y)\| 
\\
\label{eq:almost-product}
\ll &  \log(2-e^{C\eta \|y\|})^{-1} \eta \ll \eta^2 \|y\|,
\end{align}
where $C$ is a universal constant and  the last inequality holds  as $\lim_{s \rightarrow 0^+} \ln(2-e^s)^{-1}/s=1$. By \eqref{eq:almost-product}, we deduce that for $\|x\|,\|y\|<1$ and $0<\eta\ll 1$, there is $z':=z'(x,y)\in\gfr_i$ such that 
$\|z'\|\ll \max\{ \|x\|,\|y\| \}$ and 
\be\label{eq:exp-almost-hom}
\exp(\eta (x+y))=\exp(\eta x)\exp(\eta y) \exp(\eta^2 z').
\ee

\begin{lem}\label{lem:almost-additive-theta}
	In the above setting for $x,y\in 0_{\rho/4}$, $ \rho\ll 1$, and $C\ll k <m$, we have 
	\[
	\| \theta_k(x+y)-(\theta_k(x)+\theta_k(y))\| \ll \rho^{k-14C}.
	\]
\end{lem}
\begin{proof} We start with the following computation of $\theta_k(x+y)$. By the definition of $\theta_k$,~\eqref{eq:interpretation-theta}, we have 
	\begin{align}
	\notag \exp(\rho^k \theta_k(x+y))= & 
	f(\exp(\rho^k(x+y))) = 
	f(\exp(\rho^k x)\exp(\rho^k y) \exp(\rho^{2k} z'))
	&\text{(by \eqref{eq:exp-almost-hom})} 	\\
	\notag
	= & f(\exp(\rho^kx))f(\exp(\rho^k y))f(\exp(\rho^{2k} z')) \exp(w)
	& \text{(for some } \|w\|\ll \rho^m) \\
	\notag
	= &
	\exp(\rho^k \theta_k(x)) \exp(\rho^k \theta_k(y)) \exp(\rho^{2k} \theta_{2k}(z')) \exp(w) \\
	\label{eq:theta-addition-1}
	= &
	\exp\Big((\rho^k \theta_k(x))\# (\rho^k \theta_k(y)) \# (\rho^{2k} \theta_{2k}(z')) \# w\Big)
	\end{align}
	By \eqref{eq:theta-addition-1}, we obtain the following
	\be\label{eq:theta-addition}
	\rho^k \theta_k(x+y)=(\rho^k \theta_k(x))\# (\rho^k \theta_k(y)) \# (\rho^{2k} \theta_{2k}(z')) \# w.
	\ee
	By \eqref{eq:bound-norm-theta} and \eqref{eq:BCHD-2nd-approx}, we deduce the following 
	\begin{align}
	\notag
		(\rho^k \theta_k(x))\# (\rho^k \theta_k(y)) = & 
			(\rho^{k-7C} (\rho^{7C}\theta_k(x)))\# (\rho^{k-7C} (\rho^{7C}\theta_k(y)))
	\\
	\label{eq:theta-addition-first-step}
	= & \rho^k(\theta_k(x)+\theta_k(y))+\rho^{2k-14C} z,
	\end{align}
	for some $z\in \gfr_2$ with $\|z\|\ll \rho^{14C}\|\theta_k(x)\|\|\theta_k(y)\|$. By \eqref{eq:theta-addition-first-step} and \eqref{eq:BCHD-2nd-approx}, we obtain that the following holds
	\begin{align}
	\notag
		(\rho^k \theta_k(x))\# (\rho^k \theta_k(y))\# (\rho^{2k} \theta_{2k}(z')) =
		&
		(\rho^k(\theta_k(x)+\theta_k(y))+\rho^{2k-14C} z)\# (\rho^{2k} \theta_{2k}(z')) 
		\\
	\notag
	= &	
	   (\rho^k(\theta_k(x)+\theta_k(y))+\rho^{2k-14C} z)+ (\rho^{2k} \theta_{2k}(z'))
	   + \rho^{2k-14C} \overline{z}
	   \\
	   \label{eq:theta-addition-second-step}
	   = &	   
	   \rho^k(\theta_k(x)+\theta_k(y))+ \rho^{2k-14C} (z+\overline{z}+\rho^{14C} \theta_{2k}(z')).
		\end{align}
By \eqref{eq:theta-addition-second-step}, \eqref{eq:BCHD-2nd-approx}, and \eqref{eq:theta-addition}, we deduce the following
\[
\| \theta_k(x+y)-(\theta_k(x)+\theta_k(y))\| \ll \rho^{k-14C},
\]
and the claim follows.
	\end{proof}
	
Using Lemma~\ref{lem:almost-additive-theta} and Lemma~\ref{lem:theta-lip}, we can show that $\theta_k$ almost preserves scaler multiplication. 
\begin{lem}\label{lem:theta-scaler-multiplication}
	In the above setting, for $x\in 0_{\rho/2}$, $\rho\ll 1$, $-1\leq t\leq 1$, and $C\ll k<m/2$, we have
	\[
	\|\theta_k(tx)-t\theta_k(x)\|\ll \rho^{\frac{k}{2}-14C}.
	\]
\end{lem}

\begin{proof}
	There is a rational number ${r}/{s}$ such that $|t-\frac{r}{s}|\leq \rho^{k/2}$ and $|r|,|s|\leq \rho^{-k/2}$. 
	By Lemma~\ref{lem:almost-additive-theta}, we have $\|s\theta_k(\frac{x}{s})-\theta_k(x)\|\ll s\rho^{k-14C}$, which implies
	\be\label{eq:theta-scaler-denom}
	\Big\|\theta_k\Big(\frac{x}{s}\Big)-\frac{1}{s}\theta_k(x)\Big\|\ll \rho^{k-14C}.
	\ee
	Similarly, we have
	\be\label{eq:theta-scaler-num}
	\Big\|\theta_k\Big(\frac{r}{s}x\Big)-r\theta_k\Big(\frac{x}{s}\Big)\Big\|\ll r \rho^{k-14C}.
	\ee 
	Combining \eqref{eq:theta-scaler-denom} and \eqref{eq:theta-scaler-num}, we conclude that
	\be\label{eq:theta-scaler-rational}
	\Big\|\theta_k\Big(\frac{r}{s}x\Big)-\frac{r}{s}\theta_k(x)\Big\|\ll r\rho^{k-14C}\ll \rho^{\frac{k}{2}-14C}.
	\ee
	
	Now Lemma~\ref{lem:almost-additive-theta} and Lemma~\ref{lem:theta-lip}, imply that 
	\begin{align}
		\notag 
			\Big\|\theta_k(tx)-\theta_k\Big(\frac{r}{s} x\Big)\Big\|\ll 
			&
			\Big\|\theta_k\Big(tx-\frac{r}{s}x\Big)\Big\| + \rho^{k-14C}
\\
\notag
\ll
&
\rho^{-7C} \Big\|tx-\frac{r}{s}x\Big\|+ \rho^{k-14C}
\\
\label{eq:rational-scale-approx}
\ll 
&
\rho^{\frac{k}{2}-14C}.
	\end{align}
	
Note also that $\|(t-\frac{r}{s})\theta_k(x)\|\ll \rho^{-7C} |t-\frac{r}{s}|\leq \rho^{\frac{k}{2}-7C}$. 
Therefore, by \eqref{eq:theta-scaler-rational} and~\eqref{eq:rational-scale-approx}, we have
\[
\|\theta_k(tx)-t\theta_k(x)\|\ll \rho^{\frac{k}{2}-14C},
\]
as it was claimed.
\end{proof}

\begin{corollary}\label{cor:theta-approx-linear}
In the above setting, suppose $\rho\ll 1$ and $C\ll k<m/2$. Then there is a linear function $\wt{\theta}_k:\gfr_1\rightarrow \gfr_2$ such that for every $x\in 0_{\rho/2}$ we have 
\[
\|\theta_k(x)-\wt{\theta}_k(x)\|\ll \rho^{\frac{k}{2}-14C}.
\]
\end{corollary}
\begin{proof}
Suppose $\{e_1,\ldots,e_{{\done}}\}$ is an orthonormal basis of $\gfr_1$. 
Let $\wt{\theta}_k:\gfr_1\rightarrow \gfr_2$ be the linear map defined by 
$\wt{\theta}_k(e_i):=(\rho/2)^{-1}\theta_k((\rho/2) e_i)$ for every $i$. 

For every $x:=\sum_{i=1}^{{\done}}t_ie_i \in 0_{\rho/2}$, we have
\begin{align*}
	\|\theta_k(x)-\wt{\theta}_k(x)\| = &
	\|\theta_k(\sum_{i=1}^{{\done}} t_i e_i)-\wt{\theta}_k(\sum_{i=1}^{{\done}}t_i e_i)\|
	\\
	\ll 
	&
	\rho^{k-14C}+\Big\|\sum_{i=1}^{{\done}} \theta_k(t_ie_i) - \frac{2t_i}{\rho} \theta_k((\rho/2) e_i)\Big\|
	&
	\text{(by Lemma~\ref{lem:almost-additive-theta})}
	\\
	\ll 
	&
	\rho^{k-14C}+\sum_{i=1}^{{\done}} \Big\|\theta_k\Big(\frac{2t_i}{\rho} (\rho/2)e_i\Big)-\frac{2t_i}{\rho} \theta_k((\rho/2) e_i)\Big\| 
	\\
	\ll & \rho^{\frac{k}{2}-14C}
	&
	\text{(by Lemma~\ref{lem:theta-scaler-multiplication})},
\end{align*}
as we claimed in the corollary.
\end{proof}

Our next task is to show that $\theta_{2k}$ almost preserves Lie algebra commutators.	
This will be done in two steps: first we show that $\theta_{2k}([x,y])$ is close to 
$[\theta_k(x),\theta_k(y)]$, Lemma~\ref{lem:theta-commutator-2k-k}, 
then we show that $\theta_{2k}$ and $\theta_k$ are close to each other, Lemma~\ref{lem:theta-2k-k-connection}. 

Let us begin with the following consequence of the Baker-Campbell-Hausdorff-Dynkin formula, see \eqref{eq:BCHD}. 
For every $0<\eta\ll 1$ and $\|x\|,\|y\|<1$, we have 
\be\label{eq:BCHD-deg-3-approximation}
(\eta x)\#(\eta y)= \eta(x+y)+\frac{\eta^2}{2}[x,y]+\frac{\eta^3}{12}\Big([x,[x,y]]-[y,[x,y]] \Big)+\eta^4 z_4,
\ee 
for some $z_4:=z_4(x,y)\in \gfr_i$ with $\|z_4\|\ll \max\{\|x\|^3\|y\|, \|x\|^2\|y\|^2,\|x\|\|y\|^3\}$. By \eqref{eq:BCHD-deg-3-approximation}, we obtain the following 
\begin{align}
\notag
	(\eta x)\# (\eta y)\# (-\eta x)\# (-\eta y) = 
	&
	\Big( \eta(x+y)+\frac{\eta^2}{2}[x,y]+\frac{\eta^3}{12}\Big([x,[x,y]]-[y,[x,y]] \Big)+\eta^4 z_4 \Big)
	\\
	\notag 
	\#
	&
	\Big( -\eta(x+y)+\frac{\eta^2}{2}[x,y]-\frac{\eta^3}{12}\Big([x,[x,y]]-[y,[x,y]] \Big)+\eta^4 z_4' \Big)	
\\
\label{eq:BCHD-commutator}
=& 
\eta^2 [x,y]	+\eta^3 z_3', 
\end{align}
for some $z_3':=z_3'(x,y)\in \gfr_i$ with $\|z_3'\|\ll \max\{\|x\|^2\|y\|,\|x\|\|y\|^2\}$.

\begin{lem}\label{lem:theta-commutator-2k-k}
	In the above setting for $x,y\in 0_{\rho/16}$, $\rho\ll 1$, and $C\ll k<m$, we have 
	\[
	\|\theta_{2k}([x,y])-[\theta_k(x),\theta_k(y)]\|\ll \rho^{2k-14C}.
	\]
The implied constants are absolute. 
\end{lem}

\begin{proof}
Again using the definition of $\theta_k$,~\eqref{eq:interpretation-theta}, we have 
\begin{align}
\notag 
\exp(\rho^{2k} \theta_{2k}([x,y]))= & f(\exp(\rho^{2k}[x,y])) 
&
\text{(by \eqref{eq:interpretation-theta})}
\\
\notag
= &
f\Bigl(\exp\Bigl(\log([\exp(\rho^k x),\exp(\rho^k y)])-\rho^{3k} z_3'\Bigr)\Bigr)
&\text{ (by \eqref{eq:BCHD-commutator})	}
\\
\notag
= &
f([\exp(\rho^k x),\exp(\rho^k y)]u')
&
\text{(for some } u'\in 1^{(1)}_{O(\rho^{3k})} \text{)}
\\
\notag
= &
[f(\exp(\rho^k x)),f(\exp(\rho^k y))]f(u')w'
&
\text{(for some } w'\in 1^{(2)}_{O(\rho^m)} \text{)}
\\
\notag
=&
[\exp(\rho^k \theta_k(x)),\exp(\rho^k \theta_k(y))]f(u')w'
&
\text{(by \eqref{eq:interpretation-theta})}
\\
\label{eq:theta-commutator-first-step}
=& 
\exp(\rho^{2k}[\theta_k(x),\theta_k(y)]+\rho^{3k-21C} z_3'')f(u')w'
&
\text{ (by~\eqref{eq:bound-norm-theta} and \eqref{eq:BCHD-commutator})	}
\end{align}
	Moreover, by Lemma~\ref{lem:image-of-balls-under-approx-hom}, we have 
\be\label{eq:error}
f(u')w' \in 1^{(2)}_{O(\rho^{3k-7C}+\rho^m)}.
\ee
By \eqref{eq:theta-commutator-first-step} and \eqref{eq:error}, we obtain the following
\be\label{eq:theta-commutator-second-step}
 \theta_{2k}([x,y])=[\theta_k(x),\theta_k(y)]+\rho^{k-21C} z'',
\ee
for some $z''\in \gfr_2$ with $\|z''\|\ll 1$. The claim follows.
\end{proof}

We now show that $\theta_k(x)$ and $\theta_{2k}(x)$ are close to each other.

\begin{lem}\label{lem:theta-2k-k-connection}
	In the above setting for $x\in 0_{\rho/2}$, $\rho\ll 1$, and $C\ll k<m/3$, we have 
	\[
	\|\theta_{2k}(x)-\theta_k(x)\|\ll \rho^{k-7C}.
	\]
\end{lem}

\begin{proof}
	Let $\ell:=\lfloor \rho^{-k} \rfloor$. By \eqref{eq:interpretation-theta}, we $\exp(\rho^{2k}\theta_{2k}(x))=f(\exp(\rho^{2k}x))$, and so
	\begin{align}
	\notag
	\exp(\ell \rho^{2k}\theta_{2k}(x))= & \bigl(f(\exp(\rho^{2k}x))\bigr)^{\ell} 
	\\
	\label{eq:theta-2k-k-connection-step-1}
	=&
	f(\exp(\ell \rho^{2k}x))u
	& 
	\text{(for some } u\in 1^{(2)}_{\ell\rho^m}\text{).}
	\end{align}
	Now note that $\|\ell \rho^{2k} x-\rho^kx\|\leq \rho^{2k}$, therefore, 
	\be\label{eq:l-power-exp-error}
	\exp(\ell \rho^{2k} x)= \exp(\rho^k x) u'
	\ee
	for some $u'\in 1^{(1)}_{O(\rho^{2k})}$. 
	By \eqref{eq:theta-2k-k-connection-step-1} and \eqref{eq:l-power-exp-error}, we deduce the following
	\begin{align}
	\notag 
		\exp(\ell \rho^{2k}\theta_{2k}(x))= &
		f(\exp(\rho^k x))f(u')w'' 
		&
    \text{(for some } w''\in 1^{(2)}_{(\ell+1)\rho^m}\text{).}
    \\
    \label{eq:theta-2k-k-connection-step-2}
    =&
    \exp(\rho^k\theta_k(x))\bar w 
    &
    \text{(for some } \bar w\in 1^{(2)}_{O(\rho^{2k-7C}+\rho^{m-k})}\text{),}
	\end{align}
	where in the last equality we used Lemma~\ref{lem:image-of-balls-under-approx-hom} 
	and the definition of $\theta_k$ in~\eqref{eq:interpretation-theta}. 
	
	In view of \eqref{eq:theta-2k-k-connection-step-2}, we have  
	\be\label{eq:theta-2k-k-connection-step-3}
	\|\ell \rho^{2k} \theta_{2k}(x)-\rho^k\theta_k(x)\|\ll \rho^{2k-7C}+\rho^{m-k}.
	\ee
	
	Recall now that $\ell:=\lfloor \rho^{-k} \rfloor$ and $\|\theta_k(x)\| < \rho^{-7C}$, see~\eqref{eq:bound-norm-theta}. 
	Therefore, 
	\[
	\|\ell \rho^{2k} \theta_{2k}(x)-\rho^k\theta_{2k}(x)\|\leq \rho^{2k-7C}.
	\]
	This and \eqref{eq:theta-2k-k-connection-step-3} imply that 
	$\|\rho^k\theta_{2k}(x)-\rho^k\theta_k(x)\|\ll \rho^{2k-7C}+\rho^{m-k}$. In consequence, we deduce
	\[
	\|\theta_{2k}(x)-\theta_k(x)\|\ll \rho^{k-7C}+\rho^{m-2k}\ll \rho^{k-7C}
	\]
	where we used $k<m/3$. 
	
	The proof is complete.
\end{proof}

\begin{corollary}\label{cor:theta-commutator}
	In the above setting, for $x,y\in 0_{\rho/16}$, $\rho\ll 1$, and $C\ll k<m/3$, we have 
	\[
	\| \theta_{2k}([x,y])-[\theta_{2k}(x),\theta_{2k}(y)]\| \ll \rho^{k-14C}.
	\]
\end{corollary}

\begin{proof}
In view of Lemma~\ref{lem:theta-2k-k-connection}, $\theta_{2k}(\bullet)=\theta_k(\bullet)+z_{\bullet}$ where 
$\|z_{\bullet}\|\ll \rho^{k-7C}$ for $\bullet=x,y$. 
 
The claim thus follows using Lemma~\ref{lem:theta-commutator-2k-k} and~\eqref{eq:bound-norm-theta}.	
\end{proof}

\begin{corollary}\label{cor:approx-Lie-algebra-hom}
	In the above setting, suppose $\rho\ll 1$, $C\ll k<m/3$, $x,y\in 0_{\rho/16}$ and $\wt{\theta}_{2k}:\gfr_1\rightarrow \gfr_2$ is the linear map as in Corollary~\ref{cor:theta-approx-linear}. Then 
	\[
	\|\wt{\theta}_{2k}(x)-\theta_{2k}(x)\|\ll \rho^{k-14C} 
	\quad
	\text{and}
	\quad
	\| \wt{\theta}_{2k}([x,y])-[\wt{\theta}_{2k}(x),\wt{\theta}_{2k}(y)] \| \ll \rho^{k-21C}.
	\]
\end{corollary}

\begin{proof}
The first claim is proved in Corollary~\ref{cor:theta-approx-linear}. 

The second claim follows from the first claim, Corollary~\ref{cor:theta-commutator}, and~\eqref{eq:bound-norm-theta}.  
\end{proof}

Similar to the $p$-adic case, we consider the set of Lie algebra homomorphisms from $\gfr_1$ to $\gfr_2$ 
which can be viewed as an affine variety as follows: 
Let $\{e_1^{(i)},\ldots,e_{d_{0i}}^{(i)}\}$ be an orthonormal basis of $\gfr_i$, 
and suppose $c_{jks}^{(i)}\in \bbr$ are the corresponding structural constants. That is,
\[
[e_{j}^{(i)},e_{k}^{(i)}]=\sum_{s=1}^{d_{0i}} c_{jks}^{(i)} e_{s}^{(i)}.
\]
Then a linear map $T:\gfr_1\rightarrow \gfr_2, T(e_j^{(1)})=\sum_{s=1}^{{\dtwo}} x_{js} e_s^{(2)}$ 
is a Lie ring homomorphism if and only if $T([e_j^{(1)},e_k^{(1)}])=[T(e_j^{(1)}),T(e_k^{(1)})]$ for every $1\leq j,k\leq {\done}$. 
This is equivalent to having the following equations on ${\rm Mat}_{\done\times \dtwo}(\bbc)$:
\be\label{eq:equations-Lie-ring-hom-real}
f_{jkr}(\xbf):=\sum_{1\leq i_1,i_2\leq {\dtwo}} c_{i_1 i_2 r}^{(2)} x_{j i_1} x_{k i_2}-\sum_{i=1}^{{\dtwo}} c_{jki}^{(1)} x_{ri}=0
\ee
for every integers $1\leq j,k\leq d_1$ and $1\leq r\leq d_2$. 

Let $V$ be the real affine variety given by equations in \eqref{eq:equations-Lie-ring-hom-real}. Note that 
$V(\bbr)$ is non-empty as it contains the zero vector. Suppose 
\[
\wt{\theta}_{2k}(e_j^{(1)})=\sum_{s=1}^{{\dtwo}} a_{js} e_s^{(2)}.
\]

Put $\abf=(a_{js})$. In view of Corollary~\ref{cor:approx-Lie-algebra-hom}, for every $1\leq j\leq \done$, we have 
\[
\Bigl\|\wt{\theta}_{2k}\Bigl(\frac{\rho}{16} e_j^{(1)}\Bigr)\Bigr\|\ll \Bigl\|\theta_{2k}\Bigl(\frac{\rho}{16} e_j^{(1)}\Bigr)\Bigr\|+\rho^{k-14C} \ll \rho^{-7C},
\]
where we used~\eqref{eq:bound-norm-theta} in the last inequality. Therefore, 
\be\label{eq:norm-of-a}
\|\abf\|\ll_{d_{i}} \rho^{-7C}.
\ee

Moreover, by Corollary~\ref{cor:approx-Lie-algebra-hom}, we have
\be\label{eq:defining-Lie-hom-value-theta-tilde}
|f_{jkr}(\abf)|\ll_{d_{i}} \rho^{k-21C}.
\ee

\begin{lem}\label{lem:close-Lie-algebra-hom}
In the above setting, for $m'\ll_{d_{i}} k< m/3$ and $0<\rho\ll_{G_1,G_2} 1$, 
there is a Lie algebra isomorphism $\wh{\theta}:\gfr_1\rightarrow\gfr_2$ with the following properties: 
\begin{enumerate}
	\item $ \|\wh{\theta}-\wt{\theta}_{2k}\|_{\rm op} \leq \rho^{k/2}$.
	\item For every $x\in \gfr_1$ with $\|x\|<\rho/2$, $
\|\wh{\theta}(x)-\theta_{2k}(x)\|\leq \rho^{k/2}$.
\end{enumerate}
\end{lem}

\begin{proof}
Let $V$ be the variety which was defined above. Note that $V(\bbr)\neq\emptyset$, indeed $0\in V(\bbr)$.
By~\cite[Theorem 7]{Eff-Loj-SemiAlgebraic}, which is a quantitative version of \L ojasiewicz inequality,
there are positive numbers $\overline{C}:=\overline{C}(V)$ and $D:=D({\done},{\dtwo})$ 
such that for every $\xbf\in \bbr^{{\done}{\dtwo}}$ with ${\rm dist}(\xbf, V(\bbr))\leq 1$ we have 
\be\label{eq:LojasiewiczInequality}
{\rm dist}(\xbf,V(\bbr)) \leq \overline{C} \max_{j,k,r}\{|f_{j,k,r}(\xbf)|\}\cdot (1+\|\xbf\|)^{D}.
\ee
Using \eqref{eq:defining-Lie-hom-value-theta-tilde}, \eqref{eq:norm-of-a}, and \eqref{eq:LojasiewiczInequality}, 
there is $\wh{\abf}\in V(\bbr)$ such that 
\be\label{eq:approx-estimate}
\|\wh{\abf}-\abf\|\leq \overline{C}  \rho^{k-21C} \rho^{-8CD} \leq \rho^{3k/4} 
\ee
so long as $\max\{1, 21C+8CD\}\leq k/8$ and $\overline{C}\rho\leq 1$.

Since $\wh{\abf}\in V(\bbr)$, it induces a Lie algebra homomorphism $\wh{\theta}:\gfr_1(\bbr)\rightarrow\gfr_2(\bbr)$; 
moreover, \eqref{eq:approx-estimate} implies the following upper bound estimate:
\be\label{eq:Lie-hom-approx-linear}
\|\wh{\theta}-\wt{\theta}_{2k}\|_{\rm op} \leq \rho^{k/2}.
\ee
In view of \eqref{eq:Lie-hom-approx-linear}, for every $x\in 0_{\rho/2}$, we have $\|\wh{\theta}(x)-\wt{\theta}_{2k}(x)\|\leq \rho^{k/2}\|x\|\leq \rho^{1+\frac{k}{2}}$. Hence, using Corollary~\ref{cor:approx-Lie-algebra-hom}, we deduce the following
\be\label{eq:Lie-hom-approx-theta}
\|\wh{\theta}(x)-\theta_{2k}(x)\|\ll \rho^{k-14C}+\rho^{1+\frac{k}{2}}\leq \rho^{k/2},
\ee
so long as $k\geq 28C+1$ and $\rho$ is small enough.

We now combine the facts that image of $\theta_{2k}$ is {\em large}, 
see Lemma~\ref{lem:image-of-balls-under-approximate-hom-lower-bound'}, and that 
$\wh{\theta}$ is linear with~\eqref{eq:Lie-hom-approx-theta}
to show that $\wh{\theta}$ is surjective. More precisely, we will show that for every $0<\vare_0\leq 0.01$ and $\rho^{\vare_0}\ll 1$, we have 
\[
(\wh{\theta}(\gfr_1))_{\rho^{k/3}} \supseteq 0_{\rho^{2\vare_0k+ 7Cm'+6}}.
\]

To that end, let us first recall from~\eqref{eq: exp and log are inverses} that for every $g\in 1^{(i)}_{1/3}$ 
and every $x\in \gfr_i$ with $\|x\|\leq 1/3$ we have 
\be\label{eq:log-exp-almost-isometries'}	
b'^{-1} \|g-1^{(i)}\| \leq \|\log g\| \leq b' \|g-1^{(i)}\|,
\quad 
\text{and}
\quad
b'^{-1} \|x\| \leq \|\exp x-1^{(i)}\| \leq b' \|x\|,
\ee
where $b'=3$.  
Increasing $b'$, if necessary, we further 
assume that $\log(g_r) \subseteq (\log g)_{b'r}$ for every $g\in 1^{(i)}_{1/6}$ and $r<1/3$.
Recall also the parameter $0<\hat c\leq 1$ from Lemma~\ref{lem:image-of-balls-under-approximate-hom-lower-bound'}. Fix some $0<\vare_0\leq 0.01$. Choose $\rho$ small enough so that 
\be\label{eq: choose rho after vare0}
\rho^{\vare_0}\leq \min\{0.1, b'^{-1}, \hat c\}.
\ee

Let $\ell=2k+3$; then $b'\rho^{\ell-1}\leq \frac12\rho^{2k+1}$. 
Thus, Lemma~\ref{lem:image-of-balls-under-approximate-hom-lower-bound'} implies 
\be\label{lem: image is large at all scales used}
1^{(2)}_{\hat c 7^{-\ell}\rho^{\ell+ 7Cm'}}\subseteq f\bigl(\exp(0_{b'\rho^{\ell-1}})\bigr)_{18\ell\rho^m}.
\ee
Combining~\eqref{eq:log-exp-almost-isometries'} and~\eqref{lem: image is large at all scales used} implies that
\be\label{eq:image-theta-large-step-1}
0_{b'^{-1} \hat c 7^{-\ell}\rho^{\ell+7Cm'}}\subseteq\log \bigl(f\bigl(\exp(0_{b'\rho^{\ell-1}})\bigr)_{18\ell\rho^m}\bigr)\subseteq 
\bigl(\log f\bigl(\exp(0_{\frac12\rho^{2k+1}})\bigr)\bigr)_{18b'\ell\rho^m}.
\ee
We also recall from \eqref{eq:interpretation-theta} that $\exp(\rho^{2k} \theta_{2k}(x))=f(\exp(\rho^{2k} x))$
for every $x\in 0_{\rho/2}$. 
Altogether, we conclude that 
\be\label{eq:image-theta-large-step-2}
0_{b'^{-1} \hat c 7^{-\ell}\rho^{\ell+7Cm'}\rho^{-2k}}\subseteq ({\rm Im}(\theta_{2k}))_{18b'\ell \rho^{m-2k}}.
\ee

We now use~\eqref{eq:image-theta-large-step-2} to complete the proof of~(3).
Recall from~\eqref{eq: choose rho after vare0}
that $\rho^{\vare_0}\leq \min\{0.1, b'^{-1}, \hat c\}$, and also recall that $\ell=2k+3$. Hence 
\begin{align}
\label{eq:image-theta-large-step-3}
b'^{-1} \hat c 7^{-\ell}\rho^{\ell+7Cm'}\rho^{-2k}&=b'^{-1}\hat c 7^{-\ell} \rho^{7Cm'+3} 
 > \rho^{2\vare_0k+7Cm'+6},
\end{align}
where we also used $3\vare_0<1$.

Combining \eqref{eq:Lie-hom-approx-theta},~\eqref{eq:image-theta-large-step-2}, 
and \eqref{eq:image-theta-large-step-3}, we conclude that 
\be\label{eq:small-angel}
(\wh{\theta}(\gfr_1))_{\rho^{k/3}}\supseteq 0_{\rho^{2\vare_0 k+7Cm'+6}}.
\ee

Since $\vare_0\leq 0.01$, \eqref{eq:small-angel} implies that $\dim \wh\theta(\gfr_1)\geq \dim \gfr_2$ so long as $k\gg m'$.
This establishes part~(3) and also shows that $\wh\theta$ is surjective.

Furthermore, since $\gfr_1$ is a simple Lie algebra and $\wh{\theta}$ is not the zero morphism, $\wh{\theta}$ is injective. 
Altogether, we conclude that $\wh{\theta}$ is an isomorphism and the proof is complete.
\end{proof}

\begin{proof}[Proof of Theorem~\ref{thm:thm:approximate-hom-intro}]
In view of Lemma~\ref{lem: F1 R F2 Qp not possible}, Theorem~\ref{thm:p-adic-approx-hom}, 
and Corollary~\ref{cor:p-adic-to-real-not-possible}, we may assume $F_1=F_2=\bbr$.
 
Let $\wh{\theta}$ be as in Lemma~\ref{lem:close-Lie-algebra-hom}. 
By \cite[Theorem 10]{OniVin}, there is a group homomorphism $\Psi:\wt{G}_1\rightarrow G_2$ where $\wt{G}_1$ is the simply-connected cover of $G_1$ such that the following is a commuting diagram
\be\label{eq:group-hom-real}
\begin{tikzcd}
\wt{G}_1 \arrow[r, "\Psi"] \arrow[d,"\log"]
& G_2 \arrow[d, "\log"] \\ 
\gfr_1 \arrow[r, "\wh{\theta}"] &\gfr_2.
\end{tikzcd}	  
\ee
Let $\iota:\wt{G}_1\rightarrow G_1$ be the covering map. Since the kernel of $\iota$ is a finite central subgroup, $\iota$ induces a homeomorphism from $\wt{1}^{(1)}_{O_{G_1}(1)}$ to $1^{(1)}_{O_{G_1}(1)}$. Hence we will view $\Psi$ as a function on $1^{(1)}_{O_{G_1}(1)}$ as well. 

Note that by \eqref{eq:group-hom-real}, for every $x\in \gfr_1$ and $g\in \wt{G}_1$, we have
\be\label{eq:equivariant-action}
\wh{\theta}(\Ad(g)(x))=\Ad(\Psi(g))(\wh{\theta}(x)).
\ee

We will show that the theorem holds with this $\Psi$.
In view of the definition of $\theta_{2k}$, see \eqref{eq:interpretation-theta}, 
for every $x\in 0_{\rho/6}$ and $g\in 1^{(1)}_{\rho/6}$, we have 
\begin{align}
\notag
\exp\bigl(\rho^{2k} \Ad(f(g))(\theta_{2k}(x))\bigr) = &
	f(g) \exp(\rho^{2k} \theta_{2k}(x)) f(g)^{-1} 
\\
\notag
= & f(g)f(\exp(\rho^{2k} x))f(g)^{-1}
\\
\notag
= & f\bigl(g	\exp(\rho^{2k} x)g^{-1}\bigr) u
& \text{(where } u\in 1^{(2)}_{3\rho^m})
\\
\notag
= & f\bigl(\exp(\rho^{2k} \Ad(g)(x))\bigr) u
\\
\label{eq:almost-equivariant-1}
= & \exp\bigl(\rho^{2k} \theta_{2k}(\Ad(g)(x))\bigr) u.
\end{align}
By \eqref{eq:almost-equivariant-1}, we deduce that for every $x\in 0_{\rho/6}$ and $g\in 1^{(1)}_{\rho/6}$ the following holds
\be\label{eq:almost-equivariant-2}
\|\Ad(f(g))(\theta_{2k}(x))-\theta_{2k}(\Ad(g)(x))\|\ll \rho^{m-2k}\leq \rho^{k/3}.
\ee
Moreover, by part~(2) of Lemma~\ref{lem:close-Lie-algebra-hom}, we have 
\be\label{eq:consequence-Lie-alg-hom-real}
\begin{aligned}
&\|\Ad(f(g))(\theta_{2k}(x))-\Ad(f(g))(\wh{\theta}(x))\|\leq \rho^{k/2}
\quad
\text{and}\\
&\|\theta_{2k}(\Ad(g)(x))-\wh{\theta}(\Ad(g)(x))\|\leq \rho^{k/2}.
\end{aligned}
\ee
Altogether, \eqref{eq:equivariant-action}, \eqref{eq:consequence-Lie-alg-hom-real}, 
and \eqref{eq:almost-equivariant-2}, imply
\begin{align}
\notag
	\|\Ad(\Psi(g))(\wh{\theta}(x))-\Ad(f(g))(\wh{\theta}(x))\|\leq &
	\|\wh{\theta}(\Ad(g)(x))-\Ad(f(g))(\theta_{2k}(x))\|+ \rho^{k/3} \\
	\notag
	\leq & 
	\|\theta_{2k}(\Ad(g)(x))-\Ad(f(g))(\theta_{2k}(x))\|+ 2 \rho^{k/3}
	\\
	\label{eq:close-group-hom}
	\leq & 3 \rho^{k/3}.
\end{align}

Now using \eqref{eq:close-group-hom}, we deduce that  
$
\|\Ad(\Psi(g))-\Ad(f(g))\|_{\op} \ll \rho^{k/4}
$
for every $g\in 1^{(1)}_{\rho/6}$. 
Finally, using the fact that $\Ad$ induces a homeomorphism on $O_{G_2}(1)$-neighborhood of $1^{(2)}$, 
we get
\[
\|\Psi(g)-f(g)\|_{\op} \ll \rho^{k/4}.
\]

This establishes the theorem for $F_1=F_2=\bbr$, and completes the proof.
\end{proof}

\section{Discretization and couplings}\label{sec: disc coupling}

The objective of this section is to show that, under mild conditions on the groups $G_1$ and $G_2$,  
one may reduce the question of spectral independence of $G_1$ and $G_2$ 
to the case of measures on $G_1\times G_2$ whose marginals are Haar measures $m_1$ and $m_2$. 

We begin with the following definition.  

\begin{definition}\label{def: discretizable}
Let $(G,d)$ be a compact metric group, and let $0<\delta<1$. 
We say $(G, d)$ is $\delta$-discretizable if the 
there exists a partition $\{X_i\}$ of $G$ satisfying the following two properties: 
\begin{enumerate}
\item $X_i$ is a Borel set for all $i$, and $|X_i|=|X_j|$ for all $i$ and $j$. 
\item $\diam (X_i)\leq \delta$ and $X_i$ contains a ball of radius $\delta^2$ for all $i$. 
\end{enumerate} 
We refer to a partition $\{X_i\}$ satisfying~(1) and~(2) above as a $\delta$-discretization of $(G,d)$.  
\end{definition}

Note that in this section $X_i$'s denote a partition for $G$ unlike in the rest of the paper where generally $X_1, X_2$ and $X$ denote random variables. 

As we have done so throughout the paper, we often drop $d$ from the notation and simply write $G$ is $\delta$-discretizable.

\medskip

An important class of examples is provided by the following proposition. 

\begin{proposition}\label{prop: Lie gp discretizable}
Suppose $G$ is a compact analytic (real or $p$-adic) Lie group, equipped with a standard bi-invariant metric, see~\S\ref{sec: metric p-adic real Lie}. Then $G$ is $\delta$-discretizable for all $0<\delta\leq \delta_0$ where $\delta_0$ is $1$ in the $p$-adic case and depends only on the dimension of $G$ in the real case. 
\end{proposition}

\begin{proof}
In the real case, the claim follows from~\cite[Theorem 2]{Discretizable}.  
Suppose $G$ is a compact $p$-adic analytic group, recall from \S\ref{sec: metric p-adic real Lie} that $1_\delta$ is a subgroup for all $\delta>0$. Let $X_i$'s be the cosets of the subgroup $1_\delta$. Then $X_i$'s form a partition of $G$ that satisfy the desired conditions.
\end{proof}

Let $G$ be a $\delta$-discretizable group, and let $\{X_i\}$ be a $\delta$-discretization of $G$. 
Then $|X_i|>0$ for all $i$. If we further assume that 
$\frac{1}{C_1}\eta^{d_0} \le |1_{\eta}| \le C_1 \eta^{d_0}$ for $\eta=\delta,\delta^2$, then 
\be\label{eq: discretizable Xi}
C_1^{-1}\delta^{2d_0}\leq |X_i|\leq C_1\delta^{d_0},\qquad\text{for all $i$.}
\ee

The following is the main result of this section.

\begin{theorem}\label{thm: continuity prop of couplings}
Let $G_1$ and $G_2$ be two compact groups. Suppose there are constants 
$C_0, C_1, L$, $d_{01}, d_{02}$, and $\rho\leq \frac{1}{100C_0(5C_1)^L}$ so that the following properties are satisfied. 

\begin{itemize}
\item $G_i$ is $L$-locally random with coefficient $C_0$ for $i=1,2$, see~\eqref{eq:L-loc-rand-given-scale}.
\item For all $\eta=\rho^j$, $j\in\mathbb N$, the group $G_i$ satisfies 
\be\label{eq:dimension condition coupling}
\frac{1}{C_1}\eta^{\dGi} \le |1_{\eta}| \le C_1 \eta^{\dGi},\qquad\text{for $i=1,2$.} 
\ee
\item For $i=1,2$, $G_i$ is $\delta$-discretizable for all $\delta=\rho^j$ with sufficiently large $j\in\mathbb N$.
\end{itemize}

Let $\mu$ be a symmetric Borel probability measure on $G_1\times G_2$ satisfying
\be\label{eq: spectral gap assum coupling}
\max\{\lambda(\pi_{1}\mu; G_1), \lambda(\pi_{2}\mu; G_1)\}=:\lambda <1
\ee
where $\pi_i$ denotes the projection onto the $i$-th factor for $i=1,2$. 

Then, there exists a symmetric coupling $\rhonu$ of 
$m_1$ and $m_2$ so that the following holds. 
Let $C>0$ and $f\in L^2(G_1\times G_2, m_1\times m_2)$ satisfy that $\|P_\crho\ast f-f\|_2\leq \crho^{C}\|f\|_2$.
Then 
\[
\bigl\|\mu^{(\ell)}\ast f-\rhonu\ast f\bigr\|_2\leq 6\crho^{C}\|f\|_2
\]
so long as $\ell\gg \log_\lambda (\rho/C_1)$, see~\eqref{eq: choice of ell coupling} for the dependence of the implied constant.  
\end{theorem}

The proof will occupy the rest of this section and will be completed in several steps. 
Let us begin with the following lemma. 

\begin{lem}\label{lem: spectral gap implies close to Haar}
Let $H$ be a compact group; assume that for some $0<\eta<1$ and constants $C_1$ and $d_0$, we have 
\be\label{eq: dim cond conv close}
C_1^{-1}\eta^{d_0}\leq |1_{\eta}|\leq C_1\eta^{d_0}.
\ee
Let $\sigma$ be a symmetric Borel probability measure on $H$ and assume that $\lambda(\sigma; H)<1$. Then 
\[
|\sigma_\eta(X)-|X||\leq \lambda(\sigma; H)\Bigl({|X|}/{|1_\eta|}\Bigr)^{1/2}\leq \lambda(\sigma; H)C_1^{1/2}\eta^{-d_0/2}{|X|}^{1/2}
\]
\end{lem}

\begin{proof}
First note that the second estimate in the above upper bound 
is a direct consequence of~\eqref{eq: dim cond conv close} and the first inequality.

We now show the first inequality. 
Recall that $\sigma_\eta(X)=\sigma\ast P_\eta(X)=\langle T_\sigma(P_\eta),\cf_X\rangle$. 
Similarly, we have $|X|=\langle T_\sigma(\cf_H),\cf_X\rangle$ (where we also used the invariance of the constant function). 
Thus, 
\begin{align}
\notag |\sigma_\eta(X)-|X||&=\langle T_\sigma(P_\eta-\cf_H), \cf_X\rangle \\
\notag&\leq \|T_\sigma(P_\eta-\cf_H)\|_2\|\cf_X\|_2\\
\label{eq: conv close to Haar}&\leq \lambda(\sigma; H)\|P_\eta-\cf_H\|_2|X|^{1/2}
\end{align}

Therefore, we need to compute $\|P_\eta-\cf_H\|_2$. By the definition we have 
\begin{align*}
\|P_\eta-\cf_H\|_2^2&=\int|P_\eta-1|^2dh\\
&=\int_{1_\eta}|\tfrac{1}{|1_\eta|}-1|^2dh+\int_{H\setminus 1_\eta}dh\\
&=|1_\eta|\tfrac{(1-|1_\eta|)^2}{|1_\eta|^2}+1-|1_\eta|= \tfrac{1-|1_\eta|}{|1_\eta|}.
\end{align*}
This and~\eqref{eq: conv close to Haar} imply that 
\[
|\sigma_\eta(X)-|X||\leq \lambda(\sigma; H)\Bigl(\tfrac{1-|1_\eta|}{|1_\eta|}\Bigr)^{1/2}|X|^{1/2}
\]
and complete the proof.
\end{proof}

We now begin the proof of the theorem.

\begin{proof}[Proof of Theorem~\ref{thm: continuity prop of couplings}]
As was mentioned before, the proof will be completed in some steps. 

Let $\hat C$ be an integer $\geq  \max\{L(C+d_{0i})+C+1,\frac{1}{d_{0i}}+1\}$, and let $\delta=\crho^{\hat C}$. 
Let   
\be\label{eq: choice of ell coupling}
\ell\geq -(2\hat C^2+0.5)\log_\lambda C_1+ \max\{d_{01}, d_{02}\}(0.5 +4\hat C^2)\log_\lambda \crho.
\ee
Then for $i=1,2$, we have 
\be\label{eq: lambda of proj of mu ell}
\lambda(\pi_i\mu^{(\ell)}; G_i)\leq \lambda^\ell\leq C_1^{-0.5-2\hat C^2}\crho^{(0.5 +4\hat C^2)\dGi}. 
\ee

Apply Lemma~\ref{lem: spectral gap implies close to Haar}, with $\eta=\crho$, $H=G_i$ and $\sigma=\pi_i\mu^{(\ell)}$ for $i=1,2$. 
Then~\eqref{eq: lambda of proj of mu ell} and~\eqref{eq:dimension condition coupling} imply
\begin{align}
\notag\bigl|\pi_i\mu^{(\ell)}_\rho(Y_i)-|Y_i|\bigr|&\leq C_1^{-0.5-2\hat C^2}\cdot \crho^{(0.5 +4\hat C^2)\dGi}\cdot C_1^{1/2}\cdot \crho^{-\dGi/2}{|Y_i|}^{1/2}\\
\label{eq: use lem sg Haar for marginals of conv}& \leq C_1^{-2\hat C^2}\rho^{4\hat C^2\dGi}{|Y_i|}^{1/2}
\end{align}
for any Borel subset $Y_i\subset G_i$. 

\subsection*{The definition of $\rhonu$} 
Recall that by our assumption $G_1$ and $G_2$ are $\delta$-discretizable. 
For $i=1,2$, let $\{X^{i}_j: 1\leq j\leq N_i\}$ be a $\delta$-discretization of $G_i$. 
In view of~\eqref{eq:dimension condition coupling} (with $\eta=\delta, \delta^2$) and~\eqref{eq: discretizable Xi}, we have
\be\label{eq: Ni bound}
C_1^{-1}\delta^{2\dGi}\leq 1/N_i\leq C_1\delta^{\dGi}\qquad\text{for $i=1,2$.}
\ee

Let $Z_i=\{1,\ldots, N_i\}$ for $i=1,2$, and define $\tilde\mu$ on $Z_1\times Z_2$ by
\[
\tilde\mu(j,k)=\mu^{(\ell)}_\rho(X_{j}^1\times X_{k}^{2}).
\]
Then for all $1\leq j\leq N_1$, we have 
\begin{align}
\notag\bigl|\pi_1\tilde\mu(j)-\tfrac{1}{N_1}\bigr|&=\bigl|\mu^{(\ell)}_\rho(X_{j}^1\times G_2)-|X_{j}^1\times G_2|\bigr|\\
\notag&=\bigl|\pi_1\mu^{(\ell)}_\rho(X_{j}^1)-|X_{j}^1|\bigr|\leq C_1^{-2\hat C^2}\rho^{4\hat C^2\dGi}{|X_{j}^1|}^{1/2}\\
\label{eq: hat mu satisfies transport hyp}&\leq (1/N_1N_2)^{\hat C}
\end{align}
where the second to the last inequality follows from~\eqref{eq: use lem sg Haar for marginals of conv} 
and the last inequality follows from~\eqref{eq: Ni bound}.

Similarly, for all $1\leq k\leq N_2$, we have $|\pi_2\tilde\mu(k)-\tfrac{1}{N_2}|\leq (1/N_1N_2)^{\hat C}$. 

Altogether, the conditions of Proposition~\ref{prop: finding coupling for finite sets} are satisfied for $\tilde\mu$ with $A=\hat C$.  
Therefore, by that proposition, there exist $\{c_{j,k}\in [0,1]: 1\leq j\leq N_1, 1\leq k\leq N_2\}$ so that all the following hold.
\begin{enumerate}
\item For every $1\leq j\leq N_1$, we have $\sum_{k=1}^{N_2} c_{j,k}=\frac{1}{N_1}$.
\item For every $1\leq k\leq N_2$, we have $\sum_{j=1}^{N_1} c_{j,k}=\frac{1}{N_2}$.
\item For all $1\leq j\leq N_1$ and all $1\leq k\leq N_2$ we have 
\be\label{eq: mu rho and cjk}
|\mu_\rho^{(\ell)}(X_j^1\times X_k^2)-c_{j,k}|\leq (1/N_1N_2)^{\hat C-1}.
\ee
\end{enumerate} 

Let $\rhonup$ be the probability measure on $G_1\times G_2$ defined using the density 
\[
N_1N_2\sum_{j,k}c_{j,k}\cf_{X_j^1\times X_k^2};
\] 
note that $\rhonup$ depends on $\rho$. Abusing the notation, we also refer to the density of $\rhonup$ by $\rhonup$.   

Bulk of the proof is to show that $\rhonup$ satisfies the claim in the theorem, (possibly) except for being symmetric; 
the proof will then be completed by symmetrizing $\rhonup$.  

\begin{sublemma}\label{lem: rho nu is a coupling of Haars}
The measure $\rhonup$ is a coupling of $m_1$ and $m_2$. 
\end{sublemma}

\begin{proof}[Proof of the Sublemma]
Since $\rhonup$ is absolutely continuous with respect to $m_1\times m_2$, with density 
$N_1N_2\sum_{j,k}c_{j,k}\cf_{X_j^1\times X_k^2}$, it suffices to show that for $i=1,2$, we have 
\[
\int_{G_i}N_1N_2\sum_{j,k}c_{j,k}\cf_{X_j^1\times X_k^2}(g_1,g_2)dm_i=1.
\]
We prove this claim for $i=1$, the other case is proved similarly. 
Recall that $|X_j^1|=1/N_1$ and that $\sum_{j=1}^{N_1} c_{j,k}=\frac{1}{N_2}$ for all $k$. Thus we have 
\begin{align*}
\int_{G_1}&N_1N_2\sum_{j,k}c_{j,k}\cf_{X_j^1\times X_k^2}(g_1,g_2)dm_1(g_1)=N_1N_2\sum_{j,k}c_{j,k}|X_j^1|\cf_{X^2_k}(g_2)\\
&=N_2\sum_{j,k}c_{j,k}\cf_{X^2_k}(g_2)=N_2\sum_k\cf_{X^2_k}(g_2)\sum_jc_{j,k}=\sum_k\cf_{X^2_k}(g_2)=1  
\end{align*} 
where in the last equality we used the fact that $\{X_k^2\}$ is a partition of $G_2$. 
\end{proof}

Recall that $f\in L^2(G_1\times G_2, m_1\times m_2)$ and satisfies $\|f-f_\crho\|_2\leq \crho^{C}\|f\|_2$ where 
$f_\crho=P_\crho\ast f$. 

\begin{sublemma}\label{lem: replacing f by f rho}
We have
\be\label{eq: replacing f by f rho}
\bigl\|\mu^{(\ell)}\ast f-\rhonup\ast f\bigr\|_2\leq \bigl\|\mu^{(\ell)}_\crho\ast f_\rho-\rhonup\ast f_\crho\bigr\|_2+3\crho^{C}\|f\|_2.
\ee
\end{sublemma}

\begin{proof}[Proof of the Sublemma]
Indeed, by Young's inequality, we have 
\begin{subequations}
\begin{align*}
\bigl\|\mu^{(\ell)}\ast f-\mu^{(\ell)}\ast f_\crho\bigr\|_2&\leq\|f-f_\crho\|_2\leq \crho^C\|f\|_2,\\
\bigl\|\mu^{(\ell)}\ast f_{\crho}-\mu^{(\ell)}_\crho\ast f_\crho\bigr\|&=\bigl\|\mu^{(\ell)}\ast P_\crho\ast (f-f_\crho)\bigr\|_2\leq \|f-f_\crho\|_2\leq \crho^C\|f\|_2,\\
\bigl\|\rhonup\ast f_\crho-\rhonup\ast f\bigr\|_2&\leq\|f-f_\crho\|_2\leq \crho^C\|f\|_2.
\end{align*}
\end{subequations}
Now~\eqref{eq: replacing f by f rho} follows from these estimates and the triangle inequality. 
\end{proof}

In view of~\eqref{eq: replacing f by f rho}, thus we need to bound $\bigl\|\mu^{(\ell)}_\crho\ast f_\rho-\rhonup\ast f_\crho\bigr\|_2$.
This will be done using the Parseval's theorem. Let us begin with the following  
which is a consequence of the fact that $G_1$ and $G_2$ are 
locally random groups, together with the fact that $\diam(X^i_j)\leq \delta=\rho^{\hat C}$. 

\begin{sublemma}\label{lem: mu ell nu Fourier}
Let $\sigma$ be a Borel probability measure on $G=G_1\times G_2$.
Let $\varphi\in\widehat{G}$, and for all $1\leq j\leq N_1$ and $1\leq k\leq N_2$, let $g_{j,k}\in X_j^1\times X_k^2$. Then  
\[
\bigl\|\hat\sigma(\varphi)-\sum_{j,k}\sigma(X_j^1\times X_k^2)\varphi(g_{j,k})\bigr\|_{\op}\leq 2C_0(\dim\varphi)^L\delta.
\]
\end{sublemma}

\begin{proof}[Proof of the Sublemma]
Since $\{X_j^1\times X_k^2\}$ is a partition of $G$ with Borel sets, we have 
\begin{align*}
\hat\sigma(\varphi)&=\int \varphi(g)d\sigma(g)=\sum_{j,k}\int_{X_j^1\times X_k^2}\varphi(g)d\sigma(g)\\
&=\sum_{j,k}\Bigl(\int_{X_j^1\times X_k^2}\varphi(g)-\varphi(g_{j,k})d\sigma(g)+\sigma(X_j^1\times X_k^2)\varphi(g_{j,k})\Bigr).
\end{align*}

Recall that $G_1$ and $G_2$ are $L$-locally random with coefficient $C_0$, thus, 
$G$ is $L$-locally random with coefficient $2C_0$,~\cite[Lemma 5.2]{MMSG}.
In consequence, for all $g\in X_j^1\times X_k^2$, we have 
\[
\|\varphi(g)-\varphi(g_{j,k})\|_{\op}\leq 2C_0\dim(\varphi)^Ld(g,g_{j,k})\leq 2C_0\dim(\varphi)^L\delta,
\]
where we used $\diam(X^1_j\times X^2_k)\leq \delta$. 

Altogether, we conclude that  
\[
\bigl\|\hat\sigma(\varphi)-\sum_{j,k}\sigma(X_j^1\times X_k^2)\varphi(g_{j,k})\bigr\|_{\op}\leq \sum_{j,k}\int_{X_j^1\times X_k^2}2C_0\dim(\varphi)^L\delta d\sigma(g)=2C_0\dim(\varphi)^L\delta
\]
where we used the fact that $\{X_j^1\times X_k^2\}$ is a Borel partition of $G$ and $\sigma(G)=1$.
\end{proof}

Applying the above with $\sigma=\rhonup$ and $\mu^{\ell}_\crho$, we conclude the following 
\begin{subequations}
\begin{align}
\label{eq: mu ell nu Fourier 1}&\bigl\|\hat\rhonup(\varphi)-\sum_{j,k}c_{j,k}\varphi(g_{j,k})\bigr\|_{\op}\leq 2C_0(\dim\varphi)^L\delta,\quad\text{and}\\
\label{eq: mu ell nu Fourier 2}&\bigl\|\widehat{\mu^{(\ell)}_\crho}(\varphi)-\sum_{j,k}\mu^{(\ell)}_\crho(X_j^1\times X_k^2)\varphi(g_{j,k})\bigr\|_{\op}\leq 2C_0(\dim\varphi)^L\delta
\end{align}
\end{subequations}
for all $\varphi\in\widehat G$, where we also used the fact that $\rhonup(X_j^1\times X_k^2)=c_{j,k}$. 

We now combine~\eqref{eq: mu ell nu Fourier 1},~\eqref{eq: mu ell nu Fourier 2}, and~\eqref{eq: mu rho and cjk} with Parseval's theorem to deduce the following 

\begin{sublemma}
We have 
\be\label{eq: nu almost satisfies the thm}
\bigl\|\mu^{(\ell)}_\crho\ast f_\crho -\rhonup\ast f_\crho\bigr\|_2\leq 3\rho^{C}\|f\|_2.
\ee
\end{sublemma} 

\begin{proof}[Proof of the Sublemma]
The argument is similar to arguments in~\cite[\S6]{MMSG}, 
and as was mentioned before, is based on Parseval's theorem:
\[
\bigl\|\mu^{(\ell)}_\crho\ast f_\crho -\rhonup\ast f_\crho\bigr\|_2^2=\sum_{\varphi\in\widehat G}\dim\varphi\bigl\|(\widehat{\mu^{(\ell)}_\crho}(\varphi)-\hat\rhonup(\varphi))\hat f_\crho(\varphi)\bigr\|_{\HS}^2
\]
where $G=G_1\times G_2$.

Let us write $d_0=d_{01}+d_{02}$, see~\eqref{eq:dimension condition coupling}. 
Let $D=\lceil 4C_1^2\rho^{-2C-d_0}\rceil$, this choice will be justified later in the proof. We separate the above sum into $\sum_{\dim\varphi\leq D}$ and $\sum_{\dim\varphi\geq D}$. Using the notation in~\cite[\S6]{MMSG}, the first sum will be denoted by $L(\mu^{(\ell)}_\crho\ast f_\crho -\rhonup\ast f_\crho; D)$ and second sum by $H(\mu^{(\ell)}_\crho\ast f_\crho -\rhonup\ast f_\crho; D)$.

First note that in view of~\cite[Lemma 6.1]{MMSG}, we have 
\begin{align}
\notag H\bigl((\mu^{(\ell)}_\crho-\rhonup)\ast f_\crho; D\bigr)&=H\bigl((\mu^{(\ell)}_\crho-\rhonup)\ast f\ast P_\crho; D\bigr)\leq \frac{1}{D}H\bigl((\mu^{(\ell)}_\crho-\rhonup)\ast f;D\bigr)H(P_\crho;D)\\
\notag&\leq \frac{1}{D}\|\mu^{(\ell)}_\crho-\rhonup)\ast f\|_2^2\|P_\crho\|_2^2\leq \frac{4}{D|1_\crho|}\|f\|_2^2\\
\label{eq: high frequency mu ell nu}&\leq \frac{4C_1}{D\crho^{d_0}}\|f\|_2^2
\end{align}
where we used Young's in equality in the second line and~\eqref{eq:dimension condition coupling} with $\eta=\rho$ in the last line.

We now investigate 
\[
L(\mu^{(\ell)}_\crho\ast f_\crho -\rhonup\ast f_\crho; D)=\sum_{\varphi\in\widehat G, \dim\varphi\leq D}\dim\varphi\bigl\|(\widehat{\mu^{(\ell)}_\crho}(\varphi)-\hat\rhonup(\varphi))\hat f_\crho(\varphi)\bigr\|_{\HS}^2.
\]
First note that $\|(\widehat{\mu^{(\ell)}_\crho}(\varphi)-\hat\rhonup(\varphi))\hat f_\crho(\varphi)\bigr\|_{\HS}^2\leq \|\widehat{\mu^{(\ell)}_\crho}(\varphi)-\hat\rhonup(\varphi)\bigr\|_{\op}^2\|\hat f_\crho(\varphi)\bigr\|_{\HS}^2$. Moreover, by the triangle inequality, we have 
\begin{multline*}
\bigl\|\widehat{\mu^{(\ell)}_\crho}(\varphi)-\hat\rhonup(\varphi)\bigr\|_{\op}\leq \bigl\|\widehat{\mu^{(\ell)}_\crho}(\varphi)-\sum_{j,k}\mu_\rho^{\ell}(X_j^1\times X_k^2)\varphi(g_{j,k})\bigr\|_{\op}+ \bigl\|\widehat{\rhonup}(\varphi)-\sum_{j,k} c_{j,k}\varphi(g_{j,k})\bigr\|_{\op}\\+ 
\bigl\|\sum_{j,k}\bigl(\mu_\rho^{\ell}(X_j^1\times X_k^2)-c_{j,k}\bigr)\varphi(g_{j,k})\bigr\|_{\op}.
\end{multline*}
Hence, using~\eqref{eq: mu ell nu Fourier 1},~\eqref{eq: mu ell nu Fourier 2}, and~\eqref{eq: mu rho and cjk}, we have 
\[
\bigl\|\widehat{\mu^{(\ell)}_\crho}(\varphi)-\hat\rhonup(\varphi)\bigr\|_{\op}\leq 4C_0(\dim\varphi)^L\delta+N_1N_2(1/N_1N_2)^{\hat C-1}\leq
4C_0(\dim\varphi)^L\delta+ (C_1^2\delta^{d_0})^{\hat C-2}
\]
where $d_0=d_{01}+d_{02}$ and we used~\eqref{eq: Ni bound} for the last inequality. From this we conclude that

\be\label{eq: low frequency mu ell nu}
\begin{split}      
L(\mu^{(\ell)}_\crho\ast f_\crho -\rhonup\ast f_\crho; D) & \leq \Bigl(4C_0D^L\delta+ (C_1^2\delta^{d_0})^{\hat C-2}\Bigr)^2\sum_{\varphi\in\widehat G, \dim\varphi\leq D} \dim\varphi\|\hat f_\crho(\varphi)\bigr\|_{\HS}^2
\\
& \leq \Bigl(4C_0D^L\delta+ (C_1^2\delta^{d_0})^{\hat C-2}\Bigr)^2\|f\|_2^2.
\end{split}
\ee

Recall that $D=\lceil 4C_1\rho^{-2C-d_0}\rceil$, then $\frac{4C_1}{D\rho^{d_0}}\leq \rho^{2C}$. 
Since $\hat C\geq L(C+d_0)+C+1$, and $\rho\leq \frac{1}{4C_0(5C_1)^L}$, we get 
\[
4C_0D^L\delta\leq 4C_0 (4C_1\rho^{-C-d_0}+1)^L\rho^{\hat C}\leq \rho^C;
\]
moreover, since $d_0(\hat C-1)\geq 1$ we have $(C_1^2\delta^{d_0})^{\hat C-2}\leq \rho^C$. 
Thus,~\eqref{eq: high frequency mu ell nu} and~\eqref{eq: low frequency mu ell nu} imply that
\[
\bigl\|\mu^{(\ell)}_\crho\ast f_\crho -\rhonup\ast f_\crho\bigr\|_2\leq 3\rho^C
\]
as we claimed. 
\end{proof}

We are now in the position to complete the proof of Theorem~\ref{thm: continuity prop of couplings}. 
First note that in view of~\eqref{eq: replacing f by f rho} and~\eqref{eq: nu almost satisfies the thm} we have 
\[
\bigl\|\mu^{(\ell)}\ast f-\rhonup\ast f\bigr\|_2 \leq 6\rho^C\|f\|_2
\]

As was mentioned before, $\rhonup$ need not be symmetric. 
Define $\rhonu=\frac{\check \rhonup+\rhonup}{2}$ to be the symmetrization of $\rhonup$ where 
$\check h(g)=h(g^{-1})$ for any $h\in L^2(G, m_1\times m_2)$.

Recall that $\|\check h\|_2=\|h\|_2$ for all $h\in L^2(G, m_1\times m_2)$. Since $\mu$ is symmetric, we have 
\[
\bigl\|\mu^{(\ell)}\ast f-\check\rhonup\ast f\bigr\|_2 = \bigl\|(\mu^{(\ell)}\ast \check f-\rhonup\ast \check f)\check{}\;\bigr\|_2= \bigl\|\mu^{(\ell)}\ast \check f-\rhonup\ast \check f\bigr\|_2\leq 6\rho^C\|\check f\|_2=6\rho^C\|f\|_2.
\]
Altogether, and using the facts that $\rho\leq 0.01$, we get that 
\[
\bigl\|\mu^{(\ell)}\ast f-\rhonu\ast f\bigr\|_2\leq 6\rho^C\|f\|_2\leq \rho^{C-0.5}\|f\|_2.
\]
The proof is complete. 
\end{proof}

\section{Contraction of couplings at small scales}\label{s:contraction-couplings}

The main goals of this section are to prove Proposition~\ref{prop:desired-entropy} and Proposition~\ref{prop:contraction-coupling}, which are crucial ingredients for the proof of Theorem~\ref{thm:main1}.  

In this section we will be working with groups $G_1$ and $G_2$ which are $L$-locally random with coefficients $C_{01}$ and $C_{02}$, respectively, and  satisfy ${\rm DC}(d_{0i},C_{1i})$.
Throughout this section, $\mathsf D_{\,\bullet}$ denotes a constant of the form 
\be\label{eq:const-kazaa}
\X;
\ee
this means the exponent in the definition of $\mathsf D_{\,\bullet}$ does not depend on other parameters introduced in the various statements throughout this section.

\begin{proposition}\label{prop:desired-entropy}
Suppose $F_1$ and $F_2$ are two local fields of characteristic zero, $\bbg_i$ is an almost $F_i$-simple group, and $\Lie(\bbg_1)(F_1)$ and $\Lie(\bbg_2)(F_2)$ are not isomorphic. For $i=1,2$, let $G_i\subseteq \bbg_i(F_i)$ be a compact open subgroup. For every $\bar\delta>0$, there exists 
$\eta_0\geq \ref{c: desired entropy}^{-1/\bar\delta}$ where $\constc \label{c: desired entropy}$ is a constant as in \eqref{eq:const-kazaa}, and a positive integer $m:=m(\bar\delta)$ such that for every $0<\eta\leq \eta_0$ and every coupling $\mu$ of the probability Haar measures $m_{G_1}$ and $m_{G_2}$, we have 
\[
H_2(\mu^{(2^{m})};\eta)\geq \big(d_{01}+d_{02}-\bar\delta\big) \log(1/\eta).
\] 
\end{proposition}

The proof of this proposition will occupy the rest of this section. We will then use this proposition to prove Proposition~\ref{prop:contraction-coupling} below. Before stating Proposition~\ref{prop:contraction-coupling}, we recall Definition~\ref{def:fun-at-scale}: A function $f\in L^2(G)$ is said to live at scale  $\eta$ (with parameter $0<a<1$) if 
\begin{itemize}
	\item (Averaging to zero) $\|f_{\eta^{1/a}}\|_2\le \eta^{1/(2a)} \|f\|_2$.
	\item (Almost invariant) $\|f_{\eta^{a^2}}-f\|_2\le \eta^{a/2} \|f\|_2$.
 \end{itemize}

\begin{proposition}\label{prop:contraction-coupling}
In the setting of Proposition~\ref{prop:desired-entropy}, there exist a positive integer $m_0$ and a positive number $c$, depending only on $L$, $d_{01}$, and $d_{02}$ such that for every $0<\eta\leq \ref{c: contraction-coupling}^{-1}$ where $\constc \label{c: contraction-coupling}$ is a constant as in \eqref{eq:const-kazaa}, every coupling $\mu$ of the probability Haar measures $m_{G_1}$ and $m_{G_2}$, and every function $f\in L^2(G_1\times G_2)$ which lives at scale $\eta$ with a parameter $a\geq 4L(d_{01}+d_{02})$, we have 
\[
\|\mu^{(2^{m_0})}\ast f\|_2\leq \eta^c \|f\|_2.
\]  
\end{proposition}

Proposition~\ref{prop:contraction-coupling}, whose proof is based on Proposition~\ref{prop:desired-entropy}, is a crucial ingredient in the proof of Theorem~\ref{thm:main1}.

\subsection{Contraction, R\'enyi entropy, and approximate subgroups.}
In this section, using the mixing inequality as in \cite[Theorem 2.6]{MMSG} and the multi-scale version of a result of Bourgain and Gamburd (see \cite[Theorem 2.12]{MMSG}), we justify why in the proofs of the aforementioned propositions one needs to study certain type of approximate subgroups of $G_1\times G_2$.

\medskip

We start by finding a lower bound for the R\'{e}nyi entropy of every coupling of the Haar measures $m_{G_1}$ and $m_{G_2}$. 
\begin{lem}\label{lem:coupling-entropy-lower-bound}
Let $G_1$ and $G_2$ be two compact groups and $\mu$ be a coupling of the Haar measures $m_{G_1}$ and $m_{G_2}$. Then for every $0<\eta<1$, we have
\[
H_2(\mu;\eta)\geq \max(\log(1/|1^{(1)}_\eta|),\log(1/|1^{(2)}_\eta|)).
\]
\end{lem}
\begin{proof}
By \cite[Lemma 8.2]{MMSG}, we have $\mu_{\eta}(x)=\mu(x_\eta)/|1_{\eta}|$ for every $x\in G_1\times G_2$. Therefore
\[
\mu_{\eta}(x,x')=\frac{\mu(x_\eta\times x'_\eta)}{|1^{(1)}_\eta||1^{(2)}_\eta|}\leq \frac{\mu(G_1 \times x'_\eta)}{|1^{(1)}_\eta||1^{(2)}_\eta|}=\frac{|x'_\eta|}{|1^{(1)}_\eta||1^{(2)}_\eta|}=\frac{1}{|1^{(1)}_\eta|}.
\]
By symmetry, we have 
\be\label{eq:coupling-contraction-towards-initial-entropy}
\|\mu_\eta\|_{\infty} \leq \min\Big(\frac{1}{|1^{(1)}_\eta|},\frac{1}{|1^{(2)}_\eta|}\Big). 
\ee
Since $\mu_\eta$ is a probability measure, we deduce from~\eqref{eq:coupling-contraction-towards-initial-entropy} that
\begin{align*}
H_2(\mu;\eta)= & \log(1/|1_{\eta}|)-\log \|\mu_{\eta}\|_2^2\\ 
\geq & 
\log(1/|1^{(1)}_\eta|)+\log(1/|1^{(2)}_\eta|)-\log \|\mu_\eta\|_\infty \\
\geq &
\max(\log(1/|1^{(1)}_\eta|),\log(1/|1^{(2)}_\eta|)),
\end{align*}
as we claimed.
\end{proof}

\begin{lem}\label{lem:lower-bound-entropy-gives-contraction}
Suppose $G$ is an $L$-locally random group with coefficient $C_0$ which satisfies the dimension condition ${\rm DC}(d_0,C_1)$. Let $0<\eta<(10C_0+C_1)^{-8La}$, where $a$ is a positive number. 
Let $0<\rho\leq \eta$ and suppose that $\nu$ is a probability measure on $G$ such that 
\be\label{eq:coupling-target-lower-bound-entropy}
H_2(\nu;\rho)\geq \Big(d_0-\frac{1}{8La\log_{\eta}\rho}\Big) \log(1/\rho).
\ee
Then for every function $f\in L^2(G)$ that lives at scale $\eta$, we have 
\[
\|\nu_{\rho}\ast f\|_2\leq \eta^{1/(8La)} \|f\|_2.
\] 
\end{lem}

\begin{proof}
Notice that $C_0\eta^{1/(4a)}\leq 0.1$, hence by \cite[Theorem 2.6]{MMSG}, we have
\begin{align}
\notag \|\nu_{\rho}\ast f\|_2^2\leq &  
2 \|(\nu_{\rho})_{\eta^{1/a}}\ast f_{\eta^{1/a}}\|_2^2+\eta^{1/(2La)} \|\nu_{\rho}\|_2^2\h \|f\|_2^2
\\
\label{eq:coupling-contraction-mixing-inequality}
\leq & 
2 \eta^{1/(2a)} \|f\|_2+\eta^{1/(2La)} \|\nu_{\rho}\|_2^2\h \|f\|_2^2 & (f \text{ lives at scale } \eta).
\end{align}
On the hand, by \eqref{eq:coupling-target-lower-bound-entropy}, we obtain 
\begin{align}
\notag 
\log \|\nu_\rho\|_2^2 \leq & \log(1/|1_{\rho}|)-d_0\log(1/\rho)+\frac{1}{8La\log_\eta\rho}\log(1/\rho)
\\
\notag \leq & \log C_1+\frac{1}{8La}\log(1/\eta) &\text{(because of DC)}
\\
\label{eq:coupling-lower-bound-entropy-upper-bound-l2-norm} 
\leq & \frac{1}{4La}\log(1/\eta) &(\text{as } \eta<(10C_0+C_1)^{-8La}) 
\end{align}
By \eqref{eq:coupling-contraction-mixing-inequality} and \eqref{eq:coupling-lower-bound-entropy-upper-bound-l2-norm}, we deduce
\[
\|\nu_\rho\ast f\|_2^2\leq (2\eta^{1/(2a)}+\eta^{1/(4La)})\|f\|^2_2\leq \eta^{1/(8La)} \|f\|_2^2,
\]
and the claim follows.
\end{proof}

Our general strategy for the proofs of the main  propositions is as follows: starting with the initial entropy provided by Lemma~\ref{lem:coupling-entropy-lower-bound}, if we can show that each time after doubling the number of steps in the random walk we can gain $\gamma_0 \log(1/\eta)$ additional R\'{e}nyi entropy at scale $\eta$, then in $O(1)$-steps we reach to the desired lower bound for the R\'{e}nyi entropy that is given in Lemma~\ref{lem:lower-bound-entropy-gives-contraction}. 

The following lemma, which follows from \cite[Theorem 2.12]{MMSG}, is an important tool in carrying out the above strategy. Roughly speaking, it states that the failure to gain R\'{e}nyi entropy can happen only because of algebraic obstructions. 

\begin{lem}\label{lem:coupling-BG}
Suppose $G$ satisfies ${\rm DC}(d_0,C_1)$ and $X,X'$ are independent and identically distributed random variables with values in $G$. Then for every positive number $\gamma_0$, either 
\be\label{eq:entropy-gain-BG}
H_2(XX';\eta)\geq H_2(X;\eta)+\gamma_0 \log(1/\eta)
\ee
or there are $H\subseteq G$ and $x,y\in G$ such that 
\begin{enumerate}
\item (Approximate structure) $H$ is $R(1/\eta)^{R\gamma_0}$-approximate subgroup.
\item (Metric entropy) $|h(H;\eta)-H_2(X;\eta)|\leq R\gamma_0 \log(1/\eta)$.
\item (Almost equidistribution) Let $Z$ be a random variable with the uniform distribution over $1_{3\eta}$ independent of $X$. Then 
\[
\bbp(XZ\in (xH)_{\eta})\geq \eta^{R\gamma_0} \text{ and }
\bbp(XZ\in (Hy)_{\eta})\geq \eta^{R\gamma_0}.
\] 
Moreover,
\[
|\{h\in H_\eta|\h \bbp(X\in (xh)_{3\eta})\geq (C_12^{d_0})^{-R} \eta^{10\gamma_0} 
2^{-H_2(X;\mu)}\}|\geq \eta^{R\gamma_0} |H_\eta|
\]
where $R$ is a universal fixed number.
\end{enumerate}
\end{lem}
\begin{proof}
This is an immediate corollary of \cite[Theorem 2.12]{MMSG}.
\end{proof}

\subsection{Approximate subgroups and approximate homomorphisms.}
In view of Lemma~\ref{lem:coupling-BG}, we focus on the understanding of almost subgroups $H$ of $G_1\times G_2$ which satisfy properties given in Lemma~\ref{lem:coupling-BG}.
Indeed, we will interpret this approximate structure, as a local approximate group homomorphism with {\em large} image from a {\em large} ball in $G_1$ to $G_2$. Then we apply Theorem~\ref{thm:approximate-hom} to complete the proof.

Let us begin with an application of a product result proved in \cite[Theorem 2.8]{MMSG}. 

\begin{lem}\label{lem:coupling-plus-almost-equidistribution-imlpies-large-open-projections}
Suppose $G_1$ and $G_2$ are $L$-locally random with coefficients $C_{01}$ and $C_{02}$, respectively. Suppose $G_i$ satisfies ${\rm DC}(d_{0i},C_{1i})$ for $i=1,2$. Then for every $0<\vare<1$, there is $\gamma:=\gamma(\vare,L,d_{01},d_{02})\ll_{L,d_{01}d_{02}}\vare$ such that for every $\eta$ that satisfies $\eta^\vare \le  \constc\label{c:lem coupl}^{-1}$, where $\ref{c:lem coupl}$ is a constant as in \eqref{eq:const-kazaa},
 the following holds. Suppose $X:=(X_1,X_2)$ is a random variable with values in $G:=G_1\times G_2$ such that $X_i$ is uniformly distributed in $G_i$. Let $Z$ be a random variable independent of $X$ and with uniform distribution over $1_{3\eta}\subseteq G_1\times G_2$ with respect to the maximum metric. Suppose $H\subseteq G$, $x\in G$, and $\bbp(XZ\in (xH)_\eta)\geq \eta^{R\gamma}$ where $R$ is a universal fixed number. Then 
\[
\textstyle
\pr_i(H_\eta H_\eta H_\eta^{-1} H_\eta^{-1})\supseteq 1^{(i)}_{\eta^\vare}
\]
for $i=1,2$ where $\pr_i:G\rightarrow G_i$ is the projection to the $i$-th component.
\end{lem}

\begin{proof}
Let $\gamma$ be a constant which will be determined in the proof. Notice that $X_iZ_i$ is uniformly distributed in $G_i$ where $Z_i:=\pr_i(Z)$. Therefore,
\be\label{eq:lower-bound-volume-projection}
|\pr_i(H)_{\eta}|=|\pr_i((xH)_{\eta})|=\bbp(X_iZ_i\in \pr_i((xH)_{\eta}))\geq \bbp(XZ\in (xH)_\eta)\geq \eta^{R\gamma}.
\ee
From \eqref{eq:lower-bound-volume-projection}, we deduce that 
\begin{align}
\notag h(\pr_i(H);\eta) \geq & \log(1/|1^{(i)}_\eta|)-R\gamma \log(1/\eta)-2\log C_{1i}\\
\notag \geq & d_{0i}\Big(1-\frac{R\gamma}{d_{0i}} \Big)\log(1/\eta)-3\log C_{1i}\\
\notag \geq & \Big(1-\frac{R\gamma}{d_{0i}} \Big) h(G_i;\eta)-4\log C_{1i}\\
\label{eq:lower-bound-metric-entropy-projections}\geq & \Big(1-\frac{2R\gamma}{d_{0i}} \Big) h(G_i;\eta),
\end{align}
where the last inequality holds as long as 
\begin{equation}\label{eq:etabetavaanerg}
\eta^{R \gamma}  \le (C_{11} C_{12})^{-5}
\end{equation}
 By \cite[Theorem 2.8]{MMSG}, there is $\delta:=\delta(\vare,L,d_{0i})\ll_{L,d_{0i}}\vare$ such that the following holds: if $h(A;\eta)\geq (1-\delta) h(G_i;\eta)$ and 
 \begin{equation}\label{eq:etabetavaaneepsilon}
\eta^\vare \le (2 C_{01}C_{02}C_{11}C_{12})^{- \overline{R} }
\end{equation}
  (where
 $ \overline{R} $ depends polynomially on $L$ and $d_{ 0i}$, $ i=1,2$)
, then $A_\eta A_\eta A_\eta^{-1} A_\eta^{-1}\supseteq 1^{(i)}_{\eta^\vare}$. 
We claim 
\[
\gamma=\frac{1}{2R}\min(\delta(\vare,L,d_{01}),\delta(\vare,L,d_{02}))
\]
satisfies the claim in the lemma so long as $\eta$ is small enough. 

Indeed, the above definition implies
\begin{equation}\label{eq:gamma-epsilon-linear}
\gamma \ll_{L,d_{0i}}\vare. 
\end{equation}
Moreover, in view of \eqref{eq:gamma-epsilon-linear}, there
exists $R'$ depending on $L, d_{0i}$ such that 
if $ \eta^{\vare}  \le (2 C_{01}C_{02}C_{11}C_{12} )^{- R' }$, then 
\eqref{eq:etabetavaanerg} and
\eqref{eq:etabetavaaneepsilon} both hold. Hence, as it was discussed above, the lemma follows by~\eqref{eq:lower-bound-metric-entropy-projections} and \cite[Theorem 2.8]{MMSG}.
\end{proof}

For a symmetric subset $H$ of $G_1\times G_2$ containing $(1^{(1)},1^{(2)})$ and $\eta>0$, set
\be\label{eq:largest-elements-in-factors}
\textstyle
\alpha_1(H;\eta):=\inf\{\alpha\in [0,1]|\h \exists g_1\in G_1, d(g_1,1^{(1)})\geq \eta^{\alpha},\h (g_1,1^{(2)})\in \prod_3 H_\eta\},
\ee
and similarly
\[
\textstyle
\alpha_2(H;\eta):=\inf\{\alpha\in [0,1]|\h \exists g_2\in G_2, d(g_2,1^{(2)})\geq \eta^{\alpha},\h (1^{(1)},g_2)\in \prod_3 H_\eta\};
\]
recall that $d$ always denotes our fixed bi-invariant metric on the underlying compact group. 

\begin{lem}\label{lem:constructing-approx-hom}
Suppose $G_1$ and $G_2$ are two compact groups, $\eta$ is a positive number, and $H\subseteq G_1\times G_2$ is symmetric containing the identity. Then there is  $f:\pr_1(H_\eta)\rightarrow G_2$ which is a $\pr_1(H_\eta)$-partial, $\eta^{\alpha_2(H;\eta)}$-approximate homomorphism; that means 
\begin{enumerate}
\item $f(1^{(1)})=1^{(2)}$,
\item if $g_1,g_1'\in \pr_1(H_\eta)$ and $g_1g_1'\in\pr_1(H_\eta)$, we have $d(f(g_1)f(g_1'),f(g_1g_1'))\leq \eta^{\alpha_2(H;\eta)}$, and 
\item for every $g_1\in \pr_1(H_\eta)$, $d(f(g_1^{-1}),f(g_1)^{-1})\leq \eta^{\alpha_2(H;\eta)}$.	
\end{enumerate}
Furthermore $\pr_2(H_\eta)\subseteq (\im f)_{\eta^{\alpha_2(H;\eta)}}$.
\end{lem}
\begin{proof}
	For every $g_1\in \pr_1(H_\eta)$, choose $f(g_1)\in G_2$ such that $(g_1,f(g_1))\in H_\eta$. As $(1^{(1)},1^{(2)})\in H$, we can and will set $f(1^{(1)})=1^{(2)}$. For every $g_1,g_1'\in \pr_1(H_{\eta})$ with $g_1g_1'\in \pr_1(H_{\eta})$, 
	\[
	\textstyle
	(1^{(1)},f(g_1)f(g_1')f(g_1g_1')^{-1})\in \prod_3 H_\eta
	\quad \text{and}\quad
	(1^{(1)},f(g_1)f(g_1^{-1}))\in \prod_2 H_\eta\subseteq \prod_3 H_\eta.
	\]
	Hence $d(f(g_1)f(g_1'),f(g_1g_1'))\leq \eta^{\alpha_2(H;\eta)}$ and $d(f(g_1^{-1}),f(g_1)^{-1})\leq \eta^{\alpha_2(H;\eta)}$.  
	
	For every $g_2\in \pr_2(H)_\eta$, there is $g_1\in \pr_1(H_\eta)$ such that $(g_1,g_2)\in H_\eta$. Hence $(1^{(1)},g_2f(g_1)^{-1})$ is in $\prod_2 H_{\eta}$. Therefore $d(g_2,f(g_1))\leq \eta^{\alpha_2(H;\eta)}$. This completes the proof.
\end{proof}
By Lemma~\ref{lem:constructing-approx-hom}, we get a good approximate homomorphism if $\alpha_i(H;\eta)$ is large for some subset $H$ of $G_1\times G_2$. To get such a bound, inspired by Proposition~\ref{lem:coupling-BG}, we consider an $\eta^{R\gamma}$-approximate subgroup of $G_1\times G_2$ with an upper bound for its metric entropy, and study $\alpha_i(\prod_k H;\eta)$ for a fixed positive integer $k$ that will be determined later and depends on the dimensions of $G_i$'s. 

Let us recall the dimension condition of $G_i$'s. For every positive number $\eta$ we have 
\be\label{eq:DC}
C_{1i}^{-1} \eta^{d_{0i}} \leq |1^{(i)}_\eta|\leq C_{1i} \eta^{d_{0i}}
\ee
where $C_{1i}$  and $d_{0i}$ are positive numbers.  

We also recall the following two facts from~\cite[\S7]{MMSG}.
By \cite[Lemma 7.1]{MMSG}, for every non-empty subset $A$ of $G=G_1\times G_2$, we have 
\be\label{eq:metric-entropy}
\Big | h(A;\eta)-\log\Big(\frac{|A_{\eta}|}{|1_\eta|} \Big)\Big|\leq \log(\ref{c: bbbb}),
\ee
moreover, the same is true for the subsets of $G_i$'s. By \cite[Corollary 7.2]{MMSG}, for every non-empty subset $A$ of $G=G_1\times G_2$ and positive numbers $\eta$ and $a$, we have 
\be\label{eq:scale-metric-entropy}
|h(A;\eta)-h(A;a\eta)|\leq\log(\ref{c: bbbbbbb}).
\ee
where $\constc\label{c: bbbb}$ and $\constc\label{c: bbbbbbb}$ are constants as in \eqref{eq:const-kazaa}.

The following is an upgraded version of Lemma~\ref{lem:constructing-approx-hom}, and will be used the sequel.  

\begin{lem}\label{lem:constructing-approx-hom-nbhd}
Let $G_1$ and $G_2$ be two compact groups, $R, \gamma, \eta>0$, and let $H\subseteq G_1\times G_2$ be an $\eta^{-R\gamma}$-approximate subgroup. Assume further that for $i=1,2$ we have 
\be\label{eq: projection of 4fold prod is large}
1^{(i)}_{\eta^{C_2 \gamma}}\subseteq \pr_i(\textstyle\prod_4 H_\eta).
\ee
Then there exists an $\eta^{C_2\gamma}$-partial $\eta^{\alpha_2}$-approximate homomorphism 
$f:1^{(1)}_{\eta^{C_2\gamma}}\rightarrow G_2$ satisfying that 
\[
1^{(2)}_{\eta^{C_2'\gamma}}\subseteq (\im f)_{10\eta^{\alpha_2}}
\]
where $\alpha_2:=\alpha_2(\prod_8 H;\eta)$ is as in \eqref{eq:largest-elements-in-factors} and $C_2, C_2'$ depend only on $L,d_{01}$, 
and $d_{02}$. 
\end{lem}

\begin{proof}
Apply Lemma~\ref{lem:constructing-approx-hom} with $\prod_8 H_\eta$ (instead of $H$), and let $f$ be thus obtained. We may assume, without loss of generality, that $(g, f(g))\in H':=\prod_4 H_\eta$ for all $g\in \pr_1(H')$. Then since $\alpha_2(H'\cdot H'; \eta)\leq \alpha_2(H';\eta)$, Lemma~\ref{lem:constructing-approx-hom}, applied with $H'$, implies that 
\be\label{eq: f retricted to 4fold prod}
\pr_2(H')\subseteq \bigl(f\bigl(\pr_1(H'_\eta)\bigr)\bigr)_{\eta^{\alpha_2}}.
\ee

In view of~\eqref{eq: projection of 4fold prod is large}, we may restrict $f$ to $1^{(1)}_{\eta^{C_2\gamma}}$ and obtain a 
$\eta^{C_2\gamma}$-partial $\eta^{\alpha_2}$-approximate homomorphism $f:1^{(1)}_{\eta^{C_2\gamma}}\rightarrow G_2$. We now show that $f$ also satisfies the last claim in the lemma.  

To see this claim, let $C\geq1$ be a constant which will be explicated later. 
Let us put
\[
E_1=\bigl(f\bigl(1^{(1)}_{0.1\eta^{C_2\gamma}}\bigr)\bigr)_{\eta^{\alpha_2}} \quad \text{and} \quad E_2=\pr_2(H').
\]
Assume first that 
\be\label{eq: metric entropy of E1}
h(E_1; \eta^{\alpha_2})\geq  (1-C\gamma/\alpha_2)h(G_2;\eta^{\alpha_2}).
\ee
We want to apply~\cite[Theorem 2.8]{MMSG} with $\eta^{\alpha_2}$ and under the assumption~\eqref{eq: metric entropy of E1}. By ~\cite[Theorem 2.8]{MMSG}, there exists $C'=O_{L,d_{02}}(C)$ so that $\vare=C'\gamma/\alpha_2$ and $\delta=C\gamma/\alpha_2$ satisfy the conditions in that theorem. Thus~\eqref{eq: metric entropy of E1} and~\cite[Theorem 2.8]{MMSG} imply that 
\[
(E_1)_{\eta^{\alpha_2}}(E_1)_{\eta^{\alpha_2}}(E_1)_{\eta^{\alpha_2}}^{-1}(E_1)_{\eta^{\alpha_2}}^{-1}\supseteq 1^{(2)}_{\eta^{C'\gamma}}.
\]
This and the fact that $\prod_4 1^{(1)}_{0.1\eta^{C_2\gamma}}\subseteq 1^{(1)}_{\eta^{C_2\gamma}}$ imply that 
\be\label{eq: case one with C'}
1^{(2)}_{\eta^{C'\gamma}}\subset \bigl(f(1^{(1)}_{\eta^{C_2\gamma}}))_{10\eta^{\alpha_2}}.
\ee

Hence, we assume that~\eqref{eq: metric entropy of E1} fails. 
Since $1^{(2)}_{\eta^{C_2\gamma}}\subseteq E_2$, see~\eqref{eq: projection of 4fold prod is large}, we have   
\be\label{eq: metric entropy of E2}
h(E_2; \eta^{\alpha_2})\geq (1-C_2\gamma/\alpha_2) h(G_2;\eta^{\alpha_2}).
\ee
We now cover $\pr_1(H'_\eta)$ with $\leq (20)^{d_1}C_{11}^2 \eta^{-C_2\gamma d_1}$ many sets of the form $g1^{(1)}_{0.1\eta^{C_2\gamma}}$ where $g\in\pr_1(H'_\eta)$. Since 
$g1^{(1)}_{0.1\eta^{C_2\gamma}}\subseteq \pr_1(H'_\eta\cdot H'_\eta)$, we have that $f(gx)$ is defined for all $x\in 1^{(1)}_{0.1\eta^{C_2\gamma}}$. Recall from~\eqref{eq: f retricted to 4fold prod} that 
$E_2\subseteq \bigl(f(\pr_1(H'_\eta))\bigr)_{\eta^{\alpha_2}}$. Therefore, 
\be\label{eq: metric entropy E2 upp}
h(E_2; 2\eta^{\alpha_2})\leq \max_{g\in {\rm pr}_1(H'_{\eta})}\bigl\{h\bigl(f(g1^{(1)}_{0.1\eta^{C_2\gamma}}); 2\eta^{\alpha_2}\bigr) \bigr\}+ C_2\gamma d_1\log(1/\eta) + \log\bigl((20)^{d_1}C_{11}^2\bigr).
\ee

Note on the other hands that 
\[
d(f(gx), f(g)f(x))\leq \eta^{\alpha_2},\qquad\text{for every $x\in 1^{(1)}_{0.1\eta^{C_2\gamma}}$}.
\]
Thus the failure of~\eqref{eq: metric entropy of E1} implies that 
\[
h\bigl(f(g1^{(1)}_{0.1\eta^{C_2\gamma}}); 2\eta^{\alpha_2}\bigr)\leq (1-C\gamma/\alpha_2)h(G_2;\eta^{\alpha_2}), \quad\text{for all $g\in\pr_1(H'_\eta)$}.
\]
This and~\eqref{eq: metric entropy E2 upp} imply
\be\label{eq: metric entropy E2 upp'}
h(E_2; \eta^{\alpha_2})\leq (\alpha_2-C\gamma) d_2\log(1/\eta)+ C_2\gamma d_1\log(1/\eta) + \log(\ref{c: approx hom rest 1});
\ee
where we used
\[
h(E_2; \eta^{\alpha_2})=h(E_2; 2\eta^{\alpha_2})+ \log(\ref{c: approx hom rest 2})\quad\text{and}\quad h(G_2;\eta^{\alpha_2})=\alpha_2d_2\log(1/\eta)+\log(\ref{c: approx hom rest 3}),
\]  
for constants $\constc \label{c: approx hom rest 1}$, $\constc \label{c: approx hom rest 2}$, and $\constc \label{c: approx hom rest 3}$ as in \eqref{eq:const-kazaa}. 
Thus,~\eqref{eq: metric entropy E2 upp'} contradicts~\eqref{eq: metric entropy of E2} so long as $\eta$ is small enough to account for the additive constants and $C\geq 3C_2\max\{d_1/d_2,1\}$.
This and~\eqref{eq: case one with C'} finish the proof.   
\end{proof}

The following two lemmas concern $k$ fold product of approximate subgroups.  

\begin{lem}\label{lem:lower-bound-metric-entropy}
Suppose $G_1$ and $G_2$ are two compact groups, $R, \gamma, \eta>0$, and $H\subseteq G_1\times G_2$ is an $\eta^{-R\gamma}$-approximate subgroup. For a positive integer $k$, let
\[
\textstyle
H^{(1)}_{k}:=\pr_1\Big((G_1\times \{1^{(2)}\})\cap \prod_k H\Big).
\] 
Then 
\[
\textstyle
h(H_{k}^{(1)};\eta)+h(\pr_2(\prod_k H);\eta)\leq h(H;\eta)+2kR\gamma \log(1/\eta)+\log(\ref{c: metric entropy lemma}),
\]
where $\constc\label{c: metric entropy lemma}$  multiplicatively depends on $\ref{c: bbbb}$ and $\ref{c: bbbbbbb}$.
\end{lem}
 \begin{proof}
 Notice that $|(\prod_{2k} H)_{2\eta}|\geq |(H_k^{(1)})_\eta| |\pr_2(\prod_k H)_\eta|$. Hence
 \[
 \log \Big( \frac{|(\prod_{2k} H)_{2\eta}|}{|1_\eta|} \Big)
 \geq 
  \log \Big( \frac{|(H_k^{(1)})_\eta| }{|1^{(1)}_\eta|} \Big)+ 
   \log \Big( \frac{|\pr_2(\prod_k H)_\eta|}{|1^{(2)}_\eta|} \Big).
 \]
 Therefore, by \eqref{eq:metric-entropy} and \eqref{eq:scale-metric-entropy}, we obtain 
 \be\label{eq:metric-entropy-kernel-image}
 \textstyle
 h(\prod_{2k} H;\eta)+\log(\ref{c:metric-entropy-kernel-image})\geq h(H^{(1)}_k;\eta)+h(\pr_2(\prod_k H);\eta),
 \ee
 where $\constc\label{c:metric-entropy-kernel-image}$ multiplicatively depends on $\ref{c: bbbb}$ and $\ref{c: bbbbbbb}$. 
 
 Since $H$ is an $\eta^{-R\gamma}$-approximate subgroup, there is a symmetric set $A$ of cardinality at most $\eta^{-R\gamma}$ such that $HH\subseteq HA$. Therefore, $\prod_{2k} H$ is a subset of $H\prod_{2k} A$, which implies that 
 $
 |(\prod_{2k} H)_\eta|\leq \eta^{-2kR\gamma} |H_{\eta}|.
 $
 Hence, there exists $\constc\label{c:metric-entropy-approx-subgroup}$, so that   
 \be\label{eq:metric-entropy-approx-subgroup}
 \textstyle
 h(\prod_{2k} H;\eta)\leq h(H;\eta)+2kR\gamma \log(1/\eta) +\log(\ref{c:metric-entropy-approx-subgroup})
 \ee
 By \eqref{eq:metric-entropy-kernel-image} and \eqref{eq:metric-entropy-approx-subgroup}, the claim follows.
 \end{proof}

\begin{lem}\label{cor:upper-bound-metric-entropy-intersection-with-factors}
Suppose $G_1$ and $G_2$ are $L$-locally random with coefficients $C_{01}$ and $C_{02}$, respectively. Suppose $G_i$ satisfies the ${\rm DC}(d_{0i},C_{1i})$. 
Suppose $k\geq 4$ is an integer and $\bar\delta>0$. 
Let $\eta$ and $\gamma\ll_{d_{0i},L,k}\bar\delta$ be positive numbers, and $\eta^\gamma\leq \ref{c: bound-scale-factors}^{-1}$ where $\constc \label{c: bound-scale-factors}$ is a constant as in \eqref{eq:const-kazaa}. Suppose $X$ and $X'$ are independent and identically distributed random variables with values in $G$ satisfying the following properties:   
\begin{itemize}
\item $H_2(XX';\eta)< H_2(X;\eta)+\gamma \log(1/\eta)$, and 
\item $H_2(X;\eta)< (d_{01}+d_{02}-\bar\delta)\log(1/\eta)$.
\end{itemize}
 Then there is an $\eta^{-R\gamma}$-approximate subgroup $H$ of $G_1\times G_2$, where $R$ is the universal constant given in Lemma~\ref{lem:coupling-BG}, satisfying both of the following properties  
 \[
 1^{(i)}_{\eta^{C_2 \gamma}}\subseteq \pr_i(\textstyle\prod_k H_\eta),
 \quad\text{and}
 \quad
  h(H_k^{(1)};\eta) \leq \Big(1-\frac{\bar\delta}{2d_{01}}\Big)h(G_1;\eta)
 \]
where $C_2$ depends only on $L,d_{01},$ and $d_{02}$ and $H_k^{(1)}$ is defined as in Lemma~\ref{lem:lower-bound-metric-entropy}. 
\end{lem}

\begin{proof}
We let 
$0<\eta<(10C_{01}+10C_{02}+C_{11}+C_{12})^{-1/\bar\delta}$ be a small constant which will be determined later.  
By Lemma~\ref{lem:coupling-BG} and Lemma~\ref{lem:coupling-plus-almost-equidistribution-imlpies-large-open-projections}, there is an $\eta^{-R\gamma}$-approximate subgroup $H$ such that 
\[
\textstyle
1^{(i)}_{\eta^{C_2\gamma}}\subseteq \pr_i(\prod_4 H_\eta)\subseteq \pr_i(\prod_k H_\eta),
\]
where $C_2$ depends only on $d_{0i}$'s and $L$, and in the second containment we used $k\geq 4$.

To see the second claim, we have  
\begin{align}
\textstyle
\notag
h(\pr_2 (\prod_k H);\eta)\geq & 
\textstyle
h(\pr_2(\prod_4  H);4\eta)-\log (\ref{c: bbbbbbb}) 
 \geq  
 \log\Big(\frac{|\pr_2(\prod_4 H_{\eta})|}{|1^{(2)}_{\eta}|} \Big)-\log(\ref{c: bbbbbbb}\ref{c: bbbb}) \\
 \label{eq:projection-metric-entropy-lower-bound}
\geq & C_2 d_{02} \gamma \log(\eta)+d_{02}\log(1/\eta)-\log(\ref{c: lower-bound-metric-entropy-second-factor}),
\end{align}
where $\constc \label{c: lower-bound-metric-entropy-second-factor}$ is a constant as in \eqref{eq:const-kazaa}. By Lemma~\ref{lem:lower-bound-metric-entropy} and \eqref{eq:projection-metric-entropy-lower-bound}, we obtain that the following holds
\begin{align}
	\notag h(H^{(1)}_k;\eta) -C_2d_{02}\gamma \log(1/\eta)+d_{02} \log(1/\eta)
\leq
&
h(H;\eta) +2kR\gamma \log(1/\eta)+\log(\ref{c: metric entropy lemma}\ref{c: lower-bound-metric-entropy-second-factor})
\\
\notag
\leq  &
H_2(X;\eta)+(2k+1) R\gamma \log(1/\eta)+\log(\ref{c: metric entropy lemma}\ref{c: lower-bound-metric-entropy-second-factor})
\\
\notag 
\leq& 
\Big(d_{01}+d_{02}-\bar\delta\Big)\log(1/\eta) +(2k+1) R\gamma \log(1/\eta)\\ &
\label{eq:upper-bound-metric-entropy-intersection-with-factors}
 +\log(\ref{c: metric entropy lemma}\ref{c: lower-bound-metric-entropy-second-factor}),
\end{align}
where the second inequality follows from Lemma~\ref{lem:coupling-BG}. By \eqref{eq:upper-bound-metric-entropy-intersection-with-factors}, we obtain 
\be\label{eq:upper-metric-factor-aux}
h(H^{(1)}_k;\eta) \leq \Big(d_{01}-\bar\delta\Big)\log(1/\eta)+((2k+1)R+C_2d_{02})\gamma \log(1/\eta)+
\log(\ref{c: metric entropy lemma}\ref{c: lower-bound-metric-entropy-second-factor}).
\ee
Therefore, we can choose $\ref{c: bound-scale-factors}$ so that for $\eta^{\gamma}\leq \ref{c: bound-scale-factors}^{-1}$ and $\gamma\ll_{k,d_{0i},L} \bar\delta$, we have 
\[
h(H_k^{(1)};\eta) \leq \Big(1-\frac{\bar\delta}{2d_{01}} \Big)h(G_1;\eta);
\]
as we claimed. 
\end{proof}

\subsection{Proof of Propositions~\ref{prop:desired-entropy} and~\ref{prop:contraction-coupling} modulo a bounded generation result}

In this section we use the following bounded generation result, which is of independent interest, to complete the proofs of Propositions~\ref{prop:desired-entropy} and~\ref{prop:contraction-coupling}.  

\begin{proposition}\label{prop:conjugation-by-large-ball-multiplication-large-image}
Suppose $F$ is either $\bbr$ or $\bbq_p$, $\bbg\subseteq (\GL_{n_0})_F$ is a connected $F$-almost simple subgroup, and $G$ is a compact open subgroup of $\bbg(F)$. When $F=\bbr$, we assume that $G\subseteq O_{n_0}(\bbr)$, and when $F=\bbq_p$, we assume that $G\subseteq \GL_{n_0}(\bbz_p)$. In either case, we take the metric on $G$ that is induced by the operator norm on $\M_{n_0}(F)$. 
Let $p_0=2$ when $F=\bbr$ and $p_0=p$ when $F=\bbq_p$. 
Then for every $0<\rho\leq p_0^{-2}$, there are $g_{1},\ldots,g_{d^2}\in 1_{\rho}$ where $d:=\dim \bbg$ and positive numbers $C:=C(G)$ and $c:=c(G)$ 
such that the following holds. For $h\in 1_{1/4}$ and every $0<r\leq c \|h-I\| \rho^C$, we have
\[
\{(g_{1} [h,a_1] g_{1}^{-1})\cdots (g_{d^2}[h,a_{d^2}] g_{d^2}^{-1})|\h a_i\in 1_{r}\}\supseteq 1_{cr \rho^C \|h-I\|},
\]
where $[h,a_i]=ha_ih^{-1}a_i^{-1}$. Moreover, if $\bbg=\wt{\bbg}\otimes_{\bbq}\bbq_p$ where $\wt{\bbg}$ is an absolutely almost simple $\bbq$-group, then the constants $C$ and $c$ depend only on $\wt{\bbg}$. 
\end{proposition}

We postpone the proof of Proposition~\ref{prop:conjugation-by-large-ball-multiplication-large-image} to \S\ref{sec: proof bdd generation}. This proposition is used in the proof of the following corollary, which in turn will play a crucial role in the proofs of Propositions~\ref{prop:desired-entropy} and~\ref{prop:contraction-coupling}.

\begin{corollary}\label{cor:almost-kernels-are-small}
Let $G_1$, $G_2$, $X$, and $X'$ be as in Lemma~\ref{cor:upper-bound-metric-entropy-intersection-with-factors}. Let $k=50d_{01}^2$ and $\bar\delta>0$.
Suppose $\eta$ and $\gamma$ are positive numbers such that $\gamma\ll_{C(G_i),d_{0i},L}\bar\delta $, where $C(G_i)$'s are given in Proposition~\ref{prop:conjugation-by-large-ball-multiplication-large-image}, and $\eta^\gamma\leq \ref{c: bound-scale-factors'}^{-1}$ where $\constc \label{c: bound-scale-factors'}$ is a constant as in \eqref{eq:const-kazaa}. Let $H$ be the approximate subgroup given in Lemma~\ref{cor:upper-bound-metric-entropy-intersection-with-factors}, applied with $k=50d_{01}^2$ and these $\bar\delta$, $\eta$, 
and $\gamma$.
Let $\alpha_i:=\alpha_i(\prod_8 H;\eta)$ be as in \eqref{eq:largest-elements-in-factors}. Then 
\be\label{eq: relation between bar delta and alphai}
\alpha_i\geq \frac{\bar\delta}{10d_{0i}},
\ee
and there is $\eta^{O_{d_{0i},L}(\bar{\delta})}$-approximate homomorphism 
\[
f:1^{(1)}_{\eta^{C_2\gamma}}\rightarrow G_2
\]
such that $1^{(2)}_{\eta^{C_2'\gamma}}\subseteq (\im f)_{\eta^{O_{d_{0i},L}(\bar\delta)}}$, where $C_2$ and $C_2'$ depend only on $L$, $d_{01}$ and $d_{02}$.
\end{corollary}
\begin{proof}
	We first prove~\eqref{eq: relation between bar delta and alphai}; in view of the symmetry, we show this for $i=1$. 
	Recall that 
	By the definition of $\alpha_1$, see~\eqref{eq:largest-elements-in-factors}, 
	there exists $h\in G_1$ such that $\|h-I\|\geq \eta^{2\alpha_1}$ and $(h,1)\in \prod_{24} H$.
		Since $1^{(1)}_{\eta^{C_2\gamma}}\subseteq \pr_1(\prod_4 H_{\eta})$, by Proposition~\ref{prop:conjugation-by-large-ball-multiplication-large-image} (applied with $\rho:=\eta^{C_2\gamma}$), we deduce that
	\[
	1^{(1)}_{c\eta^{2CC_2\gamma+4\alpha_1}}\subseteq (H^{(1)}_{50d_{01}^2})_{50d_{01}^2\eta},
	\]
	where  $C:=C(G_1)$ is as in Proposition~\ref{prop:conjugation-by-large-ball-multiplication-large-image} and $c$ is a multiple of $c(G_1)$ given in the same statement. Hence, we obtain
	\begin{align}
	\notag 
	h(H^{(1)}_{50d_{01}^2};\eta)\geq
	&
		 \log\biggl(\frac{|1^{(1)}_{c\eta^{2CC_2\gamma+4\alpha_1}}|}{|1^{(1)}_\eta|} \biggr)-\log(\ref{c: bbbb})	
		\\
		\notag
	\geq &
	d_{01}\Big(1-(2CC_2\gamma+4\alpha_1)\Big) \log(1/\eta)-\log(\ref{c: bbbbbbb}\ref{c: bbbb})
	\\
	\label{eq:lower-bound-metric-entropy-almost-kernel}
	\geq &
	\Big(1-(3CC_2\gamma+4\alpha_1)\Big) h(G_1;\eta);
	\end{align}
notice that we can drop $\log(\ref{c: bbbbbbb}\ref{c: bbbb})$ as 
$\eta^\gamma\leq \ref{c: bound-scale-factors'}^{-1}$ can be chosen small enough. By Lemma~\ref{cor:upper-bound-metric-entropy-intersection-with-factors}, applied with $k=50d_{01}^2$, and \eqref{eq:lower-bound-metric-entropy-almost-kernel},  we obtain the following inequality:
\[
3CC_2\gamma+4\alpha_1 \geq \frac{\bar\delta}{2d_{01}}.
\]
Therefore, for $\gamma\ll_{C,d_{01},L,a} \bar\delta$, we obtain that 
\[
\alpha_1 \geq \frac{\bar\delta}{10 d_{01}}.
\]
Recalling $1^{(1)}_{\eta^{C_2\gamma}}\subseteq \pr_i(\prod_4 H_{\eta})$ and~\eqref{eq: relation between bar delta and alphai} for $\alpha_2$, the remaining assertions follow from Lemma~\ref{lem:constructing-approx-hom-nbhd}.
\end{proof}

\begin{proof}[Proof of Proposition~\ref{prop:desired-entropy}]
Let $G=G_1\times G_2$, then $G$ is $L$-locally random with coefficient $C_0:=C_{01}+C_{02}$, see~\cite[Lemma 5.2]{MMSG}. It also satisfies ${\rm DC}(d_0, C_1)$ where $d_0:=d_{01}+d_{02}$ and $C_1:=C_{11}C_{12}$.  
Suppose 
\[
0<\eta<(10C_0+C_1)^{-1/\bar\delta}. 
\]

Let $X$ be a random variable whose probability law is $\mu$. 
Let $X_i$ be a $2^i$-random walk with respect to $\mu$. 

\medskip

\begin{claim}\label{claim:entropy-dicatomy}
For sufficiently small $\gamma$ (which will be determined later and will depend linearly on $\bar\delta$) and for every non-negative integer $i$ at least one of the following holds:
\begin{subequations}
\begin{align}
    \label{eq: gaining entropy}&H_2(X_{i+1};\eta)\geq H_2(X_i;\eta)+\gamma \log(1/\eta),\\
    \label{eq: final entropy}&H_2(X_i;\eta)\geq (d_0-\bar\delta) \log(1/\eta).
\end{align}
\end{subequations}
\end{claim} 

\begin{proof}[Proof of the Claim]
Let us assume that~\eqref{eq: gaining entropy} and~\eqref{eq: final entropy} fail for some $i$. Then by
Corollary~\ref{cor:almost-kernels-are-small},
there exists $\hat c$ depending only on $C(G_i)$ (see Proposition~\ref{prop:conjugation-by-large-ball-multiplication-large-image}), $d_{0i}$, and $L$ so that if $0<\gamma<\hat c\bar\delta$ and $\eta^\gamma\leq \ref{c: bound-scale-factors'}^{-1}$, where $\ref{c: bound-scale-factors'}$ is a constant as in \eqref{eq:const-kazaa},
then there is an $\eta^{\beta}$-approximate homomorphism 
\[
f:1^{(1)}_{\eta^{C_2\gamma}}\rightarrow G_2
\]
such that $1^{(2)}_{\eta^{C_2'\gamma}}\subseteq (\im f)_{\eta^{\beta}}$, where $C_2$ and $C'_2$ depend only on $L$, $d_{01}$ and $d_{02}$, and $\beta=O_{d_{0i},L}(\bar\delta)$. 

Let $m$ be as in Theorem~\ref{thm:approximate-hom} applied with $G_1$ and $G_2$. 
For small enough $\gamma$, we have  
$\beta/(C_2\gamma)>m$.
Since $G_1$ and $G_2$ are not locally isomorphic, existence of $f$ contradicts Theorem~\ref{thm:approximate-hom} applied with $G_1$, $G_2$ and $\rho=\eta^{C_2\gamma}$ so long as $\eta^{\gamma}$ is small enough. The claim follows. 
\end{proof}

Returning to the proof of the proposition, 
first note that~\eqref{eq: gaining entropy} can hold at most $i_{\max}:=\lceil d_0/\gamma\rceil$-many times. Therefore, there exists some $i_0\leq i_{\max}$ so that~\eqref{eq: final entropy} holds. The proof is complete. 
\end{proof}

\begin{proof}[Proof of Proposition~\ref{prop:contraction-coupling}]

Fix some integer $a\geq 4Ld_0$, and let $\bar\delta=\frac{1}{8La\log_{\eta}\eta^{a^2}}=\frac{1}{8La^3}$. Let $0<\eta<\ref{c: desired entropy}^{-1/\bar\delta}$, where $\ref{c: desired entropy}$ is as in Proposition~\ref{prop:desired-entropy}. 

Recall that $f\in L^2(G_1\times G_2)$ lives at scale $\eta$ if both of the following properties are satisfied   
\[
\|f_{\eta^{a^2}}-f\|_2\leq \eta^{a/2}\|f\|_2\quad\text{and}\quad \|f_{\eta^{1/a}}\|_2\leq \eta^{1/(2a)}\|f\|_2.
\]
Apply Proposition~\ref{prop:desired-entropy} with 
$\eta^{a^2}$ and $\bar\delta$. In view of Proposition~\ref{prop:desired-entropy}, conditions of Lemma~\ref{lem:lower-bound-entropy-gives-contraction} are satisfied for $G$, $\nu=\mu^{(2^{m})}$ and $\rho=\eta^{a^2}$. Hence we have 
\[
\|\mu^{(2^{m})}_{\rho}\ast f\|_2\leq \eta^{1/(8La)} \|f\|_2.
\]
Now since $\|f_\rho-f\|_2\leq \eta^{a/2}\|f\|_2$, we conclude that
\[
\|\mu^{(2^{m})}\ast f\|_2\leq \|\mu^{(2^{m})}\ast f_\rho\|_2+\eta^{a/2}\|f\|_2=\|\mu^{(2^{m})}_{\rho}\ast f\|_2+\eta^{a/2}\|f\|_2\leq \eta^{1/(16La)} \|f\|_2
\]
This implies the proposition with $c=1/(16La)$ and $m_0=m$. 
\end{proof}

\subsection{Proof of Proposition~\ref{prop:conjugation-by-large-ball-multiplication-large-image}} \label{sec: proof bdd generation}

In this section we prove Proposition~\ref{prop:conjugation-by-large-ball-multiplication-large-image}. The proof is carried out in several steps, and among other things it relies on certain quantitative inverse function theorems that are proved in the appendix. We start with a lemma which is analogous to \cite[Lemma 40]{SG-super-approx-II} for real numbers.
\begin{lem}\label{lem:polytops-g-module}
In the setting of Proposition~\ref{prop:conjugation-by-large-ball-multiplication-large-image}, suppose  $\{\Ad(g_1),\ldots,\Ad(g_m)\}$ is a basis of the $\bbr$-subalgebra $\bbr[\Ad(G)]$ of $\End(\gfr)$, where $\gfr:=\gfr(\bbr)$ is the Lie algebra of $G$. Then there is a positive number $r_0$ depending on $g_i$'s such that for every unit element $x$ of $\gfr$, 
\[
M(x):=\Big\{\sum_{i=1}^m c_i \Ad(g_i)(x)|\h c_i\in [-1,1]\Big\} \supseteq 0_{r_0},
\]  
where $0_{r_0}$ is the ball of radius $r_0$ centered at $0$ in $\gfr$.
\end{lem}
\begin{proof}
	Since $\bbg$ is $\bbr$-simple, $\gfr$ is a simple $G$-module. Hence for every unit element $x$ of $\gfr$, $M(x)$ contains a neighborhood of $0$. Let $r(x)$ be the largest positive number such that $0_{r(x)}\subseteq M(x)$. Suppose to the contrary that there is a sequence $\{x_i\}_{i=1}^{\infty}$ of unit elements of $\gfr$ such that $\lim_{i\rightarrow \infty} r(x_i)=0$. By the compactness of the sphere of radius 1 in $\gfr$, after passing to a subsequence we can assume that $\{x_i\}_{i=1}^{\infty}$ converges to $x$, a unit element of $\gfr$. For every $y\in 0_{r(x)}$, there are $c_i\in[-1,1]$ such that $\sum_{i=1}^m \Ad(g_i)(x)=y$. For every $\vare>0$, if $n\gg_{\vare} 1$, then $\|\Ad(g_i)(x_n)-\Ad(g_i)(x)\|\leq \vare$. Therefore, 
	\[
	\|y-\sum_{i=1}^m \Ad(g_i)(x_n)\|\leq \sum_{i=1}^{m} |c_i| \|\Ad(g_i)(x)-\Ad(g_i)(x_n)\|\leq m\vare.
	\]
	Notice that $M(x_n)$ is a convex set which intersects every $m\vare$-neighborhood of points of $0_{r(x)}$. Therefore for $n\gg_{m,r(x)} 1$, we have $r(x_n)\geq r(x)/2$. This contradicts $\lim_{i\rightarrow \infty} r(x_i)=0$, and the claim follows.
\end{proof}
To formulate the next lemma, we start by recalling that for $g\in  {\rm O}_{n_0}(\bbr)$ or $g\in \GL_{n_0}(\bbz_p)$, if $\|g-I\|<1$ (if $p=2$, we assume $\|g-I\|<1/2)$), then for every $|t|\leq 1$ we can define 
\[
g^t:=\exp(t\log(g)).
\] 
Clearly $t\mapsto g^t$ is an analytic function, and one can see that 
\be\label{eq:exp-close-to-0-almost-linear}
|t|\ll_g \|g^t-I\|\ll_g |t|. 
\ee
\begin{lem}\label{lem:polytop-group-ring-adjoint}
	In the setting of Proposition~\ref{prop:conjugation-by-large-ball-multiplication-large-image}, suppose  $\{\Ad(g_1),\ldots,\Ad(g_m)\}$ is a basis of the $\bbr$-subalgebra $\bbr[\Ad(G)]$ of $\End(\gfr)$, where $\gfr:=\gfr(\bbr)$ is the Lie algebra of $G$. Suppose $\|g_i-I\|<1$ for every $i$. Then there is a positive integer $C:=C(g_1,\ldots,g_m)$ and a positive number $c:=c(g_1,\ldots,g_m)$ such that for every $0<t\leq c$ we have 
	\[
	\wt{M}_t:=\Big\{\sum_{i=1}^m c_i\Ad(g_i^t)|\h c_i\in [-1,1]\Big\}\supseteq 0_{t^C},
	\]
	where $0_{t^C}$ is the ball of radius $t^C$ centered at $0$ in $\bbr[\Ad(G)]$.
\end{lem}
\begin{proof}
We view $\Ad(g_i^t)$'s as $d^2\times 1$ column vectors, where $d:=\dim G$, and let  $A(t)$ be the $d^2\times m$ matrix that have $\Ad(g_i^t)$ in its $i$-th column. Consider $f(t):=\det(A(t)^T A(t))$ where $A(t)^T$ is the transpose of $A(t)$. Then $f$ is an analytic function, $f(1)\neq 0$, and $f(0)=0$. Since $f$ is an analytic function and non-zero, $f^{(\overline{C})}(0)\neq 0$ for some positive integer $\overline{C}$.  As $f$ is an analytic function,  for $0<t\leq \bar c$ we have $\frac{f^{(\overline{C})}(0)}{2\overline{C}!}t^{\overline{C}}\leq f(t)$. Since $\|A(t)\|_{\op}=\sqrt d$, $\|(A(t)^TA(t))^{-1}\|_{\op}\ll_{\{g_i\}} t^{-\overline{C}}$. Therefore, for every $y$ in $\bbr[\Ad(G)]$, we have 
\[
\Bigl\|\Bigl(A(t)^TA(t)\Bigr)^{-1}(A(t)^T)y\Bigr\|\ll t^{-\overline{C}} \|y\|
\quad \text{ and } \quad 
A(t)\Big((A(t)^TA(t))^{-1}(A(t)^T)y)\Big)=y.
\]
This implies the claim with $C=\overline{C}/2$ if we assume $0<t \leq c$ and $c\leq \bar c$ is sufficiently small to account for the implied multiplicative constant above.  
\end{proof}

\begin{lem}\label{lem:adjoint-action-small-ball-polytop}
	In the setting of Proposition~\ref{prop:conjugation-by-large-ball-multiplication-large-image}, for every $0<\rho\le p_0^{-2}$ (where $p_0= 2$ when $ {F}=\mathbb{R}$ and $ p_0=p$ when $F= \mathbb{Q}_p$) there are $g_{1,\rho},\ldots,g_{d^2,\rho}\in 1_{\rho}$ where $d:=\dim \bbg$ and positive number $C:=C(G)$ such that for every non-zero element $x\in \gfr(F)$ we have 
	\[
	\Big\{\sum_{i=1}^{d^2} c_i\Ad(g_{i,\rho})(x)|\h c_i\in F, |c_i|\leq 1 \Big\}\supseteq 0_{\rho^C\|x\|},
	\]
	where $0_{\rho^C\|x\|}$ is the ball of radius $\rho^C\|x\|$ centered at $0$ in $\gfr$. Moreover, if $F=\bbq_p$ and $\bbg=\wt{\bbg}\otimes_{\bbq}\bbq_p$ where $\wt{\bbg}$ is an absolutely almost simple $\bbq$-group, then the constant $C$ depends only on $\wt{\bbg}$.
\end{lem}
\begin{proof}
	We start with the $p$-adic case.  Let $l:=\lceil \log_p(1/\rho) \rceil$; then $G[p^l]:=1_\rho$ is the kernel of the residue map modulo $p^l$. Choose $\{g_{1,\rho},\ldots,g_{d^2,\rho}\}\subseteq G[p^l]$ such that the $\bbz_p$-linear span of $\Ad(g_{i,\rho})$'s is the $\bbz_p$-algebra $\bbz_p[\Ad(G[p^l])]$. Then by \cite[Proposition 44']{SG-super-approx-II}, there is a positive number $\overline{C}$ which depends on $\bbg$ (and depends only on $\wt{\bbg}$ if $\bbg=\wt{\bbg}\otimes_{\bbq}\bbq_p$) such that 
	\be\label{eq:z-p-span-ball}
	\sum_{i=1}^{d^2} \bbz_p\Ad(g_{i,\rho})\supseteq p^{\overline{C}l} \bbz_p[\Ad(G[1])]
	\ee 
	where $G[1]$ is the ball of radius 1 in $\bbg(\bbq_p)$. By \eqref{eq:z-p-span-ball}, for every $x\in \gfr$ we have
	\be\label{eq:small-nhbd-p-adic-large-ball-in-end-lie-alg}
	\sum_{i=1}^{d^2} \bbz_p\Ad(g_{i,\rho})(x) \supseteq p^{\overline{C}l} \bbz_p[\Ad(G[1])]x.
	\ee
	On the other hand, $\gfr(\bbq_p)$ is a simple $\bbq_p[\Ad(G[1])]$-module. Hence by \cite[Lemma 40]{SG-super-approx-II}, there is a positive number $\overline{C}'$ (depending only on $\wt{\bbg}$ if $\bbg=\wt{\bbg}\otimes_{\bbq}\bbq_p$) such that 
	\be\label{eq:fixed-large-ball-lie-algebra}
	\bbz_p[\Ad(G[1])]x\supseteq p^{\overline{C}'} \| x \|^{-1} (\gfr \cap \M_{n_0}(\bbz_p)).  
	\ee
	By \eqref{eq:small-nhbd-p-adic-large-ball-in-end-lie-alg} and \eqref{eq:fixed-large-ball-lie-algebra}, we deduce that
	\[
	\Big\{\sum_{i=1}^{d^2} c_i\Ad(g_{i,\rho})(x)|\h c_i\in F, |c_i|\leq 1 \Big\}\supseteq 0_{p^{-\overline{C}\, \overline{C}'l}\|x\|},
	\]
	and the $p$-adic case follows.
	
	Suppose $\{g_1,\ldots,g_m\}$ is as in Lemma~\ref{lem:polytop-group-ring-adjoint}. Notice that 
	\[
	\Big\{\sum_{i=1}^m c_i \Ad(g_i)|\h c_i\in[-1, 1] \Big\}\subseteq 0_{m},
	\]
	where $0$ is the zero of $\mathbb R[\Ad(G)]$. 
	Now by Lemma~\ref{lem:polytop-group-ring-adjoint}, we have 
	\[
	\Big\{\sum_{i=1}^m c_i \Ad(g_i^t)|\h c_i\in[-1, 1] \Big\}\supseteq 0_{t^C}\supseteq \frac{t^C}{m}\Big\{\sum_{i=1}^m c_i \Ad(g_i)|\h c_i\in[-1, 1] \Big\}, 
	\]
	for some positive numbers $C$. Combining this and Lemma~\ref{lem:polytops-g-module}, we have 
	\[
	\Big\{\sum_{i=1}^m c_i \Ad(g_i^t)(x)|\h c_i\in[-1, 1] \Big\}\supseteq 0_{r_0t^C \| x \|/m},
	\]
	where $r_0$ is the constant appearing in Lemma~\ref{lem:polytops-g-module}. The real case follows.
\end{proof}

\begin{proof}[Proof of Proposition~\ref{prop:conjugation-by-large-ball-multiplication-large-image}]
	Note that 
	\[
	\ad: \Lie(\bbg)\to \Lie(\bbg)
	\]
	is an $F$-isomorphism. 
	This implies that, for every $x\in \Lie(\bbg)(F)$, 
	\[
	\|x\|_{\op}\ll\|\ad(x)\|_{\op}\ll\|x\|_{\op}
	\]
	where the implied constant depends only on $\bbg(F)$.
	
	Moreover, if $\bbg=\wt\bbg\otimes_\bbq\bbq_p$, then $\ad$ is induced from a $\bbq$-isomorphism of $\Lie(\wt\bbg)$.
	Therefore, in this case the implied constants equal $1$ if $p$ is large enough depending only on $\wt\bbg$ --- indeed, note that $(\ad)^{-1}$ is a $\bbq$-isomorphism, and so for $p$ large enough (depending only on $\wt\bbg$) $\ad$ is an isometry.   
	
	We also notice that $\log: 1_{p_0^{-2}}\to \gfr$ is a bi-Lipschitz map where $p_0=2$ if $F=\bbr$ and $p_0=p$ if $F=\bbq_p$. Indeed if $F=\bbq_p$, then $\log$ is an isometry between $1_{p^{-2}}$ and $0_{p^{-2}}$. Thus
	\[ 
	\|\Ad(h)-I\|_{\op}\gg\|\log (\Ad(h))\|_{\op}\gg\|\ad(\log h)\|_{\op}\gg\|\log h\|_{\op}\gg\|h-I\|_{\op},
	\]
	for all $h\in 1_{p_0^{-2}}$. If $\bbg=\wt\bbg\otimes_\bbq\bbq_p$ and $p$ is large enough, $\gg$ may be replaced with $=$ in the above.
	
	We conclude that 
	\be\label{eq: lower bd Ad(h)}
	\|(\Ad(h)-I)(x)\|\ge c' \|h-I\| 
	\ee
	for a unit vector  $x \in \gfr$ and positive number $c'=c'(G)$ which depends only on $\wt\bbg$ if $\bbg=\wt\bbg\otimes_\bbq\bbq_p$.
	
Recall that $0<\rho\leq p_0^{-2}$.  
	 Let $I:=\{a\in F|\h |a|< 1/4\}$, and let $\{g_{1,\rho},\ldots,g_{d^2,\rho}\}$ be given by Lemma~\ref{lem:adjoint-action-small-ball-polytop}.
	 Set
	 	\[
	\Phi:I \times \cdots \times I \rightarrow G, \h \Phi(t_1,\ldots,t_{d^2}):=g_{1,\rho} [h,\exp(t_1x)] g_{1,\rho}^{-1} \cdots  g_{d^2,\rho} [h,\exp(t_{d^2}x)] g_{d^2,\rho}^{-1}.
	\]
	Then $\dif\Phi(0):F^{d^2}\rightarrow \gfr$ is given by
	\[
	\dif\Phi(0)(c_1,\ldots,c_{d^2})=\sum_{i=1}^{d^2} c_i \Ad(g_{i,\rho})(\Ad(h)-I)x.
	\]
	Hence by Lemma~\ref{lem:adjoint-action-small-ball-polytop} and the choice of $x$, as in~\eqref{eq: lower bd Ad(h)}, we obtain that
	\be\label{eq:lower-bound-sigma-dif}
	\sigma(\dif\Phi(0))\geq \frac{c'}{d} \|h-I\| \rho^C, 
	\ee
	where $\sigma$ of a matrix $A$ is given by $\sup\{r\in [0,\infty)|\h 0_r\subseteq A0_1\}$ and $C$ is a positive integer which depends on $G$. Writing the Taylor expansion of the exponential function, we obtain that in the $p$-adic case all the coefficients are $p$-adic integers and in the real case, for every $1\leq j,j'\leq d^2$ and $\xbf\in I\times \cdots \times I$, $\|\partial_{j,j'}\Phi(\xbf)\|\leq d^{O(1)}$. Therefore, by Theorem~\ref{thm:p-adic-inverse-function} and Theorem~\ref{thm:real-inverse-function}, for every $0<r\leq  c''\|h-I\| \rho^C$ (where $c''=c'd^{-O(1)}$), we have 
	\[
	1_{\frac{c'}{4d} r \|h-I\| \rho^C }\subseteq \Phi(0_r).
	\]
	This implies the claim with $c=\min\{\frac{c'}{4d}, c''\}$.   
\end{proof}

\section{Proof of Theorem~\ref{thm:main1}} 
In this section, we will complete the proof of Theorem~\ref{thm:main1}. 
We begin by recalling~\cite[Theorem~9.3]{MMSG}, which will be used both in this section and \S\ref{sec: proof of thm Q-form} below. 

 \begin{theorem}\label{thm:functions-scale-spec-gap}
 	Suppose $G$ is $L$-locally random with coefficients $C_0$ which satisfies the dimension condition {$\Dim(C_1, d_0)$}, see~\eqref{eq:dimension-condition-intro}. Let $\mu$ be a symmetric Borel probability measure on $G$, and the group generated by the support of $\mu$ is dense in $G$. Suppose that there exist $C_3>0$, $b>0$, and $ 0< \eta_0<1$ such that for every $\eta\le \eta_0$ and every function $g\in L^2(G)$ which lives at scale $\eta$ there exists $l\le C_3 \log(1/\eta)$ such that 
\[ 	\|\mu^{(l)}\ast g\|_2\le \eta^{b} \|g\|_2.\] 
Then there is a subrepresentation $\cal_0$ of $L^2(G)$ with $\dim \cal_0\le 2C_0\eta_0^{-d_0}$ 
such that 
\[
\lcal(\mu; L^2(G)\ominus \cal_0)\ge \frac{b}{C_3}.
\]
In particular, $\lcal(\mu;G)>0$. 
 \end{theorem}

In addition to Theorem~\ref{thm:functions-scale-spec-gap}, the proof also relies on results in~\cite[\S5]{MMSG}, 
as well as Theorem~\ref{thm: continuity prop of couplings} and Proposition~\ref{prop:contraction-coupling} in this paper. 

\begin{proof}[Proof of Theorem~\ref{thm:main1}]
The proof will be completed in some steps. We will write $F_i=\mathbb Q_{\nu_i}$. 
Recall that $G_i$ is a compact open subgroup of $\bbg_i(\bbq_{\nu_i})$ where for $i=1,2$, $\bbg_i$ is $\bbq_{\nu_i}$-simple group. 
Recall also our notation $p_\nu=\nu$ if $\nu$ is non-Archimedean and $p_\infty=2$.

Let $X=(X_1,X_2)$ be as in the statement, and let $\mu$ be the probability law of $X$. Since 
$\min\{\lcal(X_1),\lcal(X_2)\}\geq c_0>0$, we have  
\be\label{eq: marginals have gap}
\max\{\lambda(\pi_{1*}\mu; G_1), \lambda(\pi_{2*}\mu; G_2)\}=:\lambda <1
\ee
where $\pi_i$ denotes the projection onto the $i$-th factor.

\medskip

\noindent
{\bf Reduction to functions living at certain scale.}\ 
There exist $L$, $C_0$, $C_1$ and $d_{0i}$ for $i=1,2$ so that if we let $\rho_0\ll 1$, then all of the following properties are satisfied:  

\begin{enumerate}[label={(G-\arabic*)}]
\item\label{G-1} $G_i$ is $L$-locally random with coefficient $C_0$ for $i=1,2$.

\item\label{G-2} For all $\hat\eta=\rho_0^j$, $j\in\mathbb N$, the group $G_i$ satisfies 
\[
\frac{1}{C_1}\hat\eta^{d_{0i}} \le |1_{\hat\eta}^{(i)}| \le C_1 \hat\eta^{d_{0i}},\qquad\text{for $i=1,2$.} 
\]

\item\label{G-3} For $i=1,2$, $G_i$ is $\delta$-discretizable for all $\delta=\rho_0^j$, and all $j\in\mathbb N$ in the $p$-adic case and sufficiently large $j \in \mathbb N$ in the real case. 
\end{enumerate}
See~\cite[\S5]{MMSG} for the first statement, the second assertion holds in view of the choice of $\rho_0$, and the third property is satisfied by Proposition~\ref{prop: Lie gp discretizable}. We note that $L$ and $d_{01},d_{02}$ depend only on $\dim\G_i$, $C_0$ depends on $G_1$ and $G_2$, and $C_1$ depends on $\max\{p_{\nu_1},p_{\nu_2}\}$.

Let $m_i$ denote the probability Haar measure on $G_i$ for $i=1,2$, and let $m=m_1\times m_2$ denote the probability Haar measure on $G$.  

From \ref{G-1} and~\ref{G-2}, we conclude that $G=G_1\times G_2$ is $L$-locally random with coefficient $2C_0$, see~\cite[Lemma 5.2]{MMSG}. The group $G$ also satisfies ${\rm DC}(d_0, C_1)$ where $d_0=d_{01}+d_{02}$ and $C_1=C_{11}C_{12}$.

Fix some integer $a\geq 4Ld_0$. Let $\eta=\rho_0^{j}$ for some $j\geq j_0$; the parameter $j_0$ will be explicated later. Recall that $f\in L^2(G,m)$ lives at scale $\eta$ if both of the following properties are satisfied   
\be\label{eq: proof f leavs at eta}
\|f_{\eta^{a^2}}-f\|_2\leq \eta^{a/2}\|f\|_2,\quad\text{and}\quad \|f_{\eta^{1/a}}\|_2 \leq \eta^{1/(2a)}\|f\|_2.
\ee

We claim
 
\begin{claimm}\label{claim: proof 1}
There exists $\ell\leq \bar C\log(1/\eta)$ where $\bar C$ depends on $d_{0i}$, $L$ so that 
\be\label{eq: scale rho func}
\|\mu^{(\ell)} \ast f\|_2\leq \eta^{c/2} \|f\|_2
\ee
where $b$ depends only on $L$ and $d_{01}, d_{02}$
\end{claimm}

Note that in view of Theorem~\ref{thm:functions-scale-spec-gap},~\eqref{eq: scale rho func} finishes the proof. Thus it remains to prove the claim.

\proof[Proof of the Claim]
The proof relies on Theorem~\ref{thm: continuity prop of couplings} and Proposition~\ref{prop:contraction-coupling} as we now explicate. 

\medskip

\noindent 
{\bf Reduction to couplings of Haar measures.}\
Let $\eta=\rho_0^j$ for some $j\geq j_0$. Properties~\ref{G-1}, \ref{G-2}, \ref{G-3}, together with~\eqref{eq: marginals have gap} imply that Theorem~\ref{thm: continuity prop of couplings} is applicable with $G_1$, $G_2$, $\mu$, and $\rho=\eta^{a^2}$. In view of that theorem thus there exists a symmetric coupling $\sigma=\sigma^\rho$ of 
$m_1$ and $m_2$ so that the following holds. 
Let $f\in L^2(G, m)$ satisfy that $\|f_\rho-f\|_2\leq \crho^{1/(4a)}\|f\|_2$; then 
\be\label{eq: mu ell coupling}
\bigl\|\mu^{(\ell_1)}\ast f-\sigma\ast f\bigr\|_2\leq 6\crho^{1/(4a)}\|f\|_2,
\ee
so long as $\ell_1\gg_{a, d_{0i}, L} \log_\lambda (\rho/C_1)$, see~\eqref{eq: choice of ell coupling}.

\medskip

\noindent
{\bf Conclusion of the proof.}\
Let us write $\sigma=\sigma^\rho$. 
Let $j_0$ be large enough so that $\eta\leq \eta_0$ where $\eta_0$ is as in Proposition~\ref{prop:contraction-coupling}, in particular, $\eta_0=\max\{p_{\nu_1}, p_{\nu_2}\}^{-O_{\dim G}(1)}$. 
Then by Proposition~\ref{prop:contraction-coupling}, for every $f$ which lives at scale $\eta$ we have 
\be\label{eq: rhonu contracts}
\|\sigma^{(2^{n_0})}*f\|_2\leq \eta^c\|f\|_2
\ee
where $c$ and $n_0$ depend only on $L$ and $d_{01}, d_{02}$. Without loss of generality, we assume $c<1/(4a)$.

Let $0< n_1\leq 2^{n_0}$ be the smallest integer so that 
\[
\|\sigma^{(n_1)}*f\|_2\leq \rho^{1/4a}\|f\|_2
\]
if such exists, otherwise let $n_1=2^{n_0}$. Now for all $0\leq i<n_1$, we have $\rho^{-1/4a}\|\sigma^{(i)}*f\|_2\geq \|f\|_2$; hence, using \eqref{eq:Young-ineq}, we have  
\[
\|\sigma^{(i)}*f_\rho-\sigma^{(i)}*f\|_2 \leq \|f_\rho-f\|_2 \leq \rho^{1/2a}\|f\|_2\leq \rho^{1/4a}\|\sigma^{i}*f\|_2,
\]
where we also used $\sigma^{(i)}*f_\rho=(\sigma^{(i)}*f)_\rho$ and the first estimate in~\eqref{eq: proof f leavs at eta}. 
Therefore by \eqref{eq: mu ell coupling}, we deduce that 
\[
\|\mu^{(\ell_1)}\ast \sigma^{(i)}\ast f-\sigma^{(i+1)}\ast f\|_2\leq 6\rho^{1/(4a)}\|f\|_2,
\]
for every $i<n_1$. Using the triangle inequality and \eqref{eq:Young-ineq}, we get 
\begin{align*}
    \|\mu^{(\ell_1 n_1)}\ast f-\sigma^{(n_1)}\ast f\|_2\leq &
    \sum_{i=0}^{n_1-1} \|\mu^{(\ell_1 (n_1-i))}\ast \sigma^{(i)}\ast f-\mu^{(\ell_1(n_1-i-1))}\ast \sigma^{(i+1)}\ast f\|_2  
    \\
    = &
    \sum_{i=0}^{n_1-1} \|\mu^{(\ell_1 (n_1-i-1))}\ast (\mu^{(\ell_1)}\ast \sigma^{(i)}\ast f- \sigma^{(i+1)}\ast f)\|_2 
    \\
    \leq & 6n_1\rho^{1/(4a)}\|f\|_2. 
\end{align*}
By \eqref{eq: rhonu contracts} and the choice of $n_1$, we have that
$\|\sigma^{n_1}\ast f\|_2\leq \eta^c \|f\|_2$. Hence, 
\[
\|\mu^{(\ell_1 n_1)} \ast f\|_2\leq \eta^{c/2} \|f\|_2.
\]
This implies that~\eqref{eq: scale rho func} holds with $b=c/2$ and $\ell=\ell_1n_1\leq \bar C\log(1/\eta)$ where $\bar C$ depends on $d_{0i}$, $L$, and $a$. This completes the proof of the claim, and hence the proof of the theorem. 
\end{proof}

\section{Spectral independence and exceptional representations}\label{sec: spec indep and exceptional}

The objective of this section is to prove Proposition~\ref{prop: spectral indep exceptional}, which will be used in the proof of Theorem~\ref{thm:Q-form}. 
To obtain Theorem~\ref{thm:Q-form}, we apply Theorem~\ref{thm:main1} 
to the group $\Gamma_{\nu_1,\nu_2}$ where $\nu_i\in V_\Gamma$; see the notation in Theorem~\ref{thm:Q-form}. Then using Claim~\ref{claim: proof 1} in the proof of Theorem~\ref{thm:main1} and Theorem~\ref{thm:functions-scale-spec-gap}, we reduce the analysis to the study of {\em exceptional representations} of $\Gamma_{\nu_1,\nu_2}$ whose dimension is $\ll_\Gamma (\max\{p_{\nu_1},p_{\nu_2}\})^{O(1)}$. 
The proof of Proposition~\ref{prop: spectral indep exceptional} relies on the results in~\cite{Golsefidy-IMRN-SA, SG-super-approx-II} and 
Proposition~\ref{prop:conjugation-by-large-ball-multiplication-large-image} in this paper.

We begin by recalling some of the notation from Theorem~\ref{thm:Q-form}.  
Let $\bbg$ be an absolutely almost simple, connected, simply connected $\bbq$-algebraic group. 
As it was done before, we fix a $\bbq$-embedding $\bbg\subseteq(\SL_{N})_\bbq$, for some $N$. 
Let $\Omega \subseteq \bbg(\bbq)$ be a 
finite symmetric subset such that $\Gamma=\langle\Omega\rangle$ is Zariski dense in $ \bbg$. 
Denote by $V_{ \Gamma }$  the set of all places $\nu$ of $ \bbq$ such that $\Gamma$ is a bounded subset of $ \bbg( \bbq_\nu)$, 
and let $V_{{\rm f},\Gamma}\subseteq V_\Gamma$ denote the subset of finite places. 

For distinct $\nu_1, \nu_2 \in V_{ \Gamma}$, let $\Gamma_{\nu_1, \nu_2}$ denote the closure of $\Gamma$ in 
$ \bbg( \bbq_{\nu_1} ) \times   \bbg( \bbq_{\nu_2} )$. The following are our standing assumption in this section:
\[
\begin{aligned}
&\Gamma_\nu\subseteq \SL_N(\bbz_{p_\nu})&&\text{for all $\nu\in V_{{\rm f},\Gamma}$, and}\\
&\Gamma_{\nu_1,\nu_2}=\Gamma_{\nu_1}\times\Gamma_{\nu_2},&&\text{for all $\nu_1, \nu_2\in V_\Gamma$,}
\end{aligned}
\]
where $\Gamma_\nu$ denotes the closure of $\Gamma$ in $\bbg(\bbq_\nu)$. 

Let $W_\Gamma\subset V_{{\rm f},\Gamma}$ denote the set of places where $\bbg(\bbq_{\nu})\cap \SL_N(\bbz_{p_\nu})$ is a hyperspecial subgroup of $\bbg(\bbq_{p_\nu})$ and 
\[
\Gamma_\nu=\bbg(\bbq_{p_\nu})\cap \SL_N(\bbz_{p_\nu}).
\]

For every $n\geq 0$, let $\pi_{\nu,n}: \SL_N(\bbz_{p_\nu})\to\SL_N(\bbz/p_\nu^n\bbz)$.
For every $\nu\in V_{{\rm f},\Gamma}$, let 
\[
\Gamma_{\nu,n}=\Gamma_\nu\cap \ker\bigl(\pi_{\nu,n}\bigr)
\]
denote the $n$-th congruence subgroup of $\Gamma_{\nu}$. 

If $\infty\in V_\Gamma$, we put $\Gamma_{\infty,n}=\{1\}$ for all $n\in\bbn$ --- recall also our convention $p_\infty=2$.

\begin{proposition}\label{prop: spectral indep exceptional}
For all $D,d>0$, there exists $\varrho=\varrho(\Omega, d,D)>0$ so that the following holds. 
For $i=1,2$, let $\nu_i\in V_\Gamma$ and $n_i\in\bbn$.
Put $G_i=\Gamma_{\nu_i}/\Gamma_{\nu_i,n_i}$ and assume 
\be\label{eq: dimension bound}
\text{the number of connected components of $G_1\times G_2$}\leq D\max\{p_{\nu_1}, p_{\nu_2}\}^{d}.
\ee
Then $\mathcal{L} ( X; G_1\times G_2) > \varrho$,  
where $X$ is a random variable with the uniform distribution on $ \Omega$. 
\end{proposition}

The following theorem is a special case of~\cite[Theorem 24]{SG-super-approx-II}, applied with $\mathcal C=p_\nu^n$ 
for $\nu\in V_{{\rm f},\Gamma}$, and will play a key role in the proof of Proposition~\ref{prop: spectral indep exceptional}.

\begin{theorem}\label{thm: SG II bdd generation}
There exist positive numbers $\bar E=\bar E(\bbg)$, $R_1=R_1(\Omega)$, and $\vare_0=\vare_0(\Omega)$, such that 
for every $0<\vare\leq \vare_0(\Omega)$, $0<\delta\ll_{\Omega,\vare}1$ and $R_{2}\gg_{\Omega,\vare}1$ the following holds. 
Let $n\in\bbn$ and assume ${n\vare^{\bar E}}\log(p_\nu)\geq \log(R_1)$. 
Suppose for a finite symmetric subset $A\subseteq \Gamma$ and some $\ell\geq (n\log p_\nu)/\delta$, we have 
\[
\mathcal P_\Omega^{(\ell)}(A)\geq p_\nu^{-\delta n}, 
\] 
then there exists a non-negative integer $n'\leq \vare n$ so that 
\[
\textstyle
\pi_{\nu, n}(\Gamma_{\nu,n'})\subseteq \prod_{R_2}\pi_{\nu, n}(A).
\]
\end{theorem}

Let us now return to the proof of Proposition~\ref{prop: spectral indep exceptional}. 
Note that increasing $D$ and $d$ makes the assumption weaker; therefore throughout the argument, we may assume $D, d>100$.

\subsection{Both places are finite}\label{sec: exceptional both finite}
We will first consider the case 
\be\label{eq: p adic p adic}
\text{$\nu_1, \nu_2\in V_{\Gamma}\quad$ are finite places.}
\ee
Let us write $G=G_1\times G_2$, and put $M=\max\{|G_1|, |G_2|\}$. 
As before, $m_{G_i}$ and $m_G$ denote the uniform measures on $G_i$ and $G$, respectively. 
To ensure consistency with most of the existing literature, in this case where $G$ is finite, we deviate from the notation in the rest of the paper and define the convolution of $f_1,f_2\in \ell^2(G)$ using the counting measure. That is:
\[
f_1\ast  f_2(x)=\sum_{y \in G}f_1(xy^{-1})f_2(y).
\]

\subsection*{Reduction to deep levels verses large primes}
Since almost every $\nu$ belongs to $W_\Gamma$, in the proof of Proposition~\ref{prop: spectral indep exceptional} we  may assume that $\{\nu_1,\nu_2\}\cap W_\Gamma\neq\emptyset$. Thus, for the rest of the proof we will assume 
that $\max\{p_{\nu_1},p_{\nu_2}\}=p_{\nu_1}\geq D$ and that $\nu_1\in W_\Gamma$. Note that 
\be\label{eq: IGI and p nu1}
p_{\nu_1}^{n_1}\leq |G_1\times G_2|\leq Dp_{\nu_1}^d\leq p_{\nu_1}^{d+1}.
\ee
In consequence, we will assume for the rest of the argument that 
\be\label{eq: n1 is less that d}
n_1\leq d+1.
\ee

Recall from Lemma~\ref{lem: locally random family} that $\Gamma_{\nu_1,\nu_2}$ is $L$-locally random with coefficient $C_0$ depending only on $\Gamma$. This in particular implies that
for every $\rho\in\hat G$, we have 
\be\label{eq: metric quasi random G}
\dim(\rho)\geq c_\Gamma [G:\ker\rho]^{\alpha_\Gamma},\qquad\text{for some $c_\Gamma, \alpha_\Gamma>0$. }
\ee

We will assume, for the rest of the argument, that 
\be\label{eq: n2 is large}
\max\{n_1, n_2\}> 8d^2(d+1)E_2/(\vare_0(\Omega) \alpha_\Gamma),
\ee
where $d$ is as in Proposition~\ref{prop: spectral indep exceptional}, $\vare_0(\Omega)$ 
is as in Theorem~\ref{thm: SG II bdd generation}, and $E_2=2C/c$ with $C$ and $c$ as in Proposition~\ref{prop:conjugation-by-large-ball-multiplication-large-image}, see~\eqref{eq: large fiber for H1} below --- as it will be explicated later, if~\eqref{eq: n2 is large} fails, Proposition~\ref{prop: spectral indep exceptional} follows from~\cite[Theorem 1]{SG-super-approx-II}.

\subsection*{Equidistribution of marginals}
Let $X=(X_1,X_2)$ be as in the statement of Proposition~\ref{prop: spectral indep exceptional}, and let $\mu$ be the probability law of $X$. 
Since $\langle \Omega\rangle =\Gamma$ is Zariski dense in $\bbg$, which is absolutely almost simple, we have   
$\min\{\lcal(X_1),\lcal(X_2)\}\geq c_0>0$ where $c_0=c_0(\Omega)$, see~\cite[Theorem 1]{Golsefidy-IMRN-SA} and references therein. Thus    
\be\label{eq: marginals have gap 1}
\max\{\lambda(\pi_{1*}\mu; G_1), \lambda(\pi_{2*}\mu; G_2)\}\leq \lambda <1
\ee
where $\pi_i$ denotes the projection onto the $i$-th factor and $\lambda:=\lambda(\Omega)$. 
Thus, there exists $E_1\geq 1$ so that for all $\ell_1\geq E_1 \log M$, we have  
\be\label{eq: marginals close to equi}
\|\pi_{i*}\mu^{(\ell_1)}-m_{G_i}\|_2\leq |G|^{-1};
\ee
this can be seen, e.g., by applying Lemma~\ref{lem: spectral gap implies close to Haar} with $\eta=1/M$ and choosing $E_1$ so that 
$\lambda^{\ell_1/2}M^2\leq 1$ for all $\ell_1\geq E_1 \log M$. 

Let us fix one such $\ell_1$, and put $\sigma:=\mu^{(\ell_1)}$. The goal now is to show the following 

\begin{lem}\label{lem: sigma m0 equidist}
With the above notation, there exists $m_0=m_0(\Omega)\in\bbn$ so that 
\[
\|\sigma^{(2^{m_0})}-m_G\|_2\ll_\Omega|G|^{-1}.
\]
\end{lem}

The argument is similar to the one we used in the proof of Propositions~\ref{prop:desired-entropy} and~\ref{prop:contraction-coupling};
we now turn to the details. 

\begin{proof}
The proof will be completed in some steps. 
For every integer $m\geq0$, let $\sigma_m=\sigma^{(2^{m})}$.

\subsection*{Quasi randomness and its consequences}
Recall from~\eqref{eq: metric quasi random G} that
\be\label{eq: metric quasi random G'}
\dim(\rho)\geq c_\Gamma [G:\ker\rho]^{\alpha_\Gamma}\qquad\text{for any $\rho\in\hat G$.}
\ee

Let $\rho$ be an irreducible component of $\ell^2(G)$. 
If $[G:\ker\rho]< p_{\nu_1}$, then $G_1\subseteq \ker\rho$. That is: $\rho$ may be identified as a representation of 
$G/G_1\simeq G_2$ and the claim in the lemma follows from~\eqref{eq: marginals close to equi}. 
Therefore, in the proof of the lemma we may restrict to representations $\rho$ so that 
\be\label{eq: rho is not small}
[G:\ker\rho]\geq p_{\nu_1}\geq |G|^{1/(d+1)},
\ee
where in the last inequality we used~\eqref{eq: IGI and p nu1}. In view of~\eqref{eq: metric quasi random G'}, thus, 
it suffices to consider  
\be\label{eq: dim of not small rho}
\dim\rho\geq c_\Gamma|G|^{\alpha_\Gamma/(d+1)}.
\ee
Now if $f\in(\ell^2(G))_\rho$ (the $\rho$-isotypic submodule of $\ell^2(G)$) for some $\rho$ satisfying~\eqref{eq: dim of not small rho}, then  
\be\label{eq: mixing inequality finite}
\|\chi* f\|_2\ll_\Gamma |G|^{-\alpha_\Gamma/(2d+2)}|G|^{1/2}\|\chi\|_2\|f\|_2,\quad\text{for any $\chi\in\ell^2(G)$,}
\ee
see~\cite[Lemma 6.1]{MMSG} and references therein, note also that in this proof convolution operator is defined using the counting measure. 

Thus in order to prove the lemma, it suffices to show that 
\be\label{eq: G is D quasi random}
\|\sigma_{m_1}\|_2^2\leq |G|^{-1+\beta}
\ee
for some $0<\beta<\alpha:=\alpha_{\Gamma}/(2d+2)$ and some $m_1=m_1(\Omega)$.

\subsection*{Initial entropy and its consequences} 
In view of~\eqref{eq: marginals close to equi}, 
\[
\sigma_m((g_1,g_2))\leq \min\{\pi_{1*}\sigma(g_1),\pi_{2*}\sigma(g_2)\}\leq 2M^{-1}
\]
for all $m\geq 0$ and all $(g_1,g_2)\in G$. Since $\sigma_m$ is a probability measure, we conclude that 
\be\label{eq: L2 norm of sigma m}
\|\sigma_m\|_2^2\leq  2M^{-1}\qquad \text{for all $m\geq 0$;} 
\ee
see also Lemma~\ref{lem:coupling-entropy-lower-bound} with $\eta=1/|G|$.  

In view of~\eqref{eq: L2 norm of sigma m},~\eqref{eq: G is D quasi random} holds unless 
\be\label{eq: G1 and G2 are both large}
|G_i|\geq DM^{\alpha/4}\qquad\text{for $i=1,2$,}
\ee 
where $D$ is as in the statement of proposition. 
Therefore, for the rest of the argument we will assume that~\eqref{eq: G1 and G2 are both large} holds. 
 
Let $d$ be as in the statement of the proposition. For every $0<\vare<1$, put  
\[
{\rm Exc}(\vare):=\bigl\{(\nu, n)\in V_{{\rm f},\Gamma}\times \bbn: p_\nu^n\leq R_1^{1/\vare^{\bar E}}\bigr\},
\] 
where $\bar E$ and $R_1$ are as in Theorem~\ref{thm: SG II bdd generation}.

Let 
\be\label{eq: choose vare}
\vare=\vare_0(\Omega)\alpha/(8dE_2).
\ee  
Note that in the proof of Proposition~\ref{prop: spectral indep exceptional}, we are allowed to assume 
\be\label{eq: M is really large}
(\alpha\log M)/4> (\log R_1)/\vare^{\bar E}.
\ee
Since ${\rm Exc}(\vare)$ is a finite set,~\eqref{eq: M is really large} and~\eqref{eq: G1 and G2 are both large} imply that we may assume 
$(\nu_1,n_1), (\nu_2,n_2)\not\in{\rm Exc}(\vare)$ for the rest of the proof.

Let $\delta$ and $R_2$ be as in Theorem~\ref{thm: SG II bdd generation} applied with $\vare$, and let $m_2$ be so that $2^{m_2}E_1>1/\delta$. Then 
\be\label{eq: the choice of m1}
2^{m_2}\ell_1\geq 2^{m_2}E_1\log M\geq (\log M)/\delta. 
\ee

\subsection*{$\ell^2$-Flattening lemma} Let $K=M^\kappa$ for some $\kappa>0$ which will be explicated momentarily. 
Assume now that for some $m\geq m_2$, we have  
\be\label{eq: sigma m does not flatten}
\|\sigma_m*\sigma_m\|_2> K^{-1}\|\sigma_m\|_2.
\ee
Then, there exists $H\subseteq G$ which is symmetric and contains $1$, so that all of the following hold: 
\begin{enumerate}[label={(H-\arabic*)}]
\item\label{H-1} $K^{-E}\|\sigma_m\|_2^{-2}\leq |H|\leq K^{E}\|\sigma_m\|_2^{-2}$, 
\medskip
\item\label{H-2} $|H\cdot H\cdot H|\leq K^{E}|H|$, and  
\medskip
\item\label{H-3} $\sigma_m*\sigma_m(H)\geq K^{-E}$,
\end{enumerate}
where $E$ is absolute. See, e.g.,~\cite[Theorem 2.12]{MMSG} and references therein. 

We will apply the above with $\kappa=\alpha_\Gamma(\min\{\vare,\delta\})/(10R_2E\dim\bbg)$. This in particular implies   
\be\label{eq: fix kappa}
K^{-E}=M^{-E\kappa}\geq M^{-\delta\alpha/4}\geq p_{\nu_i}^{-\delta n_i},
\ee 
where in the last inequality we used~\eqref{eq: G1 and G2 are both large} and $|G_i|=p_{\nu_i}^{n_i}$.

\subsection*{Bounded generation and the structure of $H$} 
Let $\vare$ be as in~\eqref{eq: choose vare}, and let $A\subseteq \Gamma$ be a finite symmetric lift of $H$. 
Then by~\ref{H-3} and~\eqref{eq: fix kappa}, we have 
\be\label{eq: mathca P of A}
\mathcal P_\Omega^{(2^{m+1}\ell_1)}(A)\geq K^{-E}\geq p_{\nu_i}^{-\delta n_i}\quad\text{for $i=1,2$.}
\ee
Apply Theorem~\ref{thm: SG II bdd generation} with $\vare$ and $A$. 
Since $\pi_{\nu_i,n_i}(A)=\pi_i(H)$,~\eqref{eq: mathca P of A} and Theorem~\ref{thm: SG II bdd generation} 
imply that for $i=1,2$, there exists $0\leq n'_i\leq n_i\vare$ so that if we put $H_1=\prod_{R_2} H$, then 
\be\label{eq: marginals of H1 have structure}
\pi_{\nu_i,n_i}(\Gamma_{\nu_i,n'_i})\subseteq \pi_i(H_1),\qquad\text{for $i=1,2$.} 
\ee 
Recall from~\eqref{eq: n1 is less that d} that $n_1\leq d+1$. Hence, $\vare n_1\leq \vare (d+1)<1$, which implies $n'_1=0$. 
We also record for later use that $n_1\leq d+1$ and~\eqref{eq: n2 is large} imply
\be\label{eq: vare n2 is at least 2}
\vare n_2>2. 
\ee

\subsection*{A homomorphism from $G_1$ into $G_2$} 
In view of~\eqref{eq: marginals of H1 have structure} and since $n'_1=0$, we have  
\be\label{eq: marginals of H1 have structure'}
\textstyle
G_1=\pi_1(H_1)\qquad\text{and}\qquad \pi_{\nu_2,n_2}(\Gamma_{\nu_2, n'_2})\subseteq \pi_2(H_1),
\ee
where $0\leq n'_2\leq \vare n_2$. 

For every $g_2\in G_2$, put $\ell(g_2):=\max\{\ell\geq 0: g_2\in\ker(\pi_{\nu_2,\ell})\}$. 
Define a function $s: G_1\to G_2$ by 
\[
(g_1,s(g_1))\in H_1 \quad\text{and}\quad \ell(s(g_1))=\min\{\ell(g_2):(g_1,g_2)\in H_1\}.
\]
Define $\ell_s$ by ${\rm Im}(s(G_1))\subseteq \pi_{\nu_2, n_2}(\Gamma_{\nu_2,\ell_s})$ but ${\rm Im}(s(G_1))\not\subseteq \pi_{\nu_2, n_2}(\Gamma_{\nu_2,\ell_s+1})$. 
In view of the definition of $s$ and~\eqref{eq: marginals of H1 have structure'}, we have 
\be\label{eq: the size of ells}
\ell_s\leq n'_2\leq \vare n_2.
\ee
Let us also define $k_s=\min\{k'_s, k_s''\}$ where  
\[
\begin{aligned}
k_s'&=\min\bigl\{\ell\bigl(s(g_1)s(g_1^{-1})\bigr): g_1\in G_1\bigr\},\quad\text{and}\\
k_s''&=\min\bigl\{\ell\bigl(s(g_1g'_1)s(g_1')^{-1}s(g_1)^{-1}\bigr): g_1,g'_1\in G_1\bigr\}.
\end{aligned}
\]
Then $s$ induces a homomorphism $f: G_1\to G_2/\pi_{\nu_2, n}(\Gamma_{\nu_2, k_s})$. 
Recall now that $p_{\nu_1}> p_{\nu_2}$, $n_1\leq d+1$, and $\Gamma_{\nu_1}$ is a hyperspecial subgroup. Thus $G_1/\Gamma_{\nu_1,1}$ does not arise as a composition factor of any subgroup of $G_2/\pi_{\nu_2, n}(\Gamma_{\nu_2, k_s})$. We conclude that $f$ is trivial, and hence 
\[
k_s\leq\ell_s.
\]

\subsection*{Large fibers for $H_1$}
Note that 
\[
\text{$\bigl(1, s(g_1)s(g_1^{-1})\bigr)\in H_1\cdot H_1\quad$ and $\quad\bigl(1, s(g_1g'_1)s(g_1')^{-1}s(g_1)^{-1}\bigr)\in H_1\cdot H_1\cdot H_1$}
\]
 for all $g_1, g'_1\in G_1$. This and $k_s\leq \ell_s\leq \vare n_2$, see~\eqref{eq: the size of ells}, imply the following:
\be\label{eq: element with small level}
\text{There exists $h\in G_2$ with $\ell(h)\leq \vare n_2$ so that $(1,h)\in H_1\cdot H_1\cdot H_1$}
\ee

Apply Proposition~\ref{prop:conjugation-by-large-ball-multiplication-large-image} with the group 
$\Gamma_{\nu_2}$, $h$ as in~\eqref{eq: element with small level}, and $\rho=p_{\nu_2}^{-\vare n_2}$ 
(recall from~\eqref{eq: vare n2 is at least 2} that $\vare n_2>2$). 
Then by Proposition~\ref{prop:conjugation-by-large-ball-multiplication-large-image} combined with ~\eqref{eq: element with small level} and~\eqref{eq: marginals of H1 have structure'} (for $\nu_2$), there exists $E_2$ (depending on $\bbg$) 
so that if we put $t=E_2\vare n_2$, then 
\be\label{eq: large fiber for H1}
\{(1, g_2): g_2\in \pi_{\nu_2,n_2}(\Gamma_{\nu_2, t})\}\subset \textstyle\prod_{7\dim \bbg} H_1. 
\ee

\subsection*{The conclusion of the proof} 
Recall now from~\eqref{eq: marginals of H1 have structure'} that $G_1=\pi_1(H_1)$. 
This and~\eqref{eq: large fiber for H1} imply 
\be\label{eq: H1 is large}
\Bigl|\textstyle\prod_{R_2(1+7\dim \bbg)} H\Bigr|\geq |G_1||\pi_{\nu_2,n_2}(\Gamma_{\nu_2, t})|=|G_1||G_2|^{1-E_2\vare},
\ee
where we also used $H_1=\prod_{R_2}H$. 

In view of~\ref{H-1} and~\ref{H-2}, we have 
\[
\textstyle
K^{E+R_2E(1+7\dim \bbg)}\|\sigma_m\|_2^{-2}\geq K^{R_2E(1+7\dim \bbg)}|H|\geq \Bigl|\textstyle\prod_{R_2(1+7\dim \bbg)} H\Bigr| 
\]
Combining this with~\eqref{eq: H1 is large}, we get 
\[
K^{E+R_2E(1+7\dim \bbg)}\|\sigma_m\|_2^{-2}\geq |G_1||G_2|^{1-E_2\vare}.
\]
Recall that $K=\max\{|G_1|,|G_2|\}^\kappa$ where $\kappa=(\alpha\min\{\vare,\delta\})/(10R_2E\dim\bbg)$. 
In consequence, 
\[
\|\sigma_m\|_2^{-2}\geq |G_1||G_2|^{1-E_2\vare} K^{-8R_2E\dim \bbg}\geq |G|^{1-\frac{\alpha}{4}},
\]
where we also used the choice of $\vare$, see~\eqref{eq: choose vare}. 
In particular,~\eqref{eq: G is D quasi random} holds for $\sigma_m$. Altogether, 
\be\label{eq: sigma m does not flatten 1}
\text{for all $m\geq m_2$ either $\|\sigma_m*\sigma_m\|_2\leq K^{-1}\|\sigma_m\|_2$ or~\eqref{eq: G is D quasi random} holds for 
$\sigma_m$}
\ee
see~\eqref{eq: sigma m does not flatten}. 

Since $K=M^\kappa$ and $\|\sigma_m\|_2\ll M^{-1/2}$ for all $m$, see~\eqref{eq: L2 norm of sigma m}, 
we conclude from~\eqref{eq: sigma m does not flatten 1} that there exists $m_2\leq m_1\ll_\Omega 1$ 
so that~\eqref{eq: G is D quasi random} holds. The proof is complete.  
\end{proof}
 
\subsection{Infinite and finite place}\label{sec: exceptional finite and infinite}
We now investigate Proposition~\ref{prop: spectral indep exceptional} in the case 
\be\label{eq: real p adic}
\nu_1,\nu_2\in V_\Gamma,\quad\nu_2=\infty,\quad\text{and $\quad\nu_1$ is a finite place.}
\ee

\subsection*{Preliminary reductions}
Since almost every $\nu$ belongs to $W_\Gamma$, in the proof of Proposition~\ref{prop: spectral indep exceptional} we  may assume that $\nu_1\in W_\Gamma$.

Note that $G_1$ is a finite group, we thus equip $G$ with the metric $d$ induced by the admissible metric on $G_2$ and the discrete metric on $G_1$. To be more explicit, 
\[
1_\eta= \{1^{(1)}\}\times 1_\eta^{(2)} \qquad\text{for all $0<\eta<1$.}
\]

We let $m_{G_i}$ denote the probability Haar measures on $G_i$ for $i=1,2$, and let $m_G$ be the product measure. 
Then $|1_\eta|= |1^{(2)}_\eta/|G_1||$ for all $0<\eta<1$. In particular, we have 
\be\label{eq: dim cond G1}
\frac{1}{|G_1|C_1}\eta^{d_0}\leq |1_\eta|\leq \frac{C_1}{|G_1|}\eta^{d_0}.
\ee
where $d_0=\dim\bbg$ and $C_1\geq 2$ depends only on $d_0$. 
This in particular implies that for any $\theta>0$ and all $2^{-\ell-1}<\eta\leq 2^{-\ell}\leq |G_1|^{-\theta}$, we have  
\be\label{eq: dim cond eta small}
\frac{1}{2C_1}\eta^{d_0(\ell,\theta)}\leq |1_\eta|\leq 2C_1\eta^{d_0(\ell,\theta)}
\ee
where $d_0(\ell,\theta)=d_0+\frac{1}{\ell}\log_{2}|G_1|$; consequently, $d_0\leq d_0(\ell,\theta)\leq d_0+\theta^{-1}$. 

We will use the following notation:
\[
P_{\eta}^{(2)}=\frac{1}{|1_\eta^{(2)}|}1_\eta^{(2)}\quad\text{and}\quad P_\eta= \frac{1}{|1_\eta|}1_\eta. 
\]

Throughout the argument, and whenever necessary, we will assume 
\[
0<\eta\leq (C_1d_0)^{-O_{\Gamma}(1)}\qquad\text{and}\qquad p_{\nu_1}\geq (C_1d_0)^{O_{\Gamma}(1)}.
\]

As it was done in \S\ref{sec: exceptional both finite}, let $X=(X_1,X_2)$ be as in the statement of 
Proposition~\ref{prop: spectral indep exceptional}, and let $\mu$ be the probability law of $X$. 
Since $\langle \Omega\rangle =\Gamma$ is Zariski dense in $\bbg$, which is absolutely almost simple,  we have  
$\min\{\lcal(X_1),\lcal(X_2)\}\geq c_0>0$ where $c_0=c_0(\Omega)$, see~\cite{Benoist-Saxce-16, Golsefidy-IMRN-SA} and references therein. Thus    
\be\label{eq: marginals have gap 1'}
\max\{\lambda(\pi_{1*}\mu; G_1), \lambda(\pi_{2*}\mu; G_2)\}\leq \lambda <1
\ee
where $\pi_i$ denotes the projection onto the $i$-th factor and $\lambda:=\lambda(\Omega)$. 
Therefore, there exists $E_1\geq 1$ so that both of the following hold  
\be\label{eq: marginals close to equi'}
\begin{aligned}
&\bigl|\pi_{1*}\mu^{(\ell_1)}(f_1)|\leq |G_1|^{-1}\|f_1\|_2&&\text{for all $\ell_1\geq E_1\log|G_1|$}\\
&\bigl|\pi_{2*}\mu^{(\ell_1)}(f_{2,\eta})\bigr|\leq \eta\|f_{2}\|_2&&\text{for all $0<\eta<1$ and all $\ell_1\geq E_1 |\log \eta|$}
\end{aligned}
\ee
where $f_1\in L^2_0(G_1)$, $f_2\in L^2_0(G_2)$, and $f_{2,\eta}=f_2* P_\eta^{(2)}$.
This can be seen, e.g., by applying Lemma~\ref{lem: spectral gap implies close to Haar} with $\eta$ (and in the first case $\eta=|G_1|^{-1}$) and choosing $E_1$ so that $\lambda^{\ell_1/2}\eta^{-2}\leq 1$ for all $\ell_1\geq E_1 |\log \eta|$.

Recall that by the Peter--Weyl theorem we have 
\be\label{eq: Peter Weyl thm2}
L^2(G)=\bigoplus_{\rho_{1,i}\in \hat G_1, \rho_{2,j}\in \hat G_2} \dim(\rho_{1,i}\otimes\rho_{2,j})\rho_{1,i}\otimes\rho_{2,j}.
\ee

We will use the following fact in the sequel. Let $V_2=\oplus_{\rho_{2,j}\neq 1} \dim(1\otimes\rho_{2,j})1\otimes\rho_{2,j}$.
Then, $V_2$ is the natural embedding of $L^2_0(G_2)$ in $L^2_0(G)$, thus~\eqref{eq: marginals have gap 1'} implies that 
\be\label{eq: rep which factor through G2}
\mathcal L(\mu; V_2)\geq c_2(\Omega)>0. 
\ee

Finally, note that~\eqref{eq: dimension bound} and the fact that $G_2$ is connected imply that 
\[
p_{\nu_1}^{n_1}\leq |G_1|\leq Dp_{\nu_1}^d\leq p_{\nu_1}^{d+1}. 
\]
In consequence, we will assume for the rest of the argument that 
\be\label{eq: n1 is less that d'}
n_1\leq d+1.
\ee
In view of~\eqref{eq: n1 is less that d'} and part~(1) in Lemma~\ref{lem: locally random family}, we have: for any nontrivial representation $\rho\in \hat G_1$  
\be\label{eq: metric quasi random G''}
\dim(\rho)\geq c'_\Gamma\,|G_1|^{\alpha'_\Gamma},\qquad\text{for some $0<c_\Gamma', \alpha_\Gamma'\leq 1$. }
\ee

We begin with the following.

\begin{lem}\label{lem: quasi randomness}
There exists some $\alpha=\alpha(C_0, L, d_0)>0$ so that the following holds.
Assume there exists $E_1'$ so that for all $0<\eta\leq |G_1|^{-1}$ there exists some $\ell\leq E_1'|\log\eta|$ so that  
\[
\|(\mu^{(\ell)})_\eta\|_2\leq \eta^{-\alpha},
\]
then Proposition~\ref{prop: spectral indep exceptional} holds for $G=G_1\times G_2$. 
\end{lem}

\begin{proof}
This is essentially proved in~\cite[\S8]{MMSG}, we explicate some of the details for the convenience of the reader. 
First note that by Lemma~\ref{lem: locally random family}, there exists $(C_0, L)$ depending only on $\Gamma$ so that $G=G_1\times G_2$ is $L$-locally random with coefficient $C_0$.  Moreover, $G$ satisfies the dimension condition~\eqref{eq: dim cond G1}.
Therefore,~\cite[Theorem 2.10]{MMSG} implies that there exists $D$, depending only on $d_0$, and $\alpha'$, depending only on $L$ and $C_0$, so that the following holds: there is an exceptional representation $\mathcal H_0$ of $G_1\times G_2$ with 
$\dim\mathcal H_0\leq 2C_0C_1^D|G_1|^{D}$ so that, if for all $\eta\leq |G_1|^{-D}$ there exists $\ell\leq E'|\log\eta|$ satisfying 
\[
\|(\mu^{(\ell)})_\eta\|_2\leq \eta^{-\alpha'},
\] 
then we have 
\be\label{eq: large dimensional rep}
\mathcal L(\mu; L^2(G)\ominus \mathcal H_0)\geq c,
\ee 
where $c=O_{E', L, C_0}(\alpha')>0$. 

Thus, we may restrict our attention to 
\[
V:=\oplus (\dim(\rho_{1,i}\otimes\rho_{2,j}))\rho_{1,i}\otimes\rho_{2,j}\supseteq \mathcal H_0
\] 
where the sum is over representations with $\dim(\rho_{1,i}\otimes\rho_{2,j})\leq 2C_0C_1^D|G_1|^{-D}$. 

By~\cite[Lemma 8.8]{MMSG}, there exists $D'\geq 1$, depending on $C_0, L, D$, so that if we put $\delta=|G_1|^{-D'}$, 
\be\label{eq: P eta almost inv}
\|P_\delta*f-f\|_2\leq \delta^{1/2}\|f\|_2
\ee
for any $f\in V$.  

Moreover, in view of~\eqref{eq: rep which factor through G2}, we may replace $V$ by $V'\subseteq V$ 
where in the sum above we further assume that $\rho_{1,i}\neq 1$. 
Then in view of~\eqref{eq: metric quasi random G''}, for all components of $V'$, we have 
\[
\dim(\rho_{1,i})\geq c'_\Gamma\, |G_1|^{\alpha'_\Gamma}.
\]
Recall again that there exists $(C_0, L)$ 
depending only on $\Gamma$ so that $G_1\times G_2$ is $L$-locally random with coefficient $C_0$, see Lemma~\ref{lem: locally random family}. Thus by~\cite[Lemma 6.1]{MMSG}, for any $\chi\in L^2(G)$ and any $f\in V'$,  
\be\label{eq: intermediate range}
\|\chi*f\|_2\ll_\Gamma |G_1|^{-\alpha'_\Gamma/2}\|\chi\|_2\|f\|_2. 
\ee
We will show the claim in this case holds if $\alpha\leq \alpha'_\Gamma/(4D')$. 
Recall that $\delta=|G_1|^{-D'}$, and let $\ell\leq E_1'|\log\delta|$ be so that the condition of the lemma holds for this choice of $\delta$.
Then applying~\eqref{eq: intermediate range} with $\chi=(\mu^{(\ell)})_\delta$ and $f\in V'$, 
we have
\be\label{eq: contraction for V'}
\begin{aligned}
\|\mu^{(\ell)}*f\|&\leq \delta^{1/2}\|f\|_2 + \|\mu^{(\ell)}*P_\delta*f\|_2 \\
&\ll \delta^{1/2}\|f\|_2+|G_1|^{-\alpha'_\Gamma/2}\delta^{-\alpha}\|f\|_2\leq 2\delta^{\alpha'_\Gamma/(4D')}\|f\|_2,
\end{aligned}
\ee
where we used~\eqref{eq: P eta almost inv} in the first inequality and~\eqref{eq: intermediate range} in the second. 

Now,~\eqref{eq: contraction for V'} implies that 
\[
E_1'|\log\delta|\mathcal L(\mu;V')\geq \alpha'_\Gamma|\log\delta|/(8D').
\]
This, together with~\eqref{eq: large dimensional rep} and~\eqref{eq: rep which factor through G2} finish the proof.  
\end{proof}

In view of Lemma~\ref{lem: quasi randomness}, thus, Proposition~\ref{prop: spectral indep exceptional} in this case will follow from the following. 
 
\begin{lem}\label{lem: l2 flattening}
There exists $E_1'$ so that for all $0<\eta\leq |G_1|^{-1}$, there is some $\ell\leq E_1'|\log\eta|$ with 
\[
\|(\mu^{(\ell)})_\eta\|_2\leq \eta^{-\alpha}, 
\]
where $\alpha$ is as in Lemma~\ref{lem: quasi randomness}. 
\end{lem}

\begin{proof}
The proof of the lemma, which will be completed in several steps, is similar to the proof of Lemma~\ref{lem: sigma m0 equidist}.  
Let $E_1$ be as in~\eqref{eq: marginals close to equi'} and let $\ell_1\geq E_1|\log\eta|$. 
We will show that the claim holds with some $E'_1= 2^mE_1$ for some $m$.  
Let $\sigma=\mu^{(\ell_1)}$, and for every nonnegative integer $m$, let $\sigma_m=\sigma^{(2^m)}$. 

Let 
$\bar E$, $R_1$ and $\vare_0(\Omega)$ as in Theorem~\ref{thm: SG II bdd generation}, and $E_2=C/c$ with $C$ and $c$ as in Proposition~\ref{prop:conjugation-by-large-ball-multiplication-large-image}.
Throughout the argument, we will assume $p_{\nu_1}\geq R_1^{1/\vare^{\bar E}}$, where   
\be\label{eq: choose vare again}
\vare=\vare_0(\Omega)\alpha/(8dE_2).
\ee
We will also assume that $\eta^\vare$ is smaller than various constants which depend only on $d_0$, as needed.

\subsection*{Initial entropy and the range of $\eta$}
Arguing as in the proof of Lemma~\ref{lem:coupling-entropy-lower-bound}, with $\sigma$
for $\nu$ and $0<\eta\leq |G_1|^{-1}$, the estimate in~\eqref{eq: marginals close to equi'} implies that 
\be\label{eq: infty norm of mu ell}
\|\sigma_\eta\|_\infty\leq |G_1|. 
\ee
Now~\eqref{eq: infty norm of mu ell} implies that if $\eta\leq |G_1|^{-1/\alpha}$, then 
\be\label{eq: intermediate range of eta}
 \|\sigma_\eta\|_2\leq \eta^{-\alpha}.
\ee

In particular,~\eqref{eq: intermediate range of eta} implies the lemma with $\ell_1$ so long as $\eta\leq |G_1|^{-1/\alpha}$.
For the rest of the argument, thus, we assume 
\be\label{eq: range of eta}
|G_1|^{-1/\alpha}\leq \eta\leq |G_1|^{-1}.
\ee

We also recall that  
\be\label{eq: triangle inequality and doubling}
\frac{1}{C_1'}P_{\eta}*\sigma_m\leq (\sigma_\eta)^{(2^m)}\leq C_1'P_{2^m\eta}*\sigma_m
\ee
where $C_1'\geq 1$ depends only on $C_1$ and $d_0$. 
This and~\eqref{eq: infty norm of mu ell} imply that 
\be\label{eq: L2 norm of sigma m real p adic}
\|(\sigma_m)_\eta\|\leq |G_1|\leq \eta^{-1}\qquad\text{for all $m$.}
\ee

Let $\delta$ and $R_2\geq 4$ be as in Theorem~\ref{thm: SG II bdd generation} applied with $\vare$.
Apply~\cite[Theorem 2.8]{MMSG} with $\vare$ as above (for $\vare$ in loc.\ cit.); replacing $\delta$ with a smaller constant, if necessary, we assume that $\delta$ also satisfies the claim in~\cite[Theorem 2.8]{MMSG} with this $\vare$.   
  
Let $m_2$ be so that $2^{m_2}E_1>1/\delta$. Then 
\be\label{eq: the choice of m1'}
2^{m_2}\ell_1\geq 2^{m_2}E_1|\log \eta|\geq |\log \eta|/\delta\geq \log |G_1|/\delta. 
\ee

\subsection*{Flattening lemma}
Recall from~\eqref{eq: dim cond eta small} that in the range~\eqref{eq: range of eta}, $G=G_1\times G_2$ satisfies 
dimension condition with $2C_1$ and some $d'$ satisfying $d_0\leq d'\leq d_0+1$. Thus~\cite[Theorem 2.12]{MMSG} 
is applicable with $2C_1$ and $d'$ for $G$ and all $\eta$ in this range.  

Assume now that for some $m\geq m_2$, we have  
\be\label{eq: sigma m does not flatten'}
\|(\sigma_m*\sigma_m)_\eta\|_2> \eta^\kappa\|(\sigma_m)_\eta\|_2
\ee
for some $\kappa>0$, which is explicated in~\eqref{eq: fix kappa'}. 
Then, by~\cite[Theorem 2.12]{MMSG} there exists $E$ depending on $C_1$ and $d_0$, 
and $H\subseteq G$ which is symmetric and contains $1$, so that all of the following hold: 
\begin{enumerate}[label={(H-\roman*)}]
\item\label{H-i} $\eta^{E\kappa}\|(\sigma_m)_\eta\|_2^{-2}\leq |H_\eta|\leq \eta^{-E\kappa}\|(\sigma_m)_\eta\|_2^{-2}$, 
\medskip
\item\label{H-ii} $H\cdot H\subset T\cdot H$ where $T\subset H\cdot H$ satisfies $\#T\leq \eta^{-E\kappa}$, and  
\medskip
\item\label{H-iii} $\sigma_m*\sigma_m(H_{3\eta})\geq \eta^{E\kappa}$.
\end{enumerate}

We will apply the above with $\kappa=(\alpha\min\{\vare,\delta\})/(10R_2Ed_0)$. This in particular implies   
\be\label{eq: fix kappa'}
\eta^{E\kappa}\geq \eta^{\delta\alpha/(2d_0)}\geq |G_1|^{-\delta},
\ee 
where in the last inequality we used~\eqref{eq: range of eta}.

\subsection*{Bounded generation and the structure of $H$} 
Let $\vare$ be as in~\eqref{eq: choose vare again}, and let $A\subseteq \Gamma$ be a finite symmetric lift of $H$. 
Then by~\ref{H-iii} and~\eqref{eq: fix kappa'}, we have 
\be\label{eq: mathcal P of A real p adic}
\mathcal P_\Omega^{(2^{m+1}\ell_1)}(A)\geq \eta^{E\kappa}\geq |G_1|^{-\delta}
\ee
Apply Theorem~\ref{thm: SG II bdd generation} with $\vare$ and $A$. 
Since $\pi_{\nu_1,n_1}(A)=\pi_1(H)$,~\eqref{eq: mathcal P of A real p adic} and Theorem~\ref{thm: SG II bdd generation} 
imply that there exists $0\leq n'_1\leq n_1\vare$ so that if we put $H_1=\prod_{R_2} H_{4\eta}$, then 
$\pi_{\nu_1,n_1}(\Gamma_{\nu_1,n'_1})\subseteq \pi_1(H_1)$. 
Recall from~\eqref{eq: n1 is less that d} that $n_1\leq d+1$. Hence, $\vare n_1\leq \vare (d+1)<1$, and we conclude $n'_1=0$. Altogether,  
\be\label{eq: pi-1 of H1 real p-adic}
\pi_1(H_1)=G_1
\ee

Let $Y=(\pi_2(H_{3\eta}))_\eta$, then $|Y|\asymp |\pi_2(H_{3\eta})|$, see~\eqref{eq: triangle inequality and doubling}. Moreover, by~\eqref{eq: marginals close to equi'}, we have 
\[
\bigl|(\pi_{2*}\sigma_{m+1})(Y)-|Y|\bigr|\leq \eta |Y|^{1/2}\leq \eta
\] 
Using~\eqref{eq: fix kappa'}, this and~\ref{H-iii} imply that 
\[
|\pi_2(H_{3\eta})|\geq \eta^{2E\kappa}\geq \eta^{\delta/d_0},
\]
we also used $|Y|\asymp |\pi_2(H_1)|$ and assume $\eta$ is small. 

Since $R_2\geq 4$, the above and~\cite[Theorem 2.8]{MMSG}, recall also the choice of $\delta$, imply that  
\be\label{eq: pi-2 of H1 real p-adic}
\pi_2(H_1)\supseteq 1_{\eta^\vare}^{(2)}.
\ee

As it was done in~\eqref{eq:largest-elements-in-factors}, let
\[
\textstyle\beta_2:=\inf\{\beta\in [0,1]|\; \exists g_2\in G_2, d(g_2,1^{(2)})\geq \eta^{\beta},\, (1^{(1)},g_2)\in \prod_3 (H_1)_\eta\};
\]
where $d$ denotes our fixed bi-invariant metric on $G_2$.  

We claim 
\be\label{eq: beta2 and vare}
\beta_2\leq 10\vare.
\ee
Let us first assume~\eqref{eq: beta2 and vare} and complete the proof of the lemma.

\subsection*{Large fibers for $H_1$}
In view of~\eqref{eq: beta2 and vare}, 
\be\label{eq: element with small level'}
\text{There exists $h\in G_2$ with $d(1, h)\geq \eta^{10\vare}$ so that $(1,h)\in \textstyle\prod_3(H_1)_\eta$}
\ee

Apply Proposition~\ref{prop:conjugation-by-large-ball-multiplication-large-image} with the group 
$G_2$, $h$ as in~\eqref{eq: element with small level'}, and $\rho=\eta^{10\vare}$. 
Then by Proposition~\ref{prop:conjugation-by-large-ball-multiplication-large-image} combined with ~\eqref{eq: element with small level'} and~\eqref{eq: pi-2 of H1 real p-adic}, there exists $E_2$ depending only on $d_0$ (indeed $E_2=C/c$) 
so that 
\be\label{eq: large fiber for H1'}
\bigl\{(1, g_2): g_2\in 1^{(2)}_{\eta^{E_2\vare}}\bigr\}\subset \textstyle\prod_{7d_0} (H_1)_\eta. 
\ee

\subsection*{The conclusion of the proof} 
Recall now from~\eqref{eq: pi-1 of H1 real p-adic} that $G_1=\pi_1(H_1)$. 
This and~\eqref{eq: large fiber for H1'} imply that there exists some $t$, again depending only on $d_0$, so that 
\be\label{eq: H1 is large'}
\Bigl|\textstyle\prod_{R_2(1+7d_0)} H_{t\eta}\Bigr|\geq m_{G_2}\bigl(1^{(2)}_{\eta^{E_2\vare}}\bigr)\geq C_1^{-1}\eta^{E_2d_0\vare}
\ee
where we also used $H_1=\prod_{R_2}H_{4\eta}$ and the triangle inequality.

Then using~\ref{H-i} and~\ref{H-ii}, we have 
\[
\textstyle
\
\eta^{-E\kappa(1+R_2(1+7d_0))}\|(\sigma_m)_\eta\|_2^{-2}\geq \eta^{-E\kappa(R_2(1+7\dim \bbg))}|H_\eta|\geq c'_1 \Bigl|\textstyle\prod_{R_2(1+7d_0)} H_{t\eta}\Bigr| 
\]
where $c'_1$ depends only on $d_0$. 

Combining this with~\eqref{eq: H1 is large'}, we get 
\[
\eta^{-E\kappa(1+R_2(1+7d_0))}\|(\sigma_m)_\eta\|_2^{-2}=\eta^{-\kappa(E+R_2E(1+7d_0))}\|(\sigma_m)_\eta\|_2^{-2}\geq c'_1C_1^{-1}\eta^{E_2d_0\vare}
\]
Recall that $\kappa=(\alpha\min\{\vare,\delta\})/(10R_2Ed_0)$. 
In consequence, 
\[
\|(\sigma_m)_\eta\|_2^{-2}\geq c'_1C_1^{-1}\eta^{E_2d_0\vare} \eta^{8R_2Ed_0\kappa}\geq \eta^{2E_2d_0\vare}.
\]
This and the choice of $\vare$, see~\eqref{eq: choose vare again}, imply that  
\be\label{eq: L2 norm of sigma m eta}
\|(\sigma_m)_\eta\|_2\leq \eta^{-E_2d_0\vare}\leq \eta^{-\alpha/4}.
\ee 
Altogether, 
\be\label{eq: sigma m does not flatten' 1}
\text{for all $m\geq m_2$ either $\|\sigma_m*\sigma_m\|_2\leq \eta^{E\kappa}\|\sigma_m\|_2$ or~\eqref{eq: L2 norm of sigma m eta} holds for $\sigma_m$}
\ee
see~\eqref{eq: sigma m does not flatten'}. 

Since $\|(\sigma_m)_\delta\|_2\leq \eta^{-1}$ for all $m$, see~\eqref{eq: L2 norm of sigma m real p adic}, 
we conclude from~\eqref{eq: sigma m does not flatten' 1} that there exists $m_2\leq m_1\ll_\Omega 1$ 
so that~\eqref{eq: L2 norm of sigma m eta} holds. This completes the proof of the lemma assuming~\eqref{eq: beta2 and vare}. 
\end{proof}

\proof[Proof of~\eqref{eq: beta2 and vare}]
In view of Lemma~\ref{lem:constructing-approx-hom} applied with $H_1$, and using~\eqref{eq: pi-1 of H1 real p-adic} 
and~\eqref{eq: pi-2 of H1 real p-adic}, there exists 
\[
f:G_1\rightarrow G_2\qquad\text{with}\qquad  f(1^{(1)})=1^{(2)}
\]
so that all the following hold
\begin{subequations}
\begin{align}
\label{eq: almost hom thm 2 real p adic 1}&d(f(g_1)f(g_1'),f(g_1g_1'))\leq \eta^{\beta_2}&&\text{for all $g_1, g'_1\in G_1$,}\\
\label{eq: almost hom thm 2 real p adic 2}&d(f(g_1^{-1}),f(g_1)^{-1})\leq \eta^{\beta_2}&&\text{for all $g_1\in G_1$, and}\\
\label{eq: image of f thm 2 real p adic}&1_{\eta^\vare}^{(2)}\subseteq (\im f)_{\eta^{\beta_2}}
\end{align}
\end{subequations}

Assume contrary to the claim that $\beta_2>10\vare$. 
We will now use a simpler version of the argument in \S\ref{sec: target Lie} in the case 
$F_2=\bbr$ and $F_1=\bbq_p$ to get a contradiction. 

Using~\eqref{eq: image of f thm 2 real p adic} and $\beta_2>10\vare$, there exists $g_1\in G_1$ so that 
$f(g_1)\in 1_{\eta^\vare}^{(2)} \setminus 1_{\eta^\vare/2}^{(2)}$.
Recall now from~\cite{SG-super-approx-II} that there exists some $C$ depending only on $d_0$ so that  
\[
\langle g_1\rangle \subseteq \textstyle\prod_C\{h[g_1,h']h^{-1}: h,h'\in G_1\},
\]
This, the bi-invariance of the metric, and the triangle inequality imply that    
\be\label{eq: f g1k thm 2}
d(f(g_1^k),1^{(2)})\leq C'(\eta^{\vare}+\eta^{\beta_2})\leq  \eta^{\vare/2} \qquad\text{for all $k$}
\ee
where $C'$ depends only on $d_0$, see the proof of Lemma~\ref{lem:image-large-ball-under-approx-hom}.

Now sing~\eqref{eq: almost hom thm 2 real p adic 1}, ~\eqref{eq: almost hom thm 2 real p adic 2}, and~\eqref{eq: f g1k thm 2}, 
for all $\eta^{-\vare}/100\leq k\leq \eta^{-\beta_2/2}$, we have  
\[
d(f(g_1)^k, 1^{(2)})\leq d(f(g_1^k), 1^{(2)})+\vare^{\beta_2/2}\leq \eta^{\vare/3}.
\] 
However, if we write $f(g_1)=\exp(z)$ where $z\in\mathfrak g_2$ satisfies $\eta^\vare/2\leq \|z\|\leq \eta^\vare$, 
then there exists some $k$ in the above range so that 
\[
d(f(g_1)^k, 1^{(2)})=d(\exp(kz), 1^{(2)})\gg 1
\] 
This contradiction finishes the proof.   
\qed

\begin{proof}[Proof of Proposition~\ref{prop: spectral indep exceptional}]
We first assume that both $\nu_1$ and $\nu_2$ are finite places. 

If we have
\be\label{eq: n2 is not large}
\max\{n_1, n_2\}\leq 4d^2(d+1)E_2/(\vare_0(\Omega)\alpha_\Gamma),
\ee
the proposition follows from~\cite[Theorem 1]{SG-super-approx-II}, with $M_0$ in that theorem equal to 
$\frac{4d^2(d+1)E_2}{\vare_0(\Omega)\alpha_\Gamma}$. If~\eqref{eq: n2 is not large} fails on the other hand, the proposition follows from Lemma~\ref{lem: sigma m0 equidist}. 

Altogether, the proof is complete if $\nu_1,\nu_2\in V_{{\rm f},\Gamma}$.  

The case $\nu_2=\infty$ and $\nu_1\in V_{{\rm f},\Gamma}$ follows from the discussion in \S\ref{sec: exceptional finite and infinite}, see in particular, Lemma~\ref{lem: quasi randomness} and Lemma~\ref{lem: l2 flattening}. 
\end{proof}

\section{Proof of Theorem~\ref{thm:Q-form}}\label{sec: proof of thm Q-form}

We now begin the proof of Theorem~\ref{thm:Q-form}. As it was mentioned before, the proof relies on Theorem~\ref{thm:main1}, Theorem~\ref{thm:functions-scale-spec-gap} and Proposition~\ref{prop: spectral indep exceptional}. 

\begin{proof}[Proof of Theorem~\ref{thm:Q-form}]
Recall that we fixed a $\bbq$-embedding $\bbg \subseteq(\SL_N)_{\bbq}$; our constants are allowed to depend on this embedding. 

Step 1: By the strong approximation theorem, there exists a finite index subgroup $\Lambda \subset \Gamma$ so that 
\be\label{lem: finite index sgrp Lambda}
\Lambda_{\nu_1,\nu_2}=\Lambda_{\nu_1}\times \Lambda_{\nu_2}\qquad\text{for all $\nu_1,\nu_2\in V_\Gamma$,}
\ee
where $\Lambda_\nu$ denotes the closure of $\Lambda$ in $\bbg(\bbq_\nu)$. In view of Proposition~\ref{prop:passing-to-open-subgroup}, it suffices to prove the theorem for $\Lambda$. Thus, we assume~\eqref{lem: finite index sgrp Lambda} holds for $\Gamma$.

Step 2: There is a finite set of places $E_\Gamma$ so that for all $\nu\not\in E_\Gamma$, we have 
\[
\Gamma\subseteq \SL_N(\bbz_\nu).
\]
For every finite place $\nu\in  E_\Gamma$, there exists $g_{\nu}\in \GL_N(\bbq_\nu)$ such that 
\[
g_{\nu}\Gamma g_{\nu}^{-1}\subseteq \SL_N(\bbz_{\nu}).
\]
In view of the strong approximation theorem, there exists $g\in \PGL_N(\bbq)$ so that $g\in \PGL_N(\bbz_\nu)$ for all $\nu\not\in E_\Gamma$ and $g^{-1}g_\nu\in \PGL_N(\bbz_\nu)$ for every finite place $\nu\in E_\Gamma$. Therefore, 
\be\label{eq: conjugate into SLN Zp}
g\Gamma g^{-1}\subseteq \SL_N(\bbz_{\nu})\quad\text{for all non-archimedean places $\nu$.}
\ee
In view of~\eqref{eq: conjugate into SLN Zp}, we will assume 
\be\label{eq: conjugate into SLN Zp'}
\Gamma \subseteq \SL_N(\bbz_{\nu})\quad\text{for all non-archimedean places $\nu$.}
\ee
for the rest of the argument.

Step 3: If the archimedean place belongs to $V_\Gamma$, i.e., $\bbg(\bbr)$ is compact,  fix some $g\in \SL_N(\bbr)$ so that 
\[
g \bbg(\bbr)g^{-1}\subseteq \SO_N(\bbr). 
\]
In this case, we replace $\bbg(\bbr)$ by $g \bbg(\bbr)g^{-1}$ and we work with the corresponding copy of $\Gamma$ in the archimedean place. This means, for every non-archimedean place $\nu$, we are replacing $\Gamma_{\infty, \nu}$ with 
\[
(g,1)\Gamma_{\infty, \nu}(g^{-1}, 1)\subseteq \SO_N(\bbr)\times \SL_N(\bbz_\nu).
\]
Note that $\bbg(\bbq_\nu)$ is unchanged for all non-archimedean places $\nu$.

Step 4: Let $\nu_1,\nu_2\in V_\Gamma$. In view of Lemma~\ref{lem: locally random family}, we have $\Gamma_{\nu_1,\nu_2}$ is $L$-locally random with coefficient $C_0$, where the parameters $L$ and $C_0$ depend only on $\Gamma$. 

Let $\mu$ be the law of $X$. By Claim~\ref{claim: proof 1} in the proof of Theorem~\ref{thm:main1}, applied with $G_1:=\Gamma_{\nu_1}$ and $G_2:=\Gamma_{\nu_2}$, we have   
\be\label{eq: scale rho func'}
\|\mu^{(\ell)}\ast f\|_2\leq \eta^b\|f\|_2,
\ee
for all functions $f$ which live at scale $\eta\leq \eta_0$ and $\ell\leq \bar C\log(1/\eta)$, where $\bar C$ and $b$ depend only on $\dim \bbg$, $L$, and 
$\eta_0=\max\{p_{\nu_1}, p_{\nu_2}\}^{-O_{\bbg}(1)}$.

In view of~\eqref{eq: scale rho func'} and Theorem~\ref{thm:functions-scale-spec-gap}, applied with $G=\Gamma_{\nu_1,\nu_2}$, there exists a subrepresentation $\cal_0:=\cal_{\nu_1,\nu_2, 0}\subset L^2_0(\Gamma_{\nu_1,\nu_2})$ with  

\begin{subequations}
\begin{align}
\label{eq:dim hcal0}&\dim\cal_{0}\le \eta_0^{-c_0(\Gamma)}\eta_0^{-O_{\bbg}(1)}\quad\text{and} \\
\label{eq:spectral gap in complement of hcal0}&\lcal(X; L^2_0(\Gamma_{\nu_1,\nu_2})\ominus \cal_{0})\ge \frac{b}{\bar C}.
\end{align}
\end{subequations}

Step 5: We now investigate $\mathcal L(X; \mathcal H_0)$. In view of~\eqref{eq:dim hcal0} and~\cite[Proposition 33]{Golsefidy-IMRN-SA}, the representation $\mathcal H_0$ of $\Gamma_{\nu_1,\nu_2}$ factors through $G_1\times G_2$ where 
\[
G_i=\Gamma_{\nu_i}/\Gamma_{\nu_i,n_i}
\]
and the number of connected components of $G_1\times G_2$ is $\leq \eta_0^{-O_\Gamma(1)}\eta_0^{-O_{\bbg}(1)}$, see \S\ref{sec: spec indep and exceptional} for the notation. In particular, condition~\eqref{eq: dimension bound} of Proposition~\ref{prop: spectral indep exceptional} is satisfied. Hence, by Proposition~\ref{prop: spectral indep exceptional}, we have   
\[
\mathcal L(X; \mathcal H_0)\geq \mathcal L(X; G_1\times G_2)> \varrho_\Gamma>0.
\]
This and~\eqref{eq:spectral gap in complement of hcal0} complete the proof. 
\end{proof}

\appendix

\section{Passing to an open subgroup}\label{sec: finite index}
The main goal of this section is to show how we can study the spectral gap property of an action on a compact group by passing to an open subgroup and vice versa (see Proposition~\ref{prop:passing-to-open-subgroup}). Along the way, we review the connection between the spectral gap property and \emph{almost invariant} functions, and we also give the connection between $\lcal(\pcal_{\Omega_1})$ and $\lcal(\pcal_{\Omega_2})$ where $\Omega_1$ and $\Omega_2$ generate the same dense subgroup $\Gamma$ of $G$. 

We start by recalling the concept of \emph{almost invariant functions} with respect to a given finite symmetric set. For a compact group $G$, a finite symmetric subset $\Omega$ of $G$, and a non-zero function $f\in L^2(G)$, we let 
\[
\delta_{\Omega}(f):=\max_{w\in \Omega} \frac{\|w\cdot f-f\|_2}{\|f\|_2}.
\]
We let $\delta(\Omega):=\inf \{ \delta_{\Omega}(f) \mid f\in L^2(G)^\circ, \|f\|_2=1\}$ where $L^2(G)^\circ$ is the space of functions in $L^2(G)$ that are orthogonal to the constant functions on $G$. A function is called $\vare$\emph{-almost invariant with respect to $\Omega$} if $\delta_{\Omega}(f)\leq \vare$. It is well-known that $\lcal(\pcal_{\Omega})>0$ if and only if $\delta(\Omega)>0$. 
The next lemma is a quantitative version of this statement. 

\begin{lem}\label{lem:lyapunov-displacement}
Suppose $G$ is a compact group and $\Omega$ is a finite symmetric subset of $G$. Then 
\[
 \frac{1}{|\Omega|} \delta(\Omega)^2\ll \lcal^{\bullet}(\pcal_{\Omega})\ll \delta(\Omega),
\]
where $\lcal^\bullet(\pcal_{\Omega}):=\min\{\lcal(\pcal_{\Omega}),1\}$ and the implied constants are fixed positive numbers.
\end{lem}  
\begin{proof}
The lemma is proved in~\cite[Proposition I]{HRV}. See in particular, parts (4) and (6) in that proposition,
and note that since $\pcal_{\Omega}$ is self adjoint, its spectral radius equals its operator norm.
\end{proof}

In the next lemma, we show how the Lyapunov exponents with respect to two generating sets are related to each other. 
\begin{lem}\label{lem:lyapunov-two-generators}
Suppose $G$ is a compact group and $\Omega_1$ and $\Omega_2$ are two finite symmetric subsets of $G$ such that $\langle \Omega_1\rangle=\langle \Omega_2\rangle$ is dense in $G$. Suppose $k_0$ is a positive integer such that $\Omega_1\subseteq \prod_{k_0} \Omega_2$ and $\Omega_2\subseteq \prod_{k_0} \Omega_1$. Then 
\[
\frac{1}{k_0}\delta(\Omega_1)\leq \delta(\Omega_2)\leq k_0 \delta(\Omega_1),
\]
and 
\[
\frac{1}{k_0^2|\Omega_2|}\lcal^\bullet(\pcal_{\Omega_1})^2 
\ll
\lcal^\bullet(\pcal_{\Omega_2}) 
\ll
k_0|\Omega_1|^{1/2} \lcal^\bullet(\pcal_{\Omega_1})^{1/2},
\]
where $\lcal^\bullet(\pcal_{\Omega_i}):=\min\{\lcal(\pcal_{\Omega_i}),1\}$ and the implied constants are fixed positive numbers.
\end{lem}
\begin{proof}
Since $\Omega_2\subseteq \prod_{k_0}\Omega_1$, for every function $f\in L^2(G)^\circ$ we have
\[
\delta_{\Omega_2}(f)\leq k_0 \delta_{\Omega_1}(f).
\]
Hence $\delta(\Omega_2)\leq k_0\delta(\Omega_1)$. By symmetry, we have that $\delta(\Omega_1)\leq k_0\delta(\Omega_2)$. This completes proof of the first set of inequalities. 

By Lemma~\ref{lem:lyapunov-displacement} and the first set of inequalities, we obtain that 
\[
\lcal^\bullet(\pcal_{\Omega_2})\ll \delta(\Omega_2) \leq k_0 \delta(\Omega_1) \ll k_0 |\Omega_1|^{1/2} \lcal^\bullet(\pcal_{\Omega_1})^{1/2}.
\]
The proof follows by symmetry.
\end{proof}
\begin{lem}\label{lem:finite-group-extreme-displacement}
Suppose $F$ is a finite group. Then $\delta(F)\geq (2|F|^{-1})^{1/2}$.
\end{lem}
\begin{proof}
Suppose $f\in L^2(F)^\circ$. Then 
\begin{align}
\notag
\sum_{x\in F} \|x\cdot f-f\|_2^2=
&
\sum_{x,y\in F}|f(x^{-1}y)-f(y)|^2 
\\
\notag 
= & \sum_{x,y\in F}
\big(
|f(x)|^2+|f(y)|^2-2  {\rm Re}(f(x)\overline{f(y)})
\big)
\\
\label{eq:sum-displacement-finite-groups}
= &
2 \|f\|_2^2+2|\sum_{x\in F}f(x)|^2=2 \|f\|_2^2.
\end{align}
By \eqref{eq:sum-displacement-finite-groups}, we obtain that 
$
\delta_F(f)\geq (2|F|^{-1})^{1/2},
$
which completes the proof.
\end{proof}
The next lemma is a well-known result which provides us with a generating set with interesting properties for a subgroup of finite index in a finitely generated group.
\begin{lem}\label{lem:generating-set-subgroup-finite-index}
	Suppose $G$ is a compact group, $\Gamma$ is a finitely generated dense subgroup of $G$, and $H$ is an open subgroup of $G$. Suppose $\Omega$ is a finite symmetric generating set of $\Gamma$ which intersects all the cosets of $H$ in $G$. Let $s:G/H\rightarrow \Omega$ be a section; that means $s(xH)\in xH$ for every $x\in G$. Then $\Omega_H:=\overline{\Omega}_H\cup \overline{\Omega}_H^{-1}$ is a symmetric generating set of $\Gamma\cap H$ where
	\[
	\overline{\Omega}_H:=\{s(w_1w_2H)^{-1}w_1w_2\mid w_1,w_2\in \Omega\}.
	\]
	Moreover for every $x\in \Gamma\cap H$ we have
	\[
	\frac{1}{3}\ell_{\Omega}(x)\leq \ell_{\Omega_H}(x)\leq \ell_{\Omega}(x)
	\]
	where $\ell_{\Omega}(x)$ is the least non-negative integer $k$ such that $x\in \prod_k \Omega$ and $\ell_{\Omega_H}(x)$ is defined in a similar way.
\end{lem}
\begin{proof}
Since every element of $\Omega_H$ is a product of at most three elements of $\Omega$, we have that $\ell_{\Omega}(x)\leq 3\ell_{\Omega_H}(x)$ for every $x$ in the group generated by $\Omega_H$. Next  by induction on $\ell_{\Omega}(x)$, we prove that $\ell_{\Omega_H}(x)\leq \ell_{\Omega}(x)$ for every $x\in \Gamma\cap H$; this, in particular, implies that $\Omega_H$ generates $\Gamma\cap H$. The base case of $\ell_{\Omega}(x)=0$ is clear. Suppose $\ell_{\Omega}(x)=n$ for some $x\in \Gamma\cap H$. Then $x=w_1\ldots w_n$ where $w_i$'s are in $\Omega$. Let $x':=w_1\cdots w_{n-2} s(w_{n-1}w_n H)$ and $w':=s(w_{n-1}w_n H)^{-1}w_{n-1}w_{n}$. Then $w'\in \Omega_H$, $\ell_{\Omega}(x')\leq n-1$, and $x=x'w'$. Since $x$ and $w'$ are in $\Gamma\cap H$, so is $x'$. Therefore, by the induction hypothesis, we have $\ell_{\Omega_H}(x')\leq n-1$. Hence $\ell_{\Omega_H}(x)=\ell_{\Omega_H}(x'w')\leq \ell_{\Omega_H}(x')+1\leq n$. This completes the proof.  
\end{proof}

\begin{proposition}\label{prop:passing-to-open-subgroup}
Let $ G$ be a compact group and $\Gamma$ be finitely generated dense subgroup of $G$. Let 
$H$ be an open subgroup of $G$. Then $\Gamma \acts G$ has spectral gap if and only if 
$\Gamma \cap H \acts H$ has spectral gap.  More precisely, if $\Gamma$ has a generating set of size $n$, then $\Gamma$ has a finite symmetric generating set $\Omega$ and $\Gamma\cap H$ has a finite symmetric generating set $\Omega'$ such that $|\Omega|,|\Omega'|\ll_{n,[G:H]} 1$ and 
\[
\lcal^\bullet(\pcal_{\Omega})^8\ll_{n,[G:H]} \lcal^\bullet(\pcal_{\Omega'}) \ll_{n,[G:H]} \lcal^\bullet(\pcal_{\Omega})^{1/8},
\]
where $\lcal^\bullet(\cdot)=\min\{\lcal(\cdot),1 \}$ and $\lcal(\pcal_{\Omega'})$ is defined with respect to a random walk in the compact group $H$. 
\end{proposition}
\begin{proof}
Since $H$ is an open subgroup and $\Gamma$ is a dense subgroup of $G$, every coset of $H$ has a representative in $\Gamma$. Let $s:G/H\rightarrow \Gamma$ be a section such that $s(H)=1$. Let $\Omega$ be a finite symmetric generating set of $\Gamma$ which contains the image of $s$ as a subset. Let $\Omega_H$ be as in Lemma~\ref{lem:generating-set-subgroup-finite-index} with respect to $\Omega$ and section $s$.  

We start with the case that $H$ is a normal subgroup of $G$. 

\medskip

\textbf{Claim 1}.\ \emph{In the above setting, 
$\lcal^\bullet(\pcal_{\Omega})\gg \frac{1}{|\Omega|[G:H]}
\lcal^\bullet(\pcal_{\Omega_H})^2$ where the implied constant is a fixed positive number.}

\begin{proof}[Proof of Claim 1]  If $\lcal(\pcal_{\Omega_H})=0$, there is nothing to prove. Thus assume that $c_H:=\lcal(\pcal_{\Omega_H})>0$.  
We will estimate $\delta(\Omega)$ from below in terms of $c_H$, the claim will then follow from Lemma~\ref{lem:lyapunov-displacement}. 
 
To that end, suppose $f\in L^2(G)$ and $\|f\|_2=1$. 
Let $f_H$ denote the projection of $f$ into the space of $H$ invariant functions $L^2(G)^H$, then 
\be\label{eq:orthogonal-to-h-invariant-contraction}
\|\pcal_{\Omega_H}\ast f-f_H\|_2^2\leq 2^{-2c_H} \|f-f_H\|_2^2,
\ee 

We include the proof of~\eqref{eq:orthogonal-to-h-invariant-contraction} for completeness. 
Let us first recall the definition of $f_H$: for every coset $\overline{x}:=xH$, let $f_x:=f \bb1_{\overline{x}}$ where $\bb1_{\overline{x}}$ is the characteristic function of the coset $\overline{x}$. Then 
 \be\label{eq:norm-in-terms-cosets}
f=\sum_{\overline{x}\in G/H} f_{\overline{x}} \quad\text{and $f_{\overline{x}}$'s are pairwise orthogonal}
\ee
Similarly, since for every $\overline{x}\in G/H$ and $y\in H$, $(y\cdot f-f)\bb1_{\overline{x}}=y\cdot f_{\overline{x}}-f_{\overline{x}}$, we have  
\be\label{eq:norm-displace-with-cosets}
y\cdot f-f=\sum_{\overline{x}\in G/H} y\cdot f_{\overline{x}}-f_{\overline{x}}
\ee
For every $\overline{x}=xH\in G/H$, let $a_{\overline{x}}:=\int_G f_{\overline{x}}(z) dz$. Then $f_{\overline{x}}-a_{\overline{x}} \bb1_{\overline{x}}$ is in the space $L^2(Hx)^{\circ}$ of functions on $Hx=xH$ that are orthogonal to the constant functions; moreover 
\[
f_H=\sum_{\overline{x}\in G/H} a_{\overline{x}}\bb1_{\overline{x}}.
\]

In particular, we have  
\begin{subequations}
\begin{align}
\label{eq:spec-gap-H-contraction} &\|f\|_2^2=\|f_H\|_2^2+\|f-f_H\|_2^2\quad\text{and}\\
\label{eq:spec-gap-H-contraction'} &\|\pcal_{\Omega_H}\ast (f_{\overline{x}}-a_{\overline{x}} \bb1_{\overline{x}})\|_2\leq 2^{-c_H} \|f_{\overline{x}}-a_{\overline{x}} \bb1_{\overline{x}}\|_2
\end{align}
\end{subequations}
for every $\overline{x}\in G/H$. By \eqref{eq:norm-in-terms-cosets}, \eqref{eq:norm-displace-with-cosets}, and \eqref{eq:spec-gap-H-contraction'}, we obtain that 
\[
\|\pcal_{\Omega_H}\ast f-f_H\|_2^2=\|\pcal_{\Omega_H}\ast (f-f_H)\|_2^2=
\sum_{\overline{x}\in G/H} \|\pcal_{\Omega_H} \ast (f_{\overline{x}}-a_{\overline{x}} \bb1_{\overline{x}})\|_2^2
\leq 2^{-2c_H} \|f-f_H\|_2^2,
\]
where we also used $\pcal_{\Omega_H}\ast f_H=f_H$ in the first equality. 
As it was claimed in~\eqref{eq:orthogonal-to-h-invariant-contraction}. 

We will now use~\eqref{eq:orthogonal-to-h-invariant-contraction} to bound 
$\|w\cdot f_H-f_H\|_2$, for $w\in\Omega$, in terms on $c_H$ and $\delta_\Omega(f)$, see~\eqref{eq:projection-almost-invariant} below.   
Indeed, since $\Omega_H$ is a subset of the three fold product of $\Omega$, for every $y\in \Omega_H$ we have
\be\label{eq:almost-invariant}
\|y\cdot f-f\|_2\leq 3 \delta_{\Omega}(f).
\ee
The estimate \eqref{eq:almost-invariant} implies that  
\be\label{eq:average-stays-close-by}
\|\pcal_{\Omega_H}\ast f-f\|_2=\|\sum_{y\in \Omega_H} \pcal_{\Omega_H}(y) (y\cdot f-f)\|_2\leq 3 \delta_{\Omega}(f).
\ee
Now \eqref{eq:orthogonal-to-h-invariant-contraction} and \eqref{eq:average-stays-close-by} give $\|f-f_H\|_2\leq 2^{-c_H}\|f-f_H\|_2+3\delta_{\Omega}(f)$, which we write as
\be\label{eq:almost-h-invariant}
\|f-f_H\|_2\leq 3(1-2^{-c_H})^{-1}\delta_{\Omega}(f).
\ee
By \eqref{eq:almost-h-invariant} and the definition of $\delta_{\Omega}(f)$, we conclude that for every $w\in \Omega$ the following holds
\be\label{eq:projection-almost-invariant}
\|w\cdot f_H-f_H\|_2\leq (6(1-2^{-c_H})^{-1}+1) \delta_{\Omega}(f).
\ee

We find a lower bound for $\sup_\Omega\|w\cdot f_H-f_H\|_2$. Indeed, $L^2(G)^H$ can be isometrically identified with $L^2(G/H)$. 
This way we view $f_H$ as an element of $L^2(G/H)$. Since $f$ is orthogonal to constant functions on $G$, $f_H$ is orthogonal to constant functions on $G/H$. Therefore, by Lemma~\ref{lem:finite-group-extreme-displacement}, we 
obtain that there is $\overline{x}_0\in G/H$ such that 
\be\label{eq:max-displacement-finite-groups}
\|\overline{x_0}\cdot f_H-f_H\|_2\geq (2[G:H]^{-1})^{1/2}\|f_H\|_2.
\ee

Since the image of $s$ is in $\Omega$,~\eqref{eq:max-displacement-finite-groups} and~\eqref{eq:projection-almost-invariant} imply 
\be\label{eq:max-displacement-finite-groups'}
(2[G:H]^{-1})^{1/2}\|f_H\|_2\leq \|\overline{x_0}\cdot f_H-f_H\|_2\leq (6(1-2^{-c_H})^{-1}+1) \delta_{\Omega}(f).
\ee
Moreover, we have 
\be\label{eq:max-displacement-finite-groups''}
\|f_H\|_2\geq \|f\|_2-\|f-f_H\|_2=1-\|f-f_H\|_2\geq 1- 3(1-2^{-c_H})^{-1}\delta_\Omega(f),
\ee
where in the last estimate we used~\eqref{eq:almost-h-invariant}.
Using~\eqref{eq:max-displacement-finite-groups''} in~\eqref{eq:max-displacement-finite-groups'}, we conclude 
\[
(2[G:H]^{-1})^{1/2}\bigl(1- 3(1-2^{-c_H})^{-1}\delta_\Omega(f)\bigr)\leq (6(1-2^{-c_H})^{-1}+1) \delta_{\Omega}(f).
\]
This gives 
\be\label{eq:lower-bound-displacement-f-1}
\delta_{\Omega}(f)\geq (12[G:H]^{1/2})^{-1}(1-2^{-c_H}).
\ee
Using \eqref{eq:lower-bound-displacement-f-1} and the fact that $1-e^{-x}\geq \frac13 \min\{x,1\}$ for every positive $x$, we conclude that 
\be\label{eq:lower-bound-displacement-f}
\delta(\Omega)\gg [G:H]^{-1/2} \lcal^\bullet(\pcal_{\Omega_H}).
\ee
By \eqref{eq:lower-bound-displacement-f} and Lemma~\ref{lem:lyapunov-displacement}, we deduce that 
\[
\lcal(\pcal_{\Omega})\gg \frac{1}{|\Omega|[G:H]}\lcal^\bullet(\pcal_{\Omega_H})^2,
\]
which completes proof of Claim 1.
\end{proof}

\textbf{Claim 2.} \emph{In the above setting $\lcal^\bullet(\pcal_{\Omega_H})\gg \frac{1}{|\Omega|^2} \lcal^\bullet(\pcal_{\Omega})^2$ where the implied constant is a fixed number.}

\begin{proof}[Proof of Claim 2] If $\lcal(\pcal_{\Omega})=0$, there is nothing to prove. 
So without loss of generality, we can assume that $c_G:=\lcal(\pcal_\Omega)>0$. 
Let $f\in L^2(H)^\circ$ with $\|f\|_2^2=1$, and define $\wt{f}:G\rightarrow \bbc, \wt{f}(x):=f(s(xH)^{-1}x)$. Notice that 
\be\label{eq:l-2-norm-new-function}
\|\wt{f}\|_2^2=\frac{1}{[G:H]}\sum_{xH\in G/H} \int_H |\wt{f}(s(xH)y)|^2 dm_H(y)=\frac{1}{[G:H]}\sum_{xH\in G/H}\int_H |f(y)|^2 dy=1
\ee
and 
\be\label{eq:orthogonal-to-constant-functions}
\int_G \wt{f}(x) dx= \frac{1}{[G:H]}\sum_{xH\in G/H} \int_H\wt{f}(s(xH)y) dm_H(y)=\frac{1}{[G:H]}\sum_{xH\in G/H} \int_H f(y) dm_H(y)=0,
\ee
where $m_H$ is the probability Haar measure on $H$. 
In view of \eqref{eq:l-2-norm-new-function} and \eqref{eq:orthogonal-to-constant-functions}, Lemma~\ref{lem:lyapunov-displacement} implies that there exists $w_0\in \Omega$ such that 
\be\label{eq:not-almost-invariant}
\|w_0\cdot \wt{f}-\wt{f}\|_2=\delta_{\Omega}(\wt{f})\geq \delta(\Omega)\gg \lcal^\bullet(\pcal_\Omega).
\ee

Next, we want to give an upper bound for $\delta_{\Omega}(\wt{f})$  in terms of $\delta_{\Omega_H}(f)$. 
Note that for every $w\in \Omega$, 
\be\label{eq:in-gen-set-subgroup}
w'(\overline{x},w):=s(\overline{x})^{-1} w \h s(w^{-1}\overline{x})\in \Omega_H,
\ee
where as before $\overline{x}:=xH\in G/H$.

For every $w\in \Omega$ and $y\in \overline{x}=xH$, we have
\be\label{eq:shift-new-function}
\begin{aligned}
(w\cdot \wt{f})(y)=\wt{f}(w^{-1}y) &= f(s(w^{-1}\overline{x})^{-1}w^{-1} y)\\
&=f(w'(\overline{x},w)^{-1}s(\overline{x})^{-1}y)=(w'(\overline{x},w)\cdot f)(s(\overline{x})^{-1}y).
\end{aligned}
\ee
Let now $w\in \Omega$, then 
\be\label{eq:new-function-almost-invariant}
\begin{aligned}
\|w\cdot \wt{f}-\wt{f}\|_2^2
=
&
\frac{1}{[G:H]}\sum_{\overline{x}\in G/H} 
\int_H |(w\cdot \wt{f})(s(\overline{x})y) -\wt{f}(s(\overline{x})y)|^2 dm_H(y)
\\
=
&
\frac{1}{[G:H]}\sum_{\overline{x}\in G/H} 
\int_H |(w'(\overline{x},w)\cdot f)(y) -f(y)|^2 dm_H(y)&&\text{\small{by~\eqref{eq:shift-new-function}}}
\\
=
&
\frac{1}{[G:H]}\sum_{\overline{x}\in G/H} 
\|w'(\overline{x},w)\cdot f-f\|_2^2 \leq \delta_{\Omega_H}(f)^2.
\end{aligned}
\ee
By \eqref{eq:not-almost-invariant} and \eqref{eq:new-function-almost-invariant}, we obtain that 
\be\label{eq:no-almost-invariant-function-subgroup}
 \delta_{\Omega_H}(f)\gg \lcal^\bullet(\pcal_{\Omega}).
\ee
By \eqref{eq:no-almost-invariant-function-subgroup} and Lemma~\ref{lem:lyapunov-displacement}, we deduce that 
 \be\label{eq:contraction-subgroup}
 \lcal^\bullet(\pcal_{\Omega})\ll |\Omega_H|^{1/2} \lcal^\bullet(\pcal_{\Omega_H})\leq |\Omega|\lcal^\bullet(\pcal_{\Omega_H}).
 \ee
 This completes the proof of Claim 2.
\end{proof}
Next we consider the case where $H$ is an arbitrary open subgroup. Let $N$ be the normal core of $H$ in $G$, i.e., $N$ is the largest normal subgroup of $G$ which is a subgroup of $H$. 
Since $H$ is an open subgroup of $G$ and $G$ is a compact group, $N$ is also an open subgroup of $G$. 
Let $s:G/N\rightarrow \Gamma$ be a section such that $s(N)=1$, and $\Omega$ be a finite symmetric generating set of $\Gamma$ which contains the image of $s$ as a subset. Let $\Omega_N$ be as in Lemma~\ref{lem:generating-set-subgroup-finite-index} with respect $\Omega$ and the section $s$. Let 
\[
\Omega'_H:=\Omega_N\cup\{s(xN)^{\pm 1}\mid x\in H\}.
\]
Notice that the restriction of $s$ to $H/N$ gives us a section whose image is a subset of $\Omega'_H$. Let $\Omega'_N$ be a generating set of $\Gamma\cap N$ which is given by Lemma~\ref{lem:generating-set-subgroup-finite-index} with respect to $\Omega'_H$ and the section $s:H/N\rightarrow \Omega'_H$. Then by Claim 1 and Claim 2, we have 
\be\label{eq:from-the-normal-case-1}
\frac{1}{|\Omega|^2}\lcal^\bullet(\pcal_{\Omega})^2 \ll \lcal^\bullet(\pcal_{\Omega_N})\ll |\Omega|^{1/2} [G:N]^{1/2} \lcal^\bullet(\pcal_{\Omega})^{1/2}
\ee
and
\be\label{eq:from-the-normal-case-2}
\frac{1}{|\Omega'_H|^2}\lcal^\bullet(\pcal_{\Omega'_H})^2 \ll \lcal^\bullet(\pcal_{\Omega'_N})\ll |\Omega'_H|^{1/2} [H:N]^{1/2} \lcal^\bullet(\pcal_{\Omega'_H})^{1/2}.
\ee
Notice that since $s(N)=1$, $1$ is in $\Omega$. Therefore for every $x\in \Omega_N$, $s(1xN)^{-1}1x\in \Omega'_N$ as $1,x\in \Omega'_H$. Notice that $s(xN)=1$, and so $x\in \Omega'_N$ for every $x\in \Omega_N$. This means $\Omega_N\subseteq \Omega'_N$. 

Notice that by Lemma~\ref{lem:generating-set-subgroup-finite-index}, $\ell_{\Omega}(w)\leq 3$ for every $w\in \Omega'_H$. There by another application of Lemma~\ref{lem:generating-set-subgroup-finite-index}, we deduce that for every $x\in \Omega'_N$, $\ell_{\Omega_N}(x)\leq 7$. This is the case as $x^{\pm 1}$ is equal to $s(w_1w_2N)^{-1}w_1w_2$ for some $w_1,w_2\in \Omega'_H$, and so 
\begin{align*}
\ell_{\Omega_N}(x)=\ell_{\Omega_N}(s(w_1w_2N)^{-1}w_1w_2)&\leq \ell_{\Omega}(s(w_1w_2N)^{-1}w_1w_2)\\
& \leq \ell_{\Omega}(s(w_1w_2N)^{-1})+\ell_{\Omega}(w_1)+\ell_{\Omega}(w_2)\leq 7.
\end{align*}
This means that $\Omega'_{N}\subseteq \prod_7\Omega_N$. Hence by Lemma~\ref{lem:lyapunov-two-generators}, we have
\be\label{eq:between-two-gen-normal-core}
\frac{1}{|\Omega'_N|}\lcal^\bullet(\pcal_{\Omega_N})^2 \ll 
\lcal^\bullet(\pcal_{\Omega'_N})
\ll
|\Omega_N|^{1/2}\lcal^\bullet(\pcal_{\Omega_N})^{1/2}
\ee
By \eqref{eq:from-the-normal-case-1}, \eqref{eq:from-the-normal-case-2}, \eqref{eq:between-two-gen-normal-core}, and the facts that $|\Omega'_H|\ll |\Omega|^2$, $|\Omega'_N|\ll |\Omega|^4$, $|\Omega_N|\leq |\Omega|^2$, we conclude 
\be\label{eq:passing-to-an-open-subgroup-almost-done}
|\Omega|^{-18} [G:N]^{-1} \lcal^\bullet(\pcal_{\Omega})^8 \ll 
\lcal^\bullet(\pcal_{\Omega'_H})
\ll 
|\Omega|^{21/8} [G:N]^{1/8}\lcal^\bullet(\pcal_{\Omega})^{1/8}.
\ee
Notice that $[G:N]\leq [G:H]!$, hence, the claim follows from \eqref{eq:passing-to-an-open-subgroup-almost-done}.
\end{proof}

\section{Quantitative inverse function theorem}\label{sec: Inv Func}
In this section, we recall and state a quantitative version of the inverse function theorem. We start with setting up a few notation. Here $F$ is either $\bbr$ or $\bbq_p$. For $\vbf\in \bbq_p^d$, $\|\vbf\|$ denotes the max norm, and for $\vbf\in \bbr^d$, $\|\vbf\|$ denotes the Euclidean norm. For a positive real number $r$ and $\vbf\in F^d$, $\vbf_r$ denotes the closed ball of radius $r$ centered at $\vbf$. For $A\in \M_{m,n}(F)$, we let
\[
\sigma(A):=\sup\{r\in [0,\infty)|\h 0_r\subseteq A 0_1\}.
\]
 For $r_0\in \bbr^+$, $\xbf_0\in F^n$, and an analytic function $\Phi:(\xbf_0)_{r_0}\rightarrow F^m$, we view $\dif\Phi(\xbf)$ as an $m$-by-$n$ matrix with entries in $F$; the $ij$-entry of $\dif\Phi(\xbf)$ is $\partial_j \Phi_i(\xbf)$ where $\Phi=(\Phi_1,\ldots,\Phi_m)$. 

Next we state a $p$-adic analytic inverse function theorem which is essentially given in \cite[Lemma 54]{SG-super-approx-II}.
\begin{theorem}[Quantitative inverse function theorem: the $p$-adic case]\label{thm:p-adic-inverse-function}
Suppose $r_0\leq 1$, $\xbf_0\in \bbz_p^n$ and $\Phi:(\xbf_0)_{r_0}\rightarrow \bbz_p^m$ is an analytic function with the following properties.
\begin{enumerate}
\item There are $c_{\ibf,j}\in \bbz_p$ such that \[\Phi(\xbf)=\sum_\ibf (c_{\ibf,1} (\xbf-\xbf_0)^\ibf,\ldots,c_{\ibf,m} (\xbf-\xbf_0)^\ibf),\] where $(\xbf-\xbf_0)^{\ibf}=\prod_{j=1}^n (x_j-x_{0j})^{i_j}$ for a multi-index $\ibf=(i_1,\ldots,i_n)$ and $\xbf=(x_1,\ldots,x_n)$ and $\xbf_0=(x_{01},\ldots,x_{0n})$.
\item $\sigma(\dif \Phi(\xbf_0))\geq p^{-k_0}$ for some positive integer $k_0$.
\end{enumerate}
Then for every integer $l\geq k_0+1$ we have 
\[
\Phi(\xbf_0)_{p^{-k_0-l}}\subseteq \Phi((\xbf_0)_{p^{-l}}).
\]
\end{theorem}
\begin{proof}
See proofs of~\cite[Lemma 54, Lemma 54']{SG-super-approx-II}.
\end{proof}
Now the real case of Theorem~\ref{thm:p-adic-inverse-function} will be discussed. 
\begin{theorem}[Quantitative inverse function theorem: the real case]\label{thm:real-inverse-function}
Suppose $r_0\leq 1$, $\xbf_0\in \bbr^n$ and $\Phi:(\xbf_0)_{r_0}\rightarrow \bbr^m$ is a $C^2$-function with the following properties.
\begin{enumerate}
\item For every $1\leq j,j'\leq n$ and $\xbf\in (\xbf_0)_{r_0}$, $\|\partial_{j,j'}\Phi(\xbf)\|\leq \alpha$.
\item For some positive number $\sigma_0$, $\sigma(\dif \Phi(\xbf_0))\geq \sigma_0$.
\end{enumerate}
Then for every $0<r<\max(\frac{\sigma_0}{2mn\sqrt{\alpha}} ,r_0)$ we have
\[
\Phi(\xbf_0)_{(\frac{\sigma_0}{4})r}\subseteq \Phi((\xbf_0)_r).
\]
\end{theorem}
We start with a linear algebra lemma.
\begin{lem}\label{lem:sigma-function}
Suppose $\vbf_1,\ldots,\vbf_n,\vbf_1',\ldots,\vbf_n'\in \bbr^m$, and let $A:=[\vbf_1 \cdots \vbf_n]$ and $A':=[\vbf_1' \cdots \vbf_n']$. 
Suppose $\sigma_0$ is a positive number and $0<\vare<\sigma_0/\sqrt{n}$. 
Suppose $A=K_1[D\h 0_{n-m}] K_2$ is a singular value decomposition of $A$; that means $K_1\in O_m(\bbr)$, $K_2\in O_n(\bbr)$, and $D=\diag(s_1,\ldots,s_m)$ for some 
$s_1\geq \cdots\geq s_m\geq 0$. 
If $\sigma(A)\geq \sigma_0$ and $\|\vbf_i-\vbf_i'\|<\vare$ for every $i$, then 
\[
\sigma(A')\geq \sigma(A'L)\geq \sigma_0-\sqrt{n}\vare,
\]
where $L:=K_2^{-1} \left[\begin{array}{c} I_{m}\\0_{n-m}\end{array}\right]$.  
\end{lem}
\begin{proof}
 Since $K_1$ and $K_2$ are orthogonal, 
\be\label{eq:SVD}
\sigma(A'L)=\sigma(K_1^{-1}A'L)\quad \text{ and } \quad 
\sigma(A)=\sigma(K_1^{-1}AK_2^{-1})=\sigma(D)=s_m.
\ee
By the assumption,  all the columns of $Y:=\vare^{-1} (A-A')$ have length at most $1$. As $K_1$ and $K_2$ are orthogonal, all the columns of $X:=K_1^{-1} Y L$ have length at most $\sqrt{n}$. Notice that by the definition of $X$, we have $K_1^{-1}A'L=D-\vare X$. By \eqref{eq:SVD}, we deduce that $\sigma(D-\vare X)= \sigma(A'L )$, where $D:=\diag(s_1,\ldots,s_m)$. Notice that for every $\wbf\in \bbr^m$ we have
\begin{align}
\notag
	\|(D-\vare X)^{-1}\wbf\|= & \|D^{-1} (I-\vare XD^{-1})^{-1}\wbf\| \leq  s_m^{-1} \|(I-\vare XD^{-1})^{-1}\wbf\| \\
\notag
	\leq & \sigma_0^{-1} \Big(\sum_{i=0}^{\infty} (\vare\sqrt{n} \sigma_0^{-1})^i \Big) \|\wbf\|
	=  \sigma_0^{-1}\frac{1}{1-\vare\sqrt{n} \sigma_0^{-1}} \|\wbf\|
	\\
\label{eq:last-step-sigma-function}
= & \frac{1}{\sigma_0-\sqrt{n}\vare}.
\end{align}
By \eqref{eq:last-step-sigma-function}, we deduce that $\sigma(D-\vare X)\geq \sigma_0-\sqrt{n}\vare$. and the claim follows.
\end{proof}
\begin{lem}\label{lem:inverse-function-sigma-differential-under-perturbation}
Under the assumptions of Theorem~\ref{thm:real-inverse-function}, let $\dif\Phi(\xbf_0)=K_1 [D \h 0_{m,n-m}] K_2$ be a singular value decomposition; that means $K_1\in O_m(\bbr)$, $K_2\in O_n(\bbr)$, and $D=\diag(s_1,\ldots,s_m)$ for some $s_1\geq \cdots\geq s_m\geq 0$. Let $L:=K_2^{-1} \left[\begin{array}{c} I_{m}\\0_{n-m}\end{array}\right]$.
If $\|\xbf-\xbf_0\|\leq \min(r_0, \frac{\sigma_0}{2n\sqrt{m\alpha}})$, then the following holds
\[
\sigma(\dif \Phi(\xbf)L )\geq \sigma_0/2.
\]
\end{lem}
\begin{proof}
By the mean value theorem, for every $i$, $j$, and $k$, there is a point $\xbf_{ijk}$ on the segment connecting $\xbf_0$ to $\xbf$ such that 
\be\label{eq:mean-value-controling-sigma-of-dif}
\partial_j\Phi_i(\xbf)-\partial_j\Phi_i(\xbf_0)=\sum_{k=1}^n \partial_{k,j} \Phi_i(\xbf_{ijk}) (x_k-x_{0k}).
\ee
By \eqref{eq:mean-value-controling-sigma-of-dif}, we deduce that for every $j$ the following holds
\be\label{eq:sigma-dif-column-perturbation}
\|\partial_j\Phi(\xbf)-\partial_j\Phi(\xbf_0)\|\leq \sqrt{mn\alpha} \|\xbf-\xbf_0\|\leq \sqrt{mn\alpha}\frac{\sigma_0}{2n\sqrt{m\alpha}}=\frac{\sigma_0}{2\sqrt{n}}.
\ee
By Lemma~\ref{lem:sigma-function} and \eqref{eq:sigma-dif-column-perturbation}, the claim follows.
\end{proof}
\begin{proof}[Proof of Theorem~\ref{thm:real-inverse-function}]
By the singular value decomposition of $\dif\Phi(\xbf_0)$, there are $K_1\in O_m(\bbr)$ and $K_2\in O_n(\bbr)$ and $s_1\geq\ldots\geq s_m>0$ such that $\dif \Phi(\xbf_0)=K_1[D \h 0_{m,n-m}]K_2$ where $D:=\diag(s_1,\ldots,s_m)$ is the diagonal matrix with diagonal entries $s_1,\ldots,s_m$. Then $s_m=\sigma(\dif\Phi(\xbf_0))$. Let $L:\bbr^m\rightarrow \bbr^n, L(\vbf):=K_2^{-1} \begin{pmatrix}
\vbf \\ 0_{n-m}
\end{pmatrix}$. Notice that $\|L(\vbf)\|=\|\vbf\|$ for every $\vbf\in \bbr^m$ and $L$ can be represented by the matrix 
$K_2^{-1} \left[\begin{array}{c}I_m \\0_{m,n-m} \end{array}\right]$; this matrix is denoted by $L$. Furthermore for every $\vbf$ we have
\be\label{eq:lower-bound-directional-derivative}
\|\dif\Phi(\xbf_0)L(\vbf)\|\geq \sigma(\dif \Phi(\xbf_0)) \|\vbf\|\geq \sigma_0 \|\vbf\|.
\ee
Let 
\be\label{eq:coordinate-changing-inverse-function}
\Psi:0_{r_0}\rightarrow \bbr^m, \Psi(\vbf):=\Phi(L(\vbf)+\xbf_0)-\Phi(\xbf_0)-\dif\Phi(\xbf_0)L(\vbf).
\ee

Suppose $\Psi=(\Psi_1,\ldots,\Psi_m)$. By the mean value theorem, for every $i$, there is $\vbf_i$ on the segment connecting $\vbf$ to $0$ such that
\be\label{eq:mean-value-after-coordinate-change}
\Psi_i(\vbf)-\Psi_i(0)=\nabla\Psi_i(\vbf_i)\cdot \vbf.
\ee
Notice that $\nabla\Psi_i(\wbf)$ is the $i$-th row of $\dif \Psi(\wbf)$, and by the chain rule, we have 
\be\label{eq:dif-change-coordinate}
\dif \Psi(\wbf)= \Big(\dif\Phi(L(\wbf)+\xbf_0) - \dif \Phi(\xbf_0)\Big) L.
\ee
By \eqref{eq:sigma-dif-column-perturbation}, we have 
\be\label{eq:dif-perturbation}
\|\partial_j\Phi(L(\wbf)+\xbf_0)-\partial_j \Phi(\xbf_0)\|\leq \sqrt{mn\alpha}\|L(\wbf)\|=\sqrt{mn\alpha}\|\wbf\|.
\ee
By \eqref{eq:mean-value-after-coordinate-change}, \eqref{eq:dif-change-coordinate}, and \eqref{eq:dif-perturbation}, the following holds
\be\label{eq:upper-bound-difference-coordinates-psi}
|\Psi_i(\vbf)-\Psi_i(0)|\leq (\sqrt{n} \|\vbf\|)(\sqrt{mn\alpha}\|\vbf\|)=n\sqrt{m\alpha} \|\vbf\|^2.
\ee
Inequality given in \eqref{eq:upper-bound-difference-coordinates-psi} implies 
\be\label{eq:upper-bound-difference-psi}
\|\Psi(\vbf)-\Psi(0)\|\leq mn\sqrt{\alpha}\|\vbf\|^2.
\ee
Since $\Psi(0)=0$, by \eqref{eq:coordinate-changing-inverse-function}, \eqref{eq:upper-bound-difference-psi}, and \eqref{eq:lower-bound-directional-derivative}, we obtain
\begin{align}
\notag 
\|\Phi(L(\vbf)+\xbf_0)-\Phi(\xbf_0)|\geq &
 \|\dif\Phi(\xbf_0)L(\vbf)\|-mn\sqrt{\alpha}\|\vbf\|^2
 \\
\notag 
\geq & \sigma_0\|\vbf\|-mn\sqrt{\alpha}\|\vbf\|^2
\\
\label{eq:lower-bound-on-difference-psi-1}
\geq & \Big(\sigma_0-mn\sqrt{\alpha}\|\vbf\| \Big)\|\vbf\|
\end{align}
By \eqref{eq:lower-bound-on-difference-psi-1}, if $\|\vbf\|\leq \frac{\sigma_0}{2mn\sqrt{\alpha}}$, we obtain
\be\label{eq:lower-bound-difference-psi}
\|\Phi(L(\vbf)+\xbf_0)-\Phi(\xbf_0)\|\geq \frac{\sigma_0}{2}\|\vbf\|.
\ee
Let $0<r<\frac{\sigma_0}{2mn\sqrt{\alpha}}$, and for $\ybf\in \Phi(\xbf_0)_{\sigma_0 r/4}$, consider the function
\[
f_\ybf:\overline{0}_r\rightarrow \bbr, f(\vbf):=\|\Phi(L(\vbf)+\xbf_0)-\ybf\|^2,
\]
where $\overline{0}_r$ is the closed ball of radius $r$ centered at $0$. By \eqref{eq:lower-bound-difference-psi}, if $\|\vbf\|=r$, then
\begin{align}
\notag
\sqrt{f_\ybf(\vbf)}\geq &
\|\Phi(L(\vbf)+\xbf_0)-\Phi(\xbf_0)\|-\|\Phi(\xbf_0)-\ybf\|
\\
\notag
 \geq & \frac{\sigma_0}{2}\|\vbf\|-
\|\Phi(\xbf_0)-\ybf\|
\\ 
\label{eq:lower-bound-values-boundary}
> & \frac{\sigma_0r}{2}-\frac{\sigma_0r}{4}=\frac{\sigma_0r}{4}> \sqrt{f_\ybf(0)}.
\end{align}
By \eqref{eq:lower-bound-values-boundary}, the minimum of $f_\ybf$ occurs at at a critical point $\vbf_\ybf\in 0_r$ of $f_\ybf$. Knowing that $\nabla f_\ybf(\vbf_\ybf)=0$, using the chain rule, we obtain that the following holds
\[
(\Phi(L(\vbf_\ybf)+\xbf_0)-\ybf)^T\dif \Phi(L(\vbf_\ybf)+\xbf_0)L=0,
\]
where $(\Phi(L(\vbf_\ybf)+\xbf_0)-\ybf)^T$ is the row matrix form of the vector $\Phi(L(\vbf_\ybf)+\xbf_0)-\ybf$. By Lemma~\ref{lem:inverse-function-sigma-differential-under-perturbation}, $\sigma(\dif \Phi(L(\vbf_\ybf)+\xbf_0)L)\geq \sigma_0/2>0$. Therefore $\dif \Phi(L(\vbf_\ybf)+\xbf_0)L$ is injective, which implies that $\ybf=\Phi(L(\vbf_\ybf)+\xbf_0)$. The claim follows.
\end{proof}


 \section{The case of abelian groups}\label{app:abelian}

In this appendix we will prove the following. 
\begin{theorem}\label{thm:abelians}
The groups $\ZZ_p$ and $\RR/\ZZ$ are not spectrally independent. 
\end{theorem}

We start by a general criterion characterizing when a coupling of the Haar measures on two compact abelian groups has spectral gap. 

\begin{lem}\label{lem: gamma i to 1}
Let $G$ be a compact abelian group, and let $\mu$ be a symmetric probability measure on $G$. Then $\mu$ does not have spectral gap iff there exists $\gamma_j \in \widehat{G} \setminus \{1\}$ such that 
$\gamma_j$ converges to the constant function $1$ on $G$, $\mu$-a.e. 
\end{lem}

\begin{proof}
Since $G$ is an abelian group, the spectrum of the convolution operator $T_\mu$ consists of $\widehat{\mu}(\gamma)$ for 
$\gamma \in \widehat{G}$. Since $\mu$ is a symmetric probability measure $\wh{\mu}(\gamma)\in [-1,1]$ for all $\gamma\in \wh{G}$.
Assuming $\mu$ does not have spectral gap, we obtain a sequence $\{\gamma_j\}_{j=1}^\infty\subseteq \wh{G}$ such that $ |\widehat{\mu}(\gamma_j) | \to 1$ and $\gamma_j$'s are pairwise distinct. 
We will show that passing to a subsequence, if necessary, we have 
\be\label{eq: gammaj to 0}
\gamma_j^2\to 1 \qquad\text{$\mu$-a.e.}
\ee
To see this, write $\gamma_j=x_j+iy_j$. 
Indeed, after passing to a subsequence, which we continue to denote by $\{\gamma_j\}$, we have 
\[
\int\gamma_jd\mu=\int x_jd\mu\to \epsilon
\]
where $\epsilon=\pm1$. Since $|\gamma_j|=1$, we conclude that
$\int|\gamma_j-\epsilon|\h d\mu\to 0$. Thus passing to a further subsequence, if necessary, we get that $\gamma_j\to\epsilon$, $\mu$-a.e.; this implies~\eqref{eq: gammaj to 0}.  

For the converse, note that if $\gamma_j \to 1$, $\mu$-a.e., then by Lebesgue's dominant convergence theorem, we have $ \widehat{\mu}(\gamma_i) \to 1$; hence $\mu$ does not have spectral gap. 
\end{proof}

\begin{lem}\label{lem:push}
Let $(Z,\nu)$, $(X_1,\mu_1)$, and $(X_2,\mu_2)$ be probability spaces. Assume that for $i=1,2$, $f_i:(Z,\nu)\to (X_i,\mu_i)$ is a measurable map so that $f_{i*}\nu=\mu_i$. Let $\delta: Z \to X_1 \times X_2$ be defined by 
$\delta(z)=(f_1(z), f_2(z))$. Then 
\begin{equation}\label{eq:push}
\mu= \delta_{\ast}(\nu)
\end{equation}
is a coupling of $\mu_1$ and $\mu_2$. 
\end{lem}

\begin{proof} 
For $i=1,2$, denote the canonical projections from $X_1 \times X_2$ onto $X_i$ by $\pi_i$. Then since $\pi_i \circ \delta = f_i$ and $f_{i*}\nu=\mu_i$, we have 
\[
(\pi_i)_\ast \mu = (\pi_i)_\ast \circ \delta_\ast (\nu)= \mu_i\qquad\text{for $i=1,2$,}
\]
as we claimed.
\end{proof}

For the rest of this section, put $Z=\{0,1,\ldots, p-1\}^{\bbn_0}$ and let $\nu$ denote the product measure of probability counting measure on $\{0,1,\ldots, p-1\}$. 
We will apply Lemma~\ref{lem:push} with $(Z,\nu)$ and $(X_1,\mu_1)=(\bbr/\bbz, m_{\bbr/\bbz})$ and $(X_2,\mu_2)=(\bbz_p,m_{\bbz_p})$.

We will also need the following to define the marginals. 

\begin{lem}
    \label{lem:digits}
    Let $f_{\infty}:Z\to \bbr/\bbz$ and $f_p:Z\to\bbz_p$ be defined by 
    \[
    f_{\infty}(\{x_i\}_{i=0}^{\infty}):=\left(\sum_{i=0}^\infty x_i p^{-i-1}\right)+\bbz,\qquad\text{and}\qquad 
    f_{p}(\{x_i\}_{i=0}^{\infty}):=\sum_{i=0}^\infty x_i p^{i}.
    \]
    Then $f_{\infty,*}\nu=m_{\bbr/\bbz}$ and $f_{p,*}\nu=m_{\bbz_p}$.
\end{lem}

\begin{proof} 
    First, we observe that $f_{\infty}$ and $f_p$ are continuous surjective functions; in particular, they are measurable functions.
    We also notice that $f_p$ is injective, therefore it is a homeomorphism. As for $f_\infty$, note that $f_\infty^{-1}(x+\bbz)$ has one element unless $x$ is of the form $\frac{j}{p^{k}}$ in which case $f_{\infty}^{-1}(x+\bbz)$ has two elements.
    Therefore, there are $\nu$-null subset  $\mathcal N_1\subseteq Z$ and $m_{\bbr/\bbz}$-null subset $\mathcal N_2\subseteq \bbr/\bbz$, so that $f_\infty:Z\setminus \mathcal N_1\to (\bbr/\bbz)\setminus \mathcal N_2$ is a bijective measurable map whose inverse is also measurable.   
    
    To prove the lemma for $f_\infty$, it suffices to give a generating set $\ucal$ of the Borel $\sigma$-algebra of $\bbr/\bbz$ such that for every $U\in \ucal$, the sets $U\setminus \mathcal N_2$ and $f_\infty^{-1}(U\setminus \mathcal N_2)$ have the same measure. The proof for $f_p$ is similar. For $i\in \bbn_0$ and $0 \le j\le p^{i+1}-1$, let
    \[
    U_{i,j}:=\left(\frac{j}{p^{i+1}},\frac{j+1}{p^{i+1}}\right)+\bbz.
    \]
    Notice that $\ucal:=\{U_{i,j}\mid i\in \bbn_0,0\leq j\leq p-1\}$ generates the Borel $\sigma$-algebra on $\bbr/\bbz$, and 
    \[
    f_{\infty}^{-1}(U_{i,j})=\left(\{a_0\}\times\cdots\times\{a_{i}\}\times \prod_{k=i+1}^\infty \{0,\ldots,p-1\}\right)\setminus \mathcal N_1
    \]
    where $j=a_0+a_1p+\cdots + a_i p^i$ and $a_0,\ldots,a_i\in [0,p-1]$. Therefore, both $U_{i,j}$ and $f_\infty^{-1}(U_{i,j})$ have measure $\frac{1}{p^{i+1}}$. Altogether, we deduce that $f_\infty$ is measure-preserving, which means $f_{\infty,\ast}\nu=m_{\bbr/\bbz}$. The $f_p$ case is similar.  
\end{proof}

We note that for any permutation $\sigma:\bbn_0\to\bbn_0$, the induced map  $h_\sigma:Z\to Z$ is a homeomorphism which preserves $\nu$. Therefore, combining Lemma~\ref{lem:push} and Lemma~\ref{lem:digits}, we have 

\begin{corollary}\label{cor:combine push and digits}
Let $\sigma:\bbn_0\to\bbn_0$ be a permutation, and let $\delta_\sigma:Z\to \bbr/\bbz\times\bbz_p$
be  
\[
\delta_\sigma(z)=\bigl(f_\infty(z), f_p(h_\sigma(z))\bigr),
\]
and let $\mu_\sigma=\delta_{\sigma,*}\nu$. Then $\mu_\sigma$ is a coupling of $m_{\bbr/\bbz}$ and $m_{\bbz_p}$.  
\qed
\end{corollary}

\begin{proof}[Proof of Theorem \ref{thm:abelians}]
Let $ \sigma: \bbn_0 \to \bbn_0$ be the permutation fixing $0$ and $1$ and inverting each block $\{ 2^j, \dots, 2^{j+1}-1 \}$. More precisely,  all $j \ge 1 $ and $ 0 \le i \le 2^{j}-1$ we have  $ \sigma( i+2^j)= 2^{j+1}-i-1$. Let $\mu=\mu_\sigma$ be as in Corollary~\ref{cor:combine push and digits} applied with this $\sigma$. Then $ \mu$ is a coupling of the probability Haar measures on $\RR/\ZZ$ and $\ZZ_p$ supported on 
\be\label{eq:support coupling}
\bigl\{(\textstyle\sum_{i=0}^\infty x_i p^{-(i+1)},\textstyle\sum_{ i =0}^{\infty} y_i p^i)\in\bbr/\bbz\times\bbz_p | x_{i+2^j}= y_{ 2^{j+1}-i-1}, \text{for all $j \ge 1 $, $0 \le i \le 2^{j}-1$}\bigr\}.
\ee

We will use Lemma~\ref{lem: gamma i to 1} to show that $\mu$ does not have spectral gap, which will finish the proof of the Theorem as $\mu$ is a coupling of the probability Haar measures on $\RR/\ZZ$ and $\ZZ_p$.

To apply Lemma~\ref{lem: gamma i to 1}, we construct a family of characters for $ \RR/\ZZ$ and 
$\ZZ_p$. Let 
\[\ef:\bbr\to \bbc^\times, \ef(x):=e^{2\pi i x},\]
and  recall that 
\[
\wh{\bbr/\bbz}=\{e_{\infty,n}\mid n\in \bbz\},
\]
where $e_{\infty,n}(x+\bbz):=\ef(nx)$ --- notice that $e_{\infty,n}$ is well-defined as $\ef$ is $\bbz$-invariant. 

To describe the dual of $\bbz_p$, we recall that every element $x \in \bbq_p$ can be written as a sum of a rational number $r_x$ and a $p$-adic integer $z_x$. Let 
\[
e_p:\bbq_p\to \bbc^\times,\quad e_p(x):=\ef(r_x),
\]
and notice that $e_p$ does not depend on the choice of decomposition $x=r_x+z_x$ as $\ef$ is $\bbz$-invariant. Then
\[
\wh{\bbz_p}=\bigl\{e_{p,r}\mid r=\tfrac{k}{p^j}, k, j \in \bbz \bigr\},
\]
where $e_{p,r}:\bbz_p\to \bbc^\times, e_{p,r}(x)=e_p(rx)$.

For $j \ge 1$, define $\alpha_j:= e_{\infty,p^{2^j}}$, $\beta_j:=e_{p,p^{-2^{j+1}}}$, and 
\[
\gamma_j:\bbr/\bbz\times \bbz_p\to \bbc^\times,\qquad 
\gamma_j(x,y):=\alpha_j(x)/\beta_j(y).
\]
We claim that for $\mu$-a.e.\ $(x,y)$ we have $ \gamma_j(x,y) \to 1$. By \eqref{eq:support coupling}, we have that if $(x,y)$ is in the support of $\mu$, then
there are digits $x_i,y_i\in \{0,1,\ldots, p-1\}$ such that $x=\sum_{i=0}^\infty x_i p^{-(i+1)}$ and $y=\sum_{i=0}^\infty y_i p^{i+1}$, where 
\be\label{eq:x y comp support coupling}
y_{2^{j+1}-i-1}=x_{2^j+i} \qquad\text{for all $j\geq 1$ and $0\leq i<2^j$.}
\ee

Then for all $j$
\be\label{eq: estimate character x comp}  
\begin{aligned}
\alpha_j(x) = e_{\infty,p^{2^j}}(x) &= \ef  \Bigl( \sum_{i=0}^{\infty} x_i p^{ 2^j-i-1}   \Bigr)\\
&= \ef  \Bigl( \sum_{i=0}^{\infty} x_{i+2^j} p^{-(i+1)}   \Bigr)  =
\ef  \Bigl( \sum_{i=0}^{2^j-1} x_{i+2^j} p^{-(i+1)} \Bigr) + O(p^{-2^j}).
\end{aligned}
\ee

Similarly, we have
\be\label{eq: estimate character y comp}
\begin{aligned}
    \beta_j(y)&= e_{p,p^{-2^{j+1}}}(y) 
    =
    e_p\Bigl(\sum_{i=0}^\infty y_i p^{i+1-2^{j+1}}\Bigr)
    \\
    &=
    \ef  \Bigl(\sum_{i=0}^{2^{j+1}-1} y_i p^{i+1-2^{j+1}}\Bigr)
    =
    \ef  \Bigl( \sum_{i=0}^{2^{j}-1} y_{ 2^{j+1}-i-1}
p^{-(i+1)} \Bigr) + O( p^{-2^j}).
\end{aligned}
\ee
Hence, by \eqref{eq:x y comp support coupling}, \eqref{eq: estimate character x comp}, and \eqref{eq: estimate character y comp}, we obtain that 
\[
 |\gamma_j(x,y)-1|= |\alpha_j(x) - \beta_j(y)| = O( p^{-2^j}).
\]
Therefore by Lemma~\ref{lem: gamma i to 1}, we deduce that $\mu$, which is a coupling of the probability Haar measures of $\bbr/\bbz$ and $\bbz_p$, does not have spectral gap property. Thus, $\bbr/\bbz$ and $\bbz_p$ are not spectrally independent.
\end{proof}


\section{Commutator of small neighborhoods in compact semisimple Lie groups.} In~\cite{G-commutator} it is proved that if $G$ is a perfect connected compact group, then every element of $G$ is a commutator element. More recently in~\cite{dA-M-commutator}, it is proved that the commutator map 
\[
\psi:G\times G\rightarrow G, \psi(g_1,g_2):=[g_1,g_2]=g_1g_2g_1^{-1}g_2^{-1}
\]
 is an open function if $G$ is a compact semisimple Lie group. Here we show the following quantitative version of their result. For the case of unitary groups, this is part of the Solovay-Kitaev  algorithm.
\begin{proposition}\label{prop:commutator-small-nhbds}
Suppose $G$ is a compact semisimple Lie group. Suppose $G\subseteq {\rm SO}(n)$, and it is equipped with the metric induced by the operator norm.  Let $\psi:G\times G\rightarrow G, \psi(g_1,g_2):=[g_1,g_2]$ be the commutator map. Then there is a positive number $c'':=c''(G)$ such that for every $0<\rho_1,\rho_2<1$ we have $\psi(1_{\rho_1}\times 1_{\rho_2})\supseteq 1_{c''\rho_1\rho_2}$.
\end{proposition}
\begin{proof}
Let $\gfr$ be the Lie algebra of $G$, and $\overline{\psi}:\gfr\times \gfr\rightarrow \gfr, \overline{\psi}(x,y):=[x,y]$. By \cite[Theorem 2.1]{dA-M-commutator}, there is a positive number $c''_1$ such that $\overline{\psi}(0_1\times 0_1)\supseteq 0_{c''_1}$.  Since $\overline{\psi}$ is bilinear, for every positive numbers $\rho$ and $\rho'$ we have 
\be\label{eq:commutator-Lie-algebra}
\overline{\psi}(0_{\rho}\times 0_{\rho'})\supseteq 0_{c''_1\rho\rho'}.
\ee
Let $\phi(t):=\frac{e^t-1}{t}$. Notice that $\phi$ is an analytic function, $\phi(0)=1$, and $|\phi(t)-1|\leq \frac{te^t}{2}$ for every $0<t<1$.  We set
\[
\xi:\gfr\times \gfr\rightarrow \gfr, \quad \xi(x,y):=\exp(\ad(y))(x)-x,
\]
and we notice that 
\begin{align}
\notag
\xi(x,y)=&(\exp(\ad(y))-{\rm id})(x)= \ad(y)(\phi(\ad(y))(x))\\
\label{eq:step-one-going-to-group-commutator}  =& [y,\phi(\ad(y))(x)]= \psi(y,\phi(\ad(y))(x)).
\end{align}
For every  $0<\rho\ll_{\dim G} 1$, $0<\rho'<1$, and $y\in 0_{\rho}$,  $\phi(\ad(y))(0_{\rho'})\supseteq 0_{\rho'/2}$. Therefore by \eqref{eq:commutator-Lie-algebra} and \eqref{eq:step-one-going-to-group-commutator}, we obtain
\be\label{eq:second-step-going-to-group-comm}
\xi(0_{\rho}\times 0_{\rho'})\supseteq 0_{\frac{1}{2}c''_1\rho\rho'}
\ee
for every $0<\rho<1$ and $0<\rho'\ll_{\dim G} 1$. Using the Baker-Campbell-Hausdorff formula, it is deduced in \cite[Proposition 3.1]{dA-M-commutator} that 
there are analytic functions $P$ and $Q$ from a neighborhood $\ocal$ of $(0,0)\in \gfr\times \gfr$ to $G$ such that 
\be\label{eq:BCH}
P(0,0)=Q(0,0)=1, \h\h \exp(x+y)=\exp(\Ad(P(x,y))(x))\exp(\Ad(Q(x,y))(y)),
\ee 
for every $(x,y)\in \ocal$. By \eqref{eq:BCH}, we have
\begin{align}
\label{eq:exp-commutator}
	\exp(-\xi(x,y))= &[A,B], \quad\text{where}\\
	\notag
	A:=& P(x,-\exp(\ad(y))(x)) \exp(x) P(x,-\exp(\ad(y))(x))^{-1}  \text{ and }\\
	\notag
	B:=& Q(x,-\exp(\ad(y))(x)) \exp(y) P(x,-\exp(\ad(y))(x))^{-1}.
\end{align}
For $x\in 0_\rho$, we have $\|-\exp(\ad(y))(x)\|=\|\Ad(\exp(y))(x)\|\leq \rho$. Since $P$ and $Q$ are analytic and $P(0,0)=Q(0,0)=1$, for $0<\rho\ll_G 1$ and $x\in 0_\rho$ we have 
\be\label{eq:analytic-functions-in-BCH}
P:=P(x,-\exp(\ad(y))(x))\in 1_{C''\rho}\quad\text{and} \quad 
Q:=Q(x,-\exp(\ad(y))(x))\in 1_{C''\rho},
\ee
for some $C'':=C''(G)$. By \eqref{eq:analytic-functions-in-BCH}, for $x\in 0_\rho$ and $y\in 0_{\rho'}$, we have 
\be\label{eq:elements-in-commutator}
A\in 1_{2\rho}\quad\text{and} \quad B\in 1_{4C''\rho+2\rho'}.
\ee
By  \eqref{eq:second-step-going-to-group-comm}, \eqref{eq:exp-commutator}, and  \eqref{eq:elements-in-commutator}, we deduce that the following holds:
\[
[1_{2\rho},1_{4C''\rho+2\rho'}]\supseteq 1_{\frac{1}{4}c''_1\rho\rho'}.
\]
For every $0<\rho\ll_G 1$ and $\rho\leq \rho''\ll 1$, we have 
\[
[1_{\rho},1_{\rho'}]\supseteq [1_{2\frac{\rho}{8C''}},1_{4C''\frac{\rho}{8C''}+2\frac{\rho'}{4}}]\supseteq 1_{\frac{c''_1}{128C''}\rho\rho'}.
\]
This completes the proof.
\end{proof}

\bibliographystyle{plain}
\bibliography{Ref}
\end{document}

%% file: micros.tex
\DeclareMathAlphabet{\mathbbold}{U}{bbold}{m}{n}
\def\bb1{\mathbbold{1}}
\def\bbz{\mathbb{Z}}
\def\bbq{\mathbb{Q}}

\def\bbr{\mathbb{R}}

\def\bbc{\mathbb{C}}

\def\bbe{\mathbb{E}}

\def\bbn{\mathbb{N}}
\def\bbg{\mathbb{G}}

\def\bbp{\mathbb{P}}

\def\fcal{\mathcal{F}}
\def\ucal{\mathcal{U}}

\def\ocal{\mathcal{O}}
\def\ocal{\mathcal{O}}

\def\cal{\mathcal{H}}
\def\lcal{\mathcal{L}}
\def\pcal{\mathcal{P}}

\def\gl{\mathfrak{gl}}
\def\gfr{\mathfrak{g}}

\def\abf{\mathbf{a}}
\def\ybf{\mathbf{y}}
\def\xbf{\mathbf{x}}
\def\ibf{\mathbf{i}}

\def\wbf{\mathbf{w}}
\def\vbf{\mathbf{v}}
\def\ebf{\mathbf{e}}
\def\mbf{\mathbf{m}}
\def\nbf{\mathbf{n}}
\def\acts{\curvearrowright}

\DeclareMathOperator\SL{SL}
\DeclareMathOperator\GL{GL}
\DeclareMathOperator\PGL{PGL}
\DeclareMathOperator\Ad{Ad}
\DeclareMathOperator\ad{ad}
\DeclareMathOperator\Lie{Lie}

\DeclareMathOperator\im{Im}
\DeclareMathOperator\M{M}

\DeclareMathOperator\diag{diag}

\DeclareMathOperator\End{End}
\DeclareMathOperator\pr{pr}
\DeclareMathOperator\tr{Tr}

\DeclareMathOperator\op{op}
\DeclareMathOperator\SU{SU}

\DeclareMathOperator\dif{d\!}

\def\h{\hspace{1mm}}

\def\vare{\varepsilon}
\def\be{\begin{equation}}
\def\ee{\end{equation}}

\newcommand{\wt}[1]{\widetilde{#1}}
\newcommand{\wh}[1]{\widehat{#1}}

\newcounter{consta}

\newcounter{constb}[section]

\newcounter{constc}[section]
\renewcommand{\theconstc}{{{\mathsf{D}}_{\arabic{constc}}}}
\newcounter{constA}

\newcommand{\constc}{\refstepcounter{constc}\theconstc}

\makeatletter
\newcommand*\bigcdot{\mathpalette\bigcdot@{.5}}
\newcommand*\bigcdot@[2]{\mathbin{\vcenter{\hbox{\scalebox{#2}{$\m@th#1\bullet$}}}}}
\makeatother